\RequirePackage{snapshot}

\newif\ifsubsections
\subsectionstrue 

\documentclass[]{pcmi}

\usepackage[sc]{mathpazo}          
\usepackage{eulervm}               
\usepackage[scaled=0.86]{berasans} 
\usepackage[scaled=1]{inconsolata} 
\usepackage[T1]{fontenc}

\usepackage[%
	protrusion=true,
	expansion=false,
	auto=false
	]{microtype}

\usepackage{xcolor}
\usepackage{graphicx}
\graphicspath{{./figures/}}

\ifdraft
	\definecolor{linkred}{rgb}{0.7,0.2,0.2}
	\definecolor{linkblue}{rgb}{0,0.2,0.6}
\else
	\definecolor{linkred}{rgb}{0.0,0.0,0.0}
	\definecolor{linkblue}{rgb}{0,0.0,0.0}
\fi

\usepackage{amscd}
\usepackage{amssymb}
\usepackage{latexsym}
\usepackage{lscape}
\usepackage{url}
\usepackage{xypic}
\usepackage{color}
\input xy
\xyoption{all}
\usepackage[utf8]{inputenc} 

\usepackage{tikz} 

\PassOptionsToPackage{hyphens}{url} 
\usepackage[
    setpagesize=false,
    pagebackref,
	pdfpagelabels=false,
    pdfstartview={FitH 1000},
    bookmarksnumbered=false,
    linktoc=all,
    colorlinks=true,
    anchorcolor=black,
    menucolor=black,
    runcolor=black,
    filecolor=black,
    linkcolor=linkblue,
	citecolor=linkblue,
	urlcolor=linkred,
]{hyperref}
\usepackage[backrefs,msc-links,nobysame]{amsrefs}

\customizeamsrefs 

\theoremstyle{plain}

\newtheorem{thm}[equation]{Theorem}
\newtheorem{prop}[equation]{Proposition}
\newtheorem{lem}[equation]{Lemma}

\newtheorem{lem-def}[equation]{Lemma-Definition}
\newtheorem{cor}[equation]{Corollary}
\theoremstyle{remark}
\newtheorem{ex}[equation]{Example}

\newtheorem{rmk}[equation]{Remark}
\theoremstyle{definition}
\newtheorem{dfn}[equation]{Definition}

\newcommand{\quash}[1]{}  
\newcommand{\nc}{\newcommand}
\nc{\on}{\operatorname}

\newcommand{\fraka}{{\mathfrak a}}

\newcommand{\frakg}{{\mathfrak g}}

\newcommand{\frakl}{{\mathfrak l}}

\newcommand{\frakp}{{\mathfrak p}}
\newcommand{\frakq}{{\mathfrak q}}

\newcommand{\fraks}{{\mathfrak s}}
\newcommand{\frakt}{{\mathfrak t}}

\newcommand{\bA}{{\mathbb A}}
\newcommand{\bB}{{\mathbb B}}
\newcommand{\bC}{{\mathbb C}}

\newcommand{\bE}{{\mathbb E}}
\newcommand{\bF}{{\mathbb F}}
\newcommand{\bG}{{\mathbb G}}

\newcommand{\bN}{{\mathbb N}}
\newcommand{\bO}{{\mathbb O}}
\newcommand{\bP}{{\mathbb P}}
\newcommand{\bQ}{{\mathbb Q}}
\newcommand{\bR}{{\mathbb R}}

\newcommand{\bZ}{{\mathbb Z}}

\newcommand{\mA}{{\mathcal A}}
\newcommand{\mB}{{\mathcal B}}
\newcommand{\mC}{{\mathcal C}}
\newcommand{\mD}{{\mathcal D}}
\newcommand{\mE}{{\mathcal E}}
\newcommand{\mF}{{\mathcal F}}
\newcommand{\mG}{{\mathcal G}}
\newcommand{\mH}{{\mathcal H}}

\newcommand{\mL}{{\mathcal L}}
\newcommand{\mM}{{\mathcal M}}
\newcommand{\mN}{{\mathcal N}}
\newcommand{\mO}{{\mathcal O}}
\newcommand{\mP}{{\mathcal P}}

\newcommand{\mR}{{\mathcal R}}
\newcommand{\mS}{{\mathcal S}}

\newcommand{\mY}{{\mathcal Y}}

\nc{\al}{{\alpha}} 
\nc{\be}{{\beta}} 
\nc{\ga}{{\gamma}}
\nc{\la}{{\lambda}}
\nc{\ve}{{\varepsilon}} 
\nc{\Ga}{{\Gamma}} 
\nc{\La}{{\Lambda}}

\nc{\ad}{{\on{ad}}}
\newcommand{\Ad}{{\on{Ad}}}
\nc{\aff}{{\on{aff}}}
\nc{\Aff}{{\on{Aff}}}
\newcommand{\Aut}{{\on{Aut}}}
\nc{\Bun}{{\on{Bun}}}
\newcommand{\cha}{{\on{char}}}
\nc{\der}{{\on{der}}}

\nc{\diag}{{\on{diag}}}

\nc{\Fl}{{\mF\ell}}
\newcommand{\Gal}{{\on{Gal}}}
\newcommand{\Gr}{{\on{Gr}}}
\newcommand{\Hom}{{\on{Hom}}}
\nc{\IC}{{\on{IC}}}
\newcommand{\id}{{\on{id}}}
\nc{\Id}{{\on{Id}}}

\nc{\Ind}{{\on{Ind}}}
\nc{\inv}{{\on{Inv}}}
\nc{\Iso}{{\on{Isom}}}
\newcommand{\Lie}{{\on{Lie}}}
\newcommand{\Pic}{{\on{Pic}}}
\newcommand{\pr}{{\on{pr}}}
\nc{\Ql}{{\overline{\bQ}_\ell}}
\nc{\res}{{\on{res}}}
\newcommand{\Res}{{\on{Res}}}
\newcommand{\s}{{\on{sc}}}

\newcommand{\Spec}{\on{Spec}}

\nc{\tr}{{\on{tr}}}

\nc{\mmu}{\mu_\bullet}

\nc{\boxstar}{\makebox[0pt][l]{$\square$}\raisebox{.115ex}{\hspace{0.09em}\Large$\star$}\ } 
\nc{\fstar}{\makebox[0pt][l]{$\square$}\raisebox{.115ex}{\hspace{0.09em}\Large$\star$}^f } 
\newcommand*\numcircledtikz[1]{\tikz[baseline=(char.base)]{
            \node[shape=circle,draw,inner sep=-0.06em, very thin] (char) {\Large \raisebox{.02ex}{#1}};}} 
            \nc{\ostar}{\text{\kern0.12em\numcircledtikz{$\star$}\kern0.18em}} 

\newcommand{\GL}{{\on{GL}}}
\nc{\GSp}{{\on{GSp}}} \nc{\GU}{{\on{GU}}} \nc{\SL}{{\on{SL}}}
\nc{\SU}{{\on{SU}}} \nc{\SO}{{\on{SO}}}
\newcommand{\gl}{{\frakg\frakl}}
\nc{\Hk}{{\on{Hk}}}
\nc{\lhk}{\Hk^{\on{loc}}}

\nc{\pf}{{p^{-\infty}}}
\nc{\bb}{{\mathbf{b}}}

\nc{\OGr}{\on{OGr}}
\nc{\wGr}{{\widetilde{\Gr}}}
\nc{\Sat}{{\on{Sat}}}
\nc{\Sp}{\on{Sp}}
\nc{\Ran}{\on{Ran}}
\nc{\Rat}{\on{Rat}}
\nc{\Rep}{\on{Rep}}
\nc{\Loc}{\on{Loc}}

\def\xcoch{\mathbb{X}_\bullet}
\def\xch{\mathbb{X}^\bullet}

\begin{document}

%
%
%
%
%
%

\title{An introduction to affine Grassmannians and the geometric Satake equivalence}

%
%
\author{Xinwen Zhu} \thanks{Partially supported by NSF DMS -1303296/1535464, a Sloan Fellowship and PCMI} 
\address{Department of Mathematics, California Institute of Technology, Pasadena, CA 91125}
\email{xzhu@caltech.edu} 

%
%
\begin{abstract}
We introduce various affine Grassmannians, study their geometric properties, and give some applications. We also discuss the geometric Satake equivalence.
These are the expanded lecture notes for a mini-course in 2015 PCMI summer school.  
\end{abstract}
\keywords{Park City Mathematics Institute}

%
%
\maketitle

%
%


\tableofcontents

\section*{Introduction}\addcontentsline{toc}{section}{Introduction}
\subsection{Some motivations.}
Let $F$ be a complete discrete valuation field, with $\mO\subset F$ its ring of integers and $k$ its residue field. Let $V=F^n$ be an $n$-dimensional $F$-vector space. A lattice of $V$ is a finitely generated $\mO$-submodule $\La$ of $V$ such that $\La\otimes F=V$. For example, $\La_0=\mO^n$ is a lattice in $V$. The main purpose of the lecture notes is to endow the set of lattices of $V$ (and its variants and generalisations) with a structure as an infinite dimensional algebraic variety, known as the affine Grassmannian $\Gr$, to study its geometric properties, and to give some of its applications.

But let us start with an incomplete list of motivations of studying such object.

The first motivation lies in the relationship between loop algebras/groups and Kac-Moody algebras/groups. Since every lattice can be translated to $\La_0$ by a linear automorphism of $V$, $\Gr$ is naturally identified with $\GL_n(F)/\GL_n(\mO)$ as a set. When $F=k((t))$, the group $\GL_n(F)$ is closely related to an affine Kac-Moody group over $k$. Realising $\Gr$ as a (partial) flag variety of the Kac-Moody group is important for representation theory of affine Kac-Moody algebras. A precise relation between them will be briefly discussed in \S\ \ref{Kac-Moody}.

The second motivation is the study of moduli of vector bundles on an algebraic curve. Let $X$ be a smooth projective algebraic curve over $k$, and let $x\in X(k)$ be a point. We identify the complete local ring of $X$ at $x$ with $k[[t]]$ by choosing a local coordinate $t$ around $x$. Then one can identify $\Gr$ with the following set
\[\Gr=\left\{ \mE \mbox{ is a vector bundle on } X, \ \beta: \mE|_{X-x}\simeq \mO_{X-x}^n \mbox{ is a trivialisation}\right\}.\]
Therefore, there is a map from $\Gr$ to the set of isomorphism classes of rank $n$ vector bundles on $X$. The latter in fact is the underlying set of points of a much more sophisticated algebro-geometric object, namely the moduli stack $\Bun_n$ of rank $n$ vector bundles on $X$. It is therefore very desirable to upgrade the above set-theoretical map to a morphism
\[u_x:\Gr\to \Bun_n\]
of algebro-geometric objects.
Such a morphism is called the uniformization map, and is an important tool to study certain geometric properties of $\Bun_n$. We refer to \S\ \ref{one point uniformization} for more detailed explanations of this idea.

The third motivation comes from the study of some (locally) closed subvarieties in $\Gr$, arising from various aspects of the Langlands program. The list of such subvarieties include affine Springer fibers (related to harmonic analysis on reductive groups over non-archimedean local fields), affine Deligne-Lusztig varieties (related to reduction of Shimura varieties mod $p$) an Kisin varieties (related to deformation spaces of Galois representations). We refer to \S\ \ref{digression} for an introduction of these varieties.

The fourth motivation also comes from the Langlands program. Namely, the Satake isomorphism\footnote{Or rather, Langlands' re-interpretation of the Satake isomorphism.}, which describes the ring of $\GL_n(\mO)$-bi-invariant functions on $\GL_n(F)$, is the starting point of the Langlands duality. It turns out that the Satake isomorphism admits a vast generalisation, known as the geometric Satake equivalence. This is the starting point of the geometric Langlands program, and in recent years also finds some unexpected applications to the classical arithmetic Langlands program (e.g. see \cite{Laf}). But even in order to formulate the geometric Satake equivalence, we need to know the algebro-geometric structure on $\Gr$.

\subsection{Scope and contents.}

\subsubsection{}The notes are based on five lectures given at the Park City Mathematical Institute in Summer of 2015. They are aimed as a serious introduction to the subject. That said, we will address (what we think) the most basic and important aspects of the subject. Certainly there are many other important topics untouched in the notes, partly due to the limitation of the space and partly due to the author's incapability.  

On the other hand, we hope that these notes form a supplement rather than a supplantation of the existing literature. Therefore, we do not give a detailed proof of every result, in particular when such a proof already exists in literature. Instead we shall point out relevant references. On the other hand, there are some well-known and widely used results for which no (complete) proofs seem to exist in literature. In that case, we will try to supply the details.

The notes contain a few results that did not appear in literature before, although it is quite plausible that they are well-known to experts. Such results include, for example, the calculation of the relative Picard sheaves of Beilinson-Drinfeld Grassmannians (\S\ \ref{line bundle on BD Grass}).

\subsubsection{}
We briefly describe the contents of each section.

Section \ref{Lecture I} explains basic constructions, including various (moduli) interpretations of affine Grassmannians. We will also introduce the important determinant line bundle, and discuss a relationship between affine Grassmannians over complex numbers and the based loop groups. Although mostly  we consider groups over an equal characteristic local field, at the end of this section we briefly discuss the recent progress of defining affine Grassmannians for  $p$-adic groups.

Section \ref{Lecture II} studies some basic geometry of affine Grassmannians. We will introduce various subvarieties in the affine Grassmannian, including Schubert varieties, opposite Schubert ``varieties'', transversal slices, and some important subvarieties arising from the Langlands program. We study the Picard group of the affine Grassmannian for a simple simply-connected group $G$.  We also briefly discuss the relation between affine Grassmannians and partial flag varieties of affine Kac-Moody groups.

The interpretation of affine Grassmannians as certain moduli of bundles on curves admits a vast generalisation, known as Beilinson-Drinfeld Grassmannians. Section \ref{Lecture III} is devoted to these objects. A salient feature of Beilinson-Drinfeld Grassmannians is the factorisation property. We explain this property in some details, in particular formulate it via the Ran space of a curve. We also discuss line bundles on Beilinson-Drinfeld Grassmannians.

Section \ref{Lecture IV} discusses some applications of affine Grassmannians to the study of moduli of bundles on curves. The basic tools are uniformization theorems. In fact, there are two types of such theorems: One point uniformization which is valid when $G$ is semisimple, and the ad\`elic uniformization, which is the algebro-geometric version of Weil's ad\`elic description of the isomorphism classes of $G$-bundles on a curve.  As an application, we give a description of the Picard group of the moduli of $G$-bundles on a smooth projective curve when $G$ is simple and simply-connected, and interpret global sections of line bundles on the moduli space as conformal blocks.

Section \ref{Lecture V} is devoted to the geometric Satake equivalence. The theory is naturally divided into two parts. The first part is the construction of a Tannakian category from certain category of perverse sheaves on the affine Grassmannian. A key ingredient is the fusion product structure on the category, whose construction relies on the factorisation property of Beilinson-Drinfeld Grassmannians. The second part is the study of the Tannakian group. Here the main ingredient is the Mirkovi\'c-Vilonen theory on the geometry of semi-infinite orbits. We also explain how to extend the geometric Satake equivalence to the forms that are suitable for applications to the (geometric) Langlands program. At the end of this section, we explain how to recover the classical Satake isomorphism from the geometric one. 

There is also an appendix where we collect some backgrounds on $\ell$-adic sheaves that are used in the Section \ref{Lecture V}.

\subsection{Conventions and notations.}
Now we fix notations and conventions that will be used throughout the notes. Readers are \emph{strongly} suggested to skip this subsection and come back only when necessary.
 
 \subsubsection{$k$-spaces.}\label{Intro:sp}
Let $k$ denote a commutative field. 
Let $\on{Aff}_k$ denote the category of affine $k$-schemes, i.e. the category opposite to the category $k\on{-alg}$ of commutative $k$-algebras. A presheaf (resp. prestack) is a covariant functor (resp. $2$-functor)  from $\Aff_k$ to the category of sets (resp. $2$-category of groupoids). We regard a set as a discrete groupoid so sometimes regard a presheaf as a discrete prestack.

We endow $\on{Aff}_k$ with the fpqc topology. I.e., a cover of $R$ is a faithfully flat map $R\to R'$. A $k$-\emph{space} is a sheaf on $\on{Aff}_k$ with respect to this topology. Explicitly, a space $\mF$ is a covariant functor $k\on{-alg}\to \on{Set}$  
\begin{enumerate}
\item  that respects finite products, i.e.
\[\mF(\prod_i R)\to \prod_i \mF(R_i)\]
is an isomorphism;
\item and such that if $R\to R'$ is faithfully flat, 
\begin{equation}\label{fpqc sheaf}
\mF(R)\to \mF(R')\rightrightarrows\mF(R'\otimes_RR') 
\end{equation}
is an equaliser. 
\end{enumerate}
Morphisms between two spaces are natural transformations of functors. The category of $k$-spaces is denoted by $\on{Sp}_k$. It is well-known that $\on{Sp}_k$ contains the category $\on{Sch}_k$ of $k$-schemes as a full subcategory. Recall that a map $f:\mF\to\mG$ of presheaves is called schematic if for every scheme $T$, the fiber product $\mF\times_{\mG}T$ is a scheme. One can similarly define the notion of a stack. 

Let $X$ and $Y$ be two prestacks. The fiber product $X\times Y$ always means $X\times_kY$. If $Y=\Spec R$, it is sometimes also written as $X_R$ or $X\otimes R$ or $X\otimes_k R$. Sometimes, we also write $X\times_YZ$ by $X\otimes_AB$ if $Y=\Spec A$ and $Z=\Spec B$.

Let $X$ be a prestack over $k$. We define the ring of regular functions on $X$ as
\begin{equation}\label{reg fun}
\Gamma(X,\mO_X)=\Hom(X, \bG_a),
\end{equation}
where $\Hom$ is calculated in the category of prestacks. I.e. a regular function $f$ on $X$ is a coherent way to assign every $x\in X(R)$ an element $f_x\in R$, satisfying the natural compatibility conditions.
Likewise, we define the Picard groupoid of line bundles on $X$ as
\begin{equation}\label{Pic line}
\Pic(X)=\Hom(X,\bB\bG_m),
\end{equation}
where $\bB\bG_m$ is the classifying stack of $\bG_m$ (see \S~\ref{Intro:torsor} below).
I.e. an object in $\Pic(X)$, called a line bundle on $X$, is a coherent way to assign every $x\in X(R)$ a line bundle on $\Spec R$, satisfying the natural compatibility conditions.

\subsubsection{Ind-schemes.}\label{Intro:Indsch}
In the notes a special class of $k$-spaces play a prominent role.

\begin{dfn}\label{ind-sch}
An ind-scheme over $k$ is a space that can be written as a \emph{filtered} direct limit of subspaces 
\begin{equation}\label{presentation}
X=``\underrightarrow{\lim}''_{i\in I}X_i
\end{equation}
such that each $X_i$ is represented by a scheme and that for every arrow
 $i\to j$ in $I$, the corresponding map $X_i\to X_j$ is a closed embedding.
\end{dfn}

\begin{rmk}In literature, ind-schemes defined as above are sometimes called strict ind-schemes to distinguish them from general ind-objects in the category of schemes. The latter notion will never appear in the notes.
\end{rmk}

\begin{dfn}Let $X$ be an ind-scheme over $k$. Then $X$ is called: \emph{(i) ind-affine;  (ii) ind-of finite type over $k$; (iii) ind-proper over $k$; (iv) ind-projective over $k$; (v) reduced}
if in some presentation as \eqref{presentation},
each $X_i$ can be chosen to be (i) affine; (ii) of finite type over $k$; 
(iii) proper over $k$; (iv) projective over $k$; (v) reduced. We define $X_{\on{red}}:=``\underrightarrow{\lim}''_{i\in
I}(X_i)_{\on{red}}$. 
\end{dfn}

\begin{ex}A formal scheme is an ind-scheme such that $X_{\on{red}}$ is
a scheme.
\end{ex}

\subsubsection{Torsors.}\label{Intro:torsor}
Let $G$ be a $k$-group (i.e. a group object in the category of $k$-spaces). A map $E\to X$ of prestacks is called a $G$-torsor if there is an action of $G$ on $E$ such that: 
\begin{enumerate}
\item the map $E\to X$ is $G$-equivariant (where $G$ acts trivially on $X$); and
\item for every $\Spec R\to X$, there is a faithfully flat map $\Spec R'\to \Spec R$ such that $\Spec R'\times_XE$ is $G$-equivariantly isomorphic to $\Spec R'\times G$. 
\end{enumerate}
If $G$ is an affine group scheme over $k$, this is equivalent to requiring that the map $E\to X$ is schematic and faithfully flat and that the natural action map $G\times E\cong E\times_X E$ is an isomorphism. We say that this $G$-torsor is fppf (resp. \'etale, resp. Zariski) locally trivial if $\Spec R'$ in (2) can be chosen such that $\Spec R'\to \Spec R$ is a fppf (resp. \'etale, resp. Zariski) cover.

Let $X$ be a $k$-space with an action of $k$-group $G$. We denote by $[X/G]$ the fpqc stack whose $R$-points are $G$-torsors $E$ on $\Spec R$ together with a $G$-equivariant map $E\to X$. The natural projection $X\to [X/G]$ then is a $G$-torsor. If $X=\Spec k$, $[X/G]$ is usually denoted by $\bB G$, called the classifying stack of $G$. If the action of $G$ on $X$ is free, then $[X/G]$ is a $k$-space. 
Let $E\to X$ be a $G$-torsor and $Y$ be a $G$-space (i.e. a $k$-space with a $G$-action), we define the twisted product as
\[X\tilde\times Y:= E\times^G Y:=[E\times Y/G],\]
where $G$ acts on $E\times Y$ diagonally. There is a natural projection $X\tilde\times Y\to X$ with fibers isomorphic to $Y$. If $Y=V$ is a linear representation of $G$, we sometimes also write $E\times^GV$ by $V_E$, and regard it as a vector bundle on $X$.

If $G=\GL_n$ is the general linear group, and $V=k^n$ is the standard representation of $\GL_n$. Then $(\mE\to X)\mapsto (V_\mE\to X)$ induces an equivalence of groupoids between $\GL_n$-torsors on $X$ and rank $n$ vector bundles on $X$. We will use this equivalence freely in the notes.

There is an obvious generalisation of the above notations for a (non-constant) group space $\underline G$ over some base $k$-scheme $S$. We will denote by $\mE^0$ \emph{the} trivial $\underline G$-torsor.

\subsubsection{The disc.}\label{Intro:disc}
We write 
$$D=\Spec k[[t]],\quad D^*=\Spec k((t)),$$ and call them \emph{the} disc and \emph{the} punctured disc to emphasise a preferred choice of the coordinate $t$. If $R$ is a $k$-algebra, we write 
$$D_R: = D\hat\times\Spec R:=\Spec R[[t]], \quad D_R^*:= D^*\hat\times \Spec R:= \Spec R((t)),$$
and think them as a family of discs (resp. punctured discs) parameterised by $\Spec R$.

Unless otherwise stated, by a local field $F$, we mean a field over $k$ complete with respect to a discrete valuation that is trivial on $k$, and whose residue field is $k$. We denote by $\mO\subset F$ the ring of integers. 
Then upon a choice of  a uniformizer $t\in \mO$, we have $F\cong k((t))$ and $\mO\cong k[[t]]$.
Let
$$\mO\hat\otimes R:=\underleftarrow\lim (\mO/t^n\otimes R), \quad F\hat\otimes R:=\mO\hat\otimes R[1/t]$$ so $\mO\hat\otimes R\cong R[[t]]$, and $F\hat\otimes R\cong R((t))$.
We write $D_{F,R}=\Spec (\mO\hat\otimes R)$ and $D^*_{F,R}=\Spec (F\hat\otimes R)$.

Let $X$ be a smooth geometrically connected curve over $k$. Let $|X|$ denote the set of closed points of $X$. For $x\in |X|$, let $X^*= X-\{x\}$. 
Let $\mO_x$ denote the complete local ring of $X$ at $x$ and $F_x$ denote the fractional field of $\mO_x$. For a $k$-algebra $R$, We write $D_{x,R}$ and $D^*_{x,R}$ instead of $D_{F_x,R}$ and $D^*_{F_x,R}$ for simplicity.
If $x:\Spec R\to X$ is an $R$-point of $X$, we write $\Ga_x\subset X_R$ the graph of $x$, and let $\hat{\Ga}_x$ denote the formal completion of $\Ga_x$ in $X_R$.

\subsubsection{Group theoretical data.}\label{Intro:group}

If $G$ is a connected split reductive group over $k$, 
we denote by $G_\der$ the derived group of $G$, $G_\ad$ the adjoint form of $G$, and $G_\s$ the simply-connected cover of $G_\der$.

We denote by $T$ the abstract Cartan of $G$. Recall that it is defined as the quotient of a Borel subgroup $B\subset G$ by its unipotent radical. It turns out that $T$ is independent of the choice of $B$ up to a canonical isomorphism. When we need to embed $T$ into $G$ as a maximal torus, we will state it explicitly.

Let $\xch=\xch(T)$ denote the weight lattice, i.e. the free abelian group $\Hom(T,\bG_m)$, and $\xcoch=\xcoch(T)$ denote the coweight lattice, i.e. the dual of $\xch(T)$. Let $\Phi\subset \xch$ (resp. $\Phi^\vee\subset \xcoch$) denote the set of roots (resp. coroots). The quadruple $(\xch,\xcoch,\Phi,\Phi^\vee)$ is called the root datum of $G$. Let $W$ denote the Weyl group. The Borel subgroup $B\subset G$ determines the semi-group of dominant coweights $\xcoch(T)^+\subset \xcoch(T)$ and the subset of positive roots $\Phi^+\subset\Phi$, which turn out to be independent of the choice of $B$.  Then there is a canonical bijection
 \[\xcoch(T)/W\cong \xcoch(T)^+.\]
Recall that there is a partial order $\leq$ on $\xcoch(T)$: $\la\leq \mu$ if $\mu-\la$ is a non-negative integral linear combinations of simple coroots; we say $\la<\mu$ if $\la\leq \mu$ but $\la\neq \mu$. The restriction of the partial order to $\xcoch(T)^+$ sometimes is called the Bruhat order on $\xcoch(T)^+$. 
We will usually denote by $2\rho\in\xch(T)$ the sum of all positive roots, and let $\rho=\frac{1}{2}(2\rho)\in \xch(T)\otimes\bQ$.

\subsection{Acknowledgement.} The notes are based on lectures given at the Park City Mathematical Institute in Summer of 2015. I would like to thank the organisers, audiences and the staff of the PCMI. I would also like to thank Weizhe Zheng for useful discussions related to \S~\ref{ULA sheaf} and  Michael Harris, Jiuzu Hong and Joel Kamnitzer for valuable comments.

\section{Lecture I: Affine Grassmannians and their first properties}\label{Lecture I}
We start to define the objects we are going to study.

\subsection{The affine Grassmannian of $\GL_n$.}\label{aff GLn}
First it makes sense to talk about a family of lattices in $k((t))^n$. 
\begin{dfn}
Let $R$ be a $k$-algebra. An $R$-family of lattices in $k((t))^n$ is a finitely generated projective $R[[t]]$-submodule $\La$ of $R((t))^n$ such that $\La\otimes_{R[[t]]}R((t))=R((t))^n$. 
\end{dfn}

\begin{dfn}
The affine Grassmannian $\Gr_{\GL_n}$ (for $\GL_n$) is the presheaf that assigns every $k$-algebra $R$ the set of $R$-families of lattices in $k((t))^n$. 
\end{dfn}

For simplicity, we write $\Gr_{\GL_n}$ by $\Gr$ in this subsection.
The whole subject is based on the following theorem. 

\begin{thm}\label{rep}
The affine Grassmannian $\Gr$ is represented by an ind-projective scheme.
\end{thm}
Recall that this in particular means that $\Gr$ can be written as an increasing union of finite dimensional projective schemes (cf. \S~\ref{Intro:Indsch}). 

Since the actual proof of the theorem is a little bit technical but the underlying idea is very simple, we first give an informal account (at the level of $k$-points). Let $\La_0=k[[t]]^n$ denote the standard lattice. Let $\Gr^{(N)}\subset \Gr$ denote the subspace classifying lattices $\La$ in $k((t))^n$ that land in between $t^N\La_0\subset \La\subset t^{-N}\La_0$, then $\Gr$ is an increasing union of these $\Gr^{(N)}$. It is enough to show that $\Gr^{(N)}$ is represented by a projective scheme.
For a lattice $\La\in \Gr^{(N)}$, the quotient $\La/t^N\La_0$ can be regarded as a subspace of $k^{2nN}\cong t^{-N}\La_0/t^N\La_0$, stable under the action of $t$. In this way, $\Gr^{(N)}$ is realised as a closed subscheme of the usual Grassmannian variety $\Gr(2nN)$ classifying finite dimensional subspaces in $k^{2nN}$, and therefore is a projective scheme. Now we turn this idea into a rigorous proof.
\begin{proof}
Given an $R$-family of lattice $\La\subset R((t))^n$, there always exists some integer $N$ large enough such that 
\begin{equation}\label{subfun}
t^NR[[t]]^n\subset \La\subset t^{-N}R[[t]]^n.
\end{equation}
So $\Gr$ is the union of subfunctors $\Gr^{(N)}$ consisting of those lattices satisfying \eqref{subfun}.
The key observation is the following lemma.
\begin{lem}\label{projectivity}
Via the map
$$\La\mapsto Q=t^{-N}R[[t]]^n/\La,$$ $\Gr^{(N)}$ is identified with the presheaf $\Gr^{(N),f}$
\[\Gr^{(N),f}:=\left\{R[[t]] \mbox{-quotient modules of } \frac{t^{-N}R[[t]]^n}{t^NR[[t]]^n} \mbox{ that are projective as }R\mbox{-modules}\right\}.\] 
\end{lem}
As we shall see, although this statement is quite obvious for $k$-points, it is subtle for general $R$.

We will need another lemma.
Let $\Gr(m)$ denote the usual Grassmannian variety, classifying finite dimensional subspaces in $k^m$. It is the disjoint union of those $\Gr(r,m)$, classifying those subspaces of dimension $r$. 
\begin{lem}\label{Embedding}
The functor $\Gr^{(N),f}$ is represented by a closed subscheme of $\Gr(2nN)$. 
\end{lem}
We first prove Lemma \ref{Embedding}. We identify $t^{-N}k[[t]]^n/t^Nk[[t]]^n$ with $k^{2nN}$. Then multiplication by $t$ induces a nilpotent operator $\Phi$ on this vector space, and $\Gr^{(N),f}$ is the closed subscheme of $\Gr(2nN)$ classifying those subspaces of $k^{2nN}$ that are stable under $\Phi$. Note that in fact $\Gr^{(N),f}$ is a union of Springer fibers (\cite[\S~1]{Yu}).

Next we prove Lemma \ref{projectivity}. We first need to show that there is indeed a well-defined map
\begin{equation}\label{I:map}
\Gr^{(N)}\to \Gr^{(N),f}
\end{equation}
as claimed in the lemma. I.e. we need to show that $Q$ is projective as an $R$-module. Since $\Lambda/t\Lambda$ is $R$-projective,
$$R((t))^n/\Lambda\cong\bigoplus_{k\geq 0}t^{-k-1}\Lambda/t^{-k}\Lambda$$ is also $R$-projective. Then the projectivity of $Q$ follows from the exact sequence of $R$-modules
\[0\to Q\to R((t))^n/\Lambda\to R((t))^n/t^{-N}R[[t]]^n\to 0.\]

Therefore \eqref{I:map} realises $\Gr^{(N)}$ as a subpresheaf of $\Gr^{(N),f}$. We also need to show that  \eqref{I:map} is surjective on $R$-points for every $R$. In other words, we need to show that for every $k$-algebra $R$ and an $R$-point of $\Gr^{(N),f}$ represented by a quotient map $ t^{-N}R[[t]]^n/t^NR[[t]]^n\to Q$, the $R[[t]]$-module
\[\Lambda:=\ker( t^{-N}R[[t]]^n\to  t^{-N}R[[t]]^n/t^NR[[t]]^n\to Q)\]
is a finite projective $R[[t]]$-module. Then $\Lambda\otimes R((t))\simeq R((t))^n$ automatically.

Since $\Gr^{(N),f}$ is locally of finite type (by Lemma \ref{Embedding}), we have (cf. \cite[Tag01ZC]{St})
$$\Gr^{(N),f}(\underrightarrow\lim R_i)=\underrightarrow\lim \Gr^{(N),f}(R_i).$$ By writing $R$ as a union of finitely generated sub-$k$-algebras, we can therefore 
assume that $R$ is finitely generated over $k$.
Then the natural ring map $R[t]\to R[[t]]$ is flat since $R$ is noetherian (cf. \cite[Tag00MB]{St}), and induces an isomorphism $t^{-N}R[t]^n/t^NR[t]^n\cong t^{-N}R[[t]]^n/t^NR[[t]]^n$. Then we can define an $R[t]$-module $\Lambda_f$ as
\[\Lambda_f:= \ker( t^{-N}R[t]^n\to  t^{-N}R[t]^n/t^NR[t]^n\to Q).\]
By the flatness of the map $R[t]\to R[[t]]$, $\Lambda=\Lambda_f\otimes_{R[t]}R[[t]]$. Therefore, it is enough to prove that $\Lambda_f$ is a finite projective $R[t]$-module. Clearly, $\Lambda_f$ is finitely generated, so it is enough to prove that it is flat over $R[t]$. Since $\La_f$ is already $R$-flat, by \cite[Tag00MH]{St} it is enough to show that for every $x\in \Spec R$, with residue field $K$, the base change $\Lambda_f\otimes_{R[t]}K[t]$ is flat over $K[t]$. Since $Q$ is $R$-flat, we have
\[\Lambda_f\otimes_{R[t]}K[t]= \ker ( t^{-N}K[t]^n\to Q\otimes_R K),\]
which is  $t$-torsion free and therefore is flat over $K[t]$, as desired. This finishes the proof of Lemma \ref{projectivity}.

Finally, it is not hard to see that the inclusion $\Gr^{(N)}\to \Gr^{(N+1)}$ is a closed embedding. The theorem then follows.
\end{proof}

\begin{rmk}The idea of replacing $R[[t]]$ by $R[t]$ is a special case of Beauville-Laszlo's realisation of $\Gr$ via a global curve. See Theorem \ref{glob aff} below.
\end{rmk}

\subsection{Affine Grassmannians of general groups.}
Let $G$ be a smooth affine $k$-group. Recall the notion of torsors (cf. \S~\ref{Intro:torsor}) and $R$-family of discs (cf. \S~\ref{Intro:disc}).
We define the affine Grassmannian $\Gr_{G}$ of $G$ as
\begin{equation}\label{loc aff}
\Gr_{G}(R)=\left\{(\mE,\beta)\ \left|\ \begin{split}&\mE \mbox{ is a } G\mbox{-torsor on } D_R , \mbox{ and }\\
& \beta:\mE|_{D^*_R}\cong  \mE^0|_{D^*_R}\mbox{ is a trivialisation }\end{split}\right.\right\}.
\end{equation}

Note that this definition makes sense for non-constant group schemes. Namely, if $\underline G$ is an affine group scheme smooth over $k[[t]]$, we can define the affine Grassmannian $\Gr_{\underline G}$ of $\underline G$ exactly as in \eqref{loc aff}, with $G$ replaced by $\underline G$. 

\begin{thm}\label{I:rep aff}
The presheaf $\Gr_{\underline G}$ is represented an ind-scheme, ind-of finite type over $k$. If $\underline G$ is reductive, then $\Gr_{\underline G}$ is ind-projective.
\end{thm}
\begin{rmk}
(i) Certainly in the above definition and theorem one can replace $k[[t]]$ by the ring of integers of any local field. 

(ii) It is easy to see that for the additive group $G=\bG_a$, $\Gr_{\bG_a}$ is represented by an inductive limit of affine spaces and therefore is not ind-projective. 
\end{rmk}

\begin{rmk}In the notes, we will state the results for general group schemes when no additional effort is needed. Readers unfamiliar with group schemes can safely replace $\underline G$ by a $k$-group $G$. 
\end{rmk}

The proof of the theorem is based on the following two results.
\begin{prop}\label{lin rep}
There exists a faithful $k[[t]]$-linear representation $\rho: \underline G\to \GL_n$ such that $\GL_n/\underline G$ is quasi-affine. In addition, if $\underline G$ is reductive, one can choose $\rho$ such that $\GL_n/\underline G$ is affine. 
\end{prop}
\begin{proof}See \cite[Proposition 1.3]{PR} for the first statement. See \cite[Corollary 9.7.7]{Ap} for the second statement.
\end{proof}

\begin{prop}\label{gen G}
Let $\rho:\underline G\to \GL_n$ be a linear representation such that $\GL_n/\underline G$ is quasi-affine, then the induced map $f_\rho:\Gr_{\underline G}\to \Gr_{\GL_n}$ is a locally closed embedding. In addition, if $\GL_n/\underline G$ is affine, it is a closed embedding.
\end{prop}
The proof is not related in the rest of the notes and readers who take the proposition on faith can skip it.
\begin{proof}
Let $\Spec R\to \Gr_{\GL_n}$ be a map, represented by $(\mE,\beta)$.
We want to show that the base change map 
$$\mF:=\Spec R\times_{\Gr_{\GL_n}}\Gr_{\underline G}\to \Spec R$$ is a locally closed embedding, and is a closed embedding if $\GL_n/\underline G$ is affine. Let us denote $\pi:\mE\to D_R$ for the projection, and regard $\beta$ as a section of $\pi$ over $D^*_R$. Let $$\bar\pi:[\mE/\underline G]\to D_R$$ denote the fpqc quotient. Since $\GL_n/\underline G$ is (quasi-)affine, by the faithfully flat descent for (quasi-)affine morphisms (cf. \cite[Tag0244, Tag0246]{St}), there is a scheme $W$, affine and of finite presentation over $D_R$, such that $[\mE/\underline G]\subset W$ is quasicompact open. The section $\beta$ induces a section of $\bar\pi$ over $D^*_R$, denoted by $\bar\beta$. Recall that giving a reduction of $\mE$ to a $\underline G$-torsor $\mE'$ is the same as giving a section of of $\bar\pi$. It follows that $\mF$ is the presheaf over $\Spec R$ that assigns every $R\to R'$ the set of sections $\beta'$ of $\bar\pi$ over $D_{R'}$ such that $\beta'|_{D^*_{R'}}=\bar\beta|_{D^*_{R'}}$. We have
\begin{lem}
Let $p:V\to D_R$ be an affine scheme of finite presentation, and $s$ be a section of $p$ over $D_R^*$. Then the presheaf over $\Spec R$ that assigns every $R\to R'$ the set of sections $s'$ of $p$ over $D_{R'}$ such that $s'|_{D^*_{R'}}=s|_{D^*_{R'}}$ is represented by a closed subscheme of $\Spec R$.
\end{lem}
\begin{proof}
We can embed $V\subset \bA^N_{D_R}$ as a closed subscheme of some (big) affine space. Then using coordinates, we can write $s$ as
\[s=(s_1(t),\ldots,s_N(t)),\quad s_i(t)=\sum s_{ij}t^{j}\in R((t)).\]
The presheaf in question then is represented by the closed subscheme $\Spec A\subset \Spec R$ defined by the ideal generated by $\{s_{ij}=0, i=1,\ldots,N,j<0\}$.
\end{proof}
Applying this lemma to $V=W$ and $s=\bar\beta$, we obtain a closed subscheme $\Spec A\subset\Spec R$ and a tautological map $\beta_A: D_A\to W_{D_A}$ whose restriction to $D_A^*$ is $\bar\beta|_{D^*_A}$. Base change along the closed embedding $\Spec A\to D_A$ defined by $t=0$ induces a section $s_0:\Spec A\to W\times_{D_R}\Spec A$. Then $\mF$ is represented by the open subscheme $s_0^{-1}([\mE/\underline G]\times_{D_R}\Spec A)$ of $\Spec A$.
\end{proof}

\begin{rmk}\label{general base}
In the notes, we work over a fixed base field $k$. But for certain applications (e.g. \cite{PR2,PZ}), one needs to work over a base scheme $S=\Spec A$. Starting from an affine smooth group scheme $\underline G$ over $A[[t]]$, one can then similarly define the affine Grassmannian $\Gr_{\underline G}$ by \eqref{loc aff}, which is a presheaf over $S$. If the corresponding Proposition \ref{lin rep} holds in this generality, then the same arguments as above show that $\Gr_{\underline G}$ is representable as an ind-scheme ind-of finite presentation over $S$. It is known Proposition \ref{lin rep} holds if: (i) $\underline G$ is the base change of a linear reductive over $S$ ; (ii) $S$ is a Dedekind domain (cf. \cite[Appendix]{PZ}).
\end{rmk}

Interesting (and important) non-constant group schemes are parahoric group schemes from Bruhat-Tits theory (see \cite{Ti} for a summary). Instead of explaining what they are in general, we give two basic examples. 
\begin{ex}\label{hyperspacial and Iwahori}
Let $G$ be a connected reductive group over $k$ and $B\subset G$ be a Borel subgroup.
\begin{enumerate}
\item The group scheme $\underline G:=G\otimes_k \mO$ is parahoric, usually called a hyperspecial parahoric of $G$. 
\item Let $\underline G$ be as above. There is the evaluation map $\on{ev}: \underline G(\mO)\to G(k)$ given by $g\mapsto (g\mod t)$. Then there exists a smooth fiberwise connected affine group scheme $\underline G'$ over $\mO$ together with a homomorphism $\underline G'\to \underline G$, which induces an isomorphism $\underline G'\otimes_{\mO}F\cong \underline G\otimes_\mO F=G\otimes_k F$ and identifies $\underline G'(\mO)$ with $\on{ev}^{-1}(B(k))$. Usually $\underline G'$ is called an Iwahori group scheme of $G$. 
\end{enumerate}
We refer to Example \ref{hyperspacial and Iwahori2} for more discussions.
\end{ex}
\begin{rmk}\label{terminology} 
In this notes, we call $\Gr_{\underline G}$ the affine Grassmannian of $\underline G$. This is \emph{inconsistent} with the standard terminology in literature. For example,
if $\underline G$ is an Iwahori group scheme of $G$ as above, then $\Gr_{\underline G}$  is usually called the affine flag variety of $G$ in literature, and is usually denoted by $\Fl_{G}$.
\end{rmk}

The following proposition generalises the second statement of Theorem \ref{I:rep aff}.
\begin{prop}\label{indproper}
Let $\underline G$ be a fiberwise connected smooth affine group scheme over $\mO$ whose generic fiber is connected reductive. Then $\Gr_{\underline G}$ is ind-projective if and only if $\underline G$ is a parahoric group scheme.
\end{prop}
This is similar to the classical result that a homogenous space of a connected reductive group $G$ is projective if and only if it is a partial flag variety of $G$. Since this result is not used elsewhere in the notes, we omit the proof.

Let us introduce  the convolution Grassmannian $\Gr_{\underline G}\tilde\times\Gr_{\underline G}$. It is defined as
\begin{equation}\label{I:conv}
\Gr_{\underline G}\tilde\times\Gr_{\underline G}(R)=\left\{(\mE_1,\mE_2,\beta_1,\beta_2)\ \left|\ \begin{split}&\mE_1,\mE_2 \mbox{ are } \underline G\mbox{-torsors on } D_R, \\
&  \beta_1:\mE_1|_{D_R^*}\cong \mE^0|_{D_R^*}, \beta_2:\mE_2|_{D_R^*}\cong  \mE_1|_{D_R^*}\end{split}\right.\right\}.
\end{equation}
There is a map
\begin{equation}\label{conv m0}
m:\Gr_{\underline G}\tilde\times\Gr_{\underline G}\to \Gr_{\underline G},\quad (\mE_1,\mE_2,\beta_1,\beta_2)\mapsto (\mE_2, \beta_1\beta_2),
\end{equation}
and a natural projection
$$\pr_1:\Gr_{\underline G}\tilde\times\Gr_{\underline G}\to \Gr_{\underline G}, \quad (\mE_1,\mE_2,\beta_1,\beta_2)\mapsto (\mE_1,\beta_1),$$
which together induce an isomorphism
\begin{equation}\label{I: isom}
(\pr_1,m):\Gr_{\underline G}\tilde\times\Gr_{\underline G}\cong \Gr_{\underline G}\times \Gr_{\underline G}.
\end{equation} 
In particular, the convolution Grassmannian is representable.  The map $m$ is usually called the convolution map.

Similarly, there exists the $n$-fold convolution Grassmannian $\Gr_{\underline G}\tilde\times\cdots\tilde\times\Gr_{\underline G}$, classifying $(\mE_i,\beta_i), \ i=1,\ldots,n$, where $\mE_i$ is a $\underline G$-torsor on $D_R$ and $\beta_i: \mE_{i}|_{D_R^*}\simeq \mE_{i-1}|_{D_R^*}$ is an isomorphism (where $\mE_0=\mE^0$ is trivial). 
There is a natural projection $\pr_j$ to the $j$-fold convolution Grassmannian for each $j=1,\ldots,n-1$ by remembering $(\mE_i,\beta_i),\ i=0,\ldots,j$, and
an $n$-fold convolution map
\begin{equation}\label{conv m}
m:\Gr_{\underline G}\tilde\times\cdots\tilde\times\Gr_{\underline G}\to \Gr_{\underline G},
\end{equation}
sending $(\mE_i,\beta_i)$ to $(\mE_n, \beta_1\cdots\beta_n)$.

\subsection{Groups attached to the punctured disc.}
There is another useful interpretation of $\Gr_{\underline G}$. First, we introduce jet and loop groups.
\begin{dfn}
Let $X$ be a presheaf over $\mO=k[[t]]$. The space of $n$-jets $L^nX$ of $X$ is the presheaf that assigns every $R$ the set
$$L^nX(R)=X(R[t]/t^n).$$ 
The formal jet space (sometimes also called the arc space, or the positive loop space) $L^+X$ of $X$ is the presheaf that assigns every $R$ the set 
$$LX(R)=X(R[[t]]).$$ 
Let $X$ be a presheaf over $F=k((t))$. The loop space $LX$ of $X$ is the presheaf that assigns every $R$ the set 
$$LX(R)=X(R((t))).$$
If $X$ is a presheaf over $k$, we write $L^+(X\otimes_k\mO)$ by $L^+X$ and $L(X\otimes_k F)$ by $LX$ for simplicity, if no confusion will arise.
\end{dfn}

\begin{prop}~
\begin{enumerate}
\item Let $X$ be a scheme of finite type over $\mO=k[[t]]$. Then $L^nX$ is represented by a scheme of finite type over $k$. In addition 
$$L^+X\cong \underleftarrow\lim\ L^nX$$ 
is represented by a scheme. If $X$ is affine, so is $L^+X$. If $f:X\to Y$ is an open (resp. closed embedding), so is $L^+f: L^+X\to L^+Y$.

\item Let $X$ be an affine scheme of finite type over $F=k((t))$. Then $LX$ is represented by an ind-affine scheme over $k$. If $f:X\to Y$ is a closed embedding, so is $Lf:LX\to LY$. In addition, if $X=\underline X\otimes_\mO F$ for some affine scheme $\underline X$ of finite type over $\mO$, then $L^+\underline X\subset LX$ is a closed subscheme. 
\end{enumerate}
\end{prop}
\begin{proof}See \cite[1.a]{PR}.
\end{proof}

\begin{ex}If $X=\bA^1$, then $L^+X(R)=R[[t]]=\{\sum r_it^i \mid r_i\in R, r_i=0 \mbox{ for } i<0 \}$ and $LX(R)=R((t))=\{\sum r_it^i\mid  r_i\in R, r_i=0 \mbox{ for } i\ll 0\}$. Therefore $L^+X\cong \Spec k[r_0,r_1,\ldots]$ and $LX=\underrightarrow\lim_i \Spec k[r_{-i}, r_{-i+1},\ldots]$.
\end{ex}

\begin{ex}\label{hyperspacial and Iwahori2}
We consider the group schemes as in Example \ref{hyperspacial and Iwahori}. The natural homomorphism $\underline G'\to \underline G$ then induces a closed embedding $L^+\underline G'\subset L^+\underline G$. Indeed, the evaluation map can be upgraded to a map of affine $k$-groups $\on{ev}:L^+G\to G$ and $L^+\underline G'=\on{ev}^{-1}(B)$. Note that, however, the homomorphism $\underline G'\to \underline G$ is not a closed embedding since it induces an isomorphism when base change to $F$.
\end{ex}

Back to the affine Grassmannian.  Note that there is a natural action
\begin{equation}\label{I:action}
L\underline G\times \Gr_{\underline G}\to \Gr_{\underline G},\quad (A, (\mE,\beta))\mapsto (\mE,A\beta).
\end{equation}

\begin{prop}\label{I:quotient}
The affine Grassmannian $\Gr_{\underline G}$ can be identified with the fpqc quotient $[L\underline G/L^+\underline G]$. 
\end{prop}
\begin{proof}
We first need a lemma.
\begin{lem}\label{et loc triv}
Every $\underline G$-torsor on $D_R$ can be trivialised over $D_{R'}$ for some \'etale covering $\Spec R'\to \Spec R$.
\end{lem}
\begin{proof}First the map $\mE\otimes_{R[[t]]}R\to\Spec R$ is smooth (since $\underline G$ is smooth) and therefore admits a section for some \'etale covering  $\Spec R'\to \Spec R$. By the smoothness of $\mE\to D_R$, such a section can be lifted to a section $\epsilon$ of $\mE$ over $D_{R'}$.
\end{proof}
Now we prove the proposition,
The claim is equivalent to saying that $L\underline G$ is an $L^+\underline G$-torsor over $\Gr_{\underline G}$. More precisely, for a $k$-algebra $R$, there is a canonical isomorphism
\[L\underline G(R)\cong\left\{(\mE,\beta, \epsilon)\ \left|\ \begin{split}&(\mE,\beta)\in \Gr_{\underline G}(R) \\
& \epsilon: \mE^0\cong \mE \mbox{ is a trivialisation}\end{split}\right.\right\}.\]
In fact, the desired isomorphism is given by $A\mapsto (\mE^0,A,\on{id})$ with the inverse map given by $(\mE,\beta,\epsilon)\mapsto A:=\beta\epsilon$.
\end{proof}

\begin{rmk}
According to the proof of the proposition, the $L^+\underline G$-torsor $L\underline G\to \Gr_{\underline G}$ is locally trivial for the \'etale topology.
If $\underline{G}=G\otimes \mO$ is constant reductive (or more generally so-called tamely ramified), it is even locally trivial for the Zariski topology (see Lemma \ref{big cell}). 
\end{rmk}

The convolution Grassmannian \eqref{I:conv} also admits a similar interpretation. Using the $L^+\underline G$-torsor $L\underline G\to \Gr_{\underline G}$  and the left action of $L^+\underline G$ on $\Gr_{\underline G}$ (defined by \eqref{I:action}), we can form a twisted product   $L\underline G\times^{L^+ \underline G}\Gr_{\underline G}$ (cf. \S\ \ref{Intro:torsor}). By a similar argument, one sees that $L\underline G\times^{L^+ \underline G}\Gr_{\underline G}$ is isomorphic to $\Gr_{\underline G}\tilde\times\Gr_{\underline G}$ (in particular our notations are consistent).

We have a direct corollary.
\begin{cor}The affine Grassmannian $\Gr_{\underline G}$ is formally smooth.
\end{cor}
\begin{proof}Indeed, both $L\underline G$ and $L^+\underline G$ are formally smooth.
\end{proof}
\begin{rmk}\label{formal smoothness}
Formal smoothness for an ind-scheme is a weak notion. For example, $L\bG_m$ is formally smooth, but it is easy to see that it is highly \emph{non-reduced}.
\end{rmk}

Here is a deeper application of loop groups. We assume that $k$ is algebraically closed and $\underline G_F=\underline G\otimes F$ is connected reductive. Let $I=\Gal(\overline F/F)$ be the Galois group of $F$.  It acts on the fundamental group $\pi_1(\underline G_F)$ of $\underline G_F$. Let $\pi_1(\underline G_F)_I$ denote the group of coinvariants. Let $p$ denote the characteristic exponent of $k$.
\begin{thm}\label{components}~
\begin{enumerate}
\item There is a canonical isomorphism $\pi_0(L\underline G)\cong \pi_1(\underline G_F)_I$.
 The map $L\underline G\to \Gr_{\underline G}$ induces an isomorphism $\pi_0(L\underline G)\cong \pi_0(\Gr_{\underline G})$. Let $(\Gr_{\underline G})^0$ denote the neutral connected component.
\item Assume that $\underline G_F$ be semisimple and simply-connected. Then $\Gr_{\underline G}$ is reduced.
\item Assume that $\underline G_F$ is semisimple and $p\nmid \pi_1(\underline G_F)$. Let $G_{\on{sc}}$ be the simply-connected cover of $\underline G_F$ and assume that it extends to an \'etale cover $\underline G_{\on{sc}}\to \underline G$. Then it induces an isomorphism
$\Gr_{\underline G_{\on{sc}}}\cong (\Gr_{\underline G})^0$. In particular, $\Gr_{\underline G}$ is reduced. 
\item In general, assume that $p\nmid \pi_1(G_\der)$, where $G_\der$ is the derived group of $\underline G_F$. Then the map $\underline G_{\on{sc}}\to \underline G$ induced an isomorphism $\Gr_{\underline G_{\on{sc}}}\cong (\Gr_{\underline G})^0_{\on{red}}$.
\end{enumerate}
\end{thm}
The proof of this theorem (in various generalities) can be extracted from various places (\cite[\S 4.5]{BD}, \cite{BL}, \cite{LS}, \cite{BLS}, \cite{Fa} and ultimately \cite{PR}). Note that (1) is analogous to the well-known fact in topology: $\pi_0$ of a loop space is the same as $\pi_1$ of the original space (this is in fact more than an analogy when $k=\bC$, see Theorem \ref{top triv}).

By this theorem, the study of many questions on $\Gr_{\underline G}$ reduces to the study of the case when $G$ is simply-connected.

\begin{ex}
(i) Let $G=\bG_m$. By the theorem $\pi_0(\Gr_G)\cong \bZ$. In fact, $(\Gr_G)_{\on{red}}=\pi_0(G)=\bZ$. However, as mentioned in Remark \ref{formal smoothness}, it is non-reduced.

(ii) It was shown in \cite[Remark 6.4]{PR} that $\Gr_G$ is non-reduced for $G=\on{PGL}_2$ and $\cha\ k=2$.
\end{ex}

In addition to the natural action of $L\underline G$ on $\Gr_{\underline G}$, there is another important symmetry on $\Gr_{\underline G}$, at least when $\underline G=G\otimes_k\mO$ is constant.
Let $\Aut(D)$ be the group ind-scheme of automorphisms of $\mO=k[[t]]$. Precisely, it is the presheaf of groups on $\Aff_k$ defined as
\[\Aut(D)(R)=\Aut_R(R[[t]]).\]
Any $R$-automorphism $\varphi$ of $R[[t]]$ is determined by the image $$\varphi(t)=a_0+a_1t+a_2t^2+\cdots\in R[[t]].$$ It is an easy exercise to show that there exists $\psi$ such that $\psi(\varphi(t))=t$ if and only if $a_0$ is nilpotent and $a_1$ is invertible. Therefore,
\[\Aut(D)\cong \on{Spf} k[[a_0]] \times \Spec k[a_1^{\pm 1}]\times \Spec\ k[a_2,\cdots].\]
Note that $\Aut^+(D)=\Spec k[a_1^{\pm 1},a_2,\cdots]$ is a closed normal subgroup scheme that classifies those automorphisms of $R[[t]]$ that reduce to the identity map mod $t$. Then we have
\begin{equation}\label{decomp1}
\Aut(D)\cong \Aut^+(D)\rtimes \hat{\bG}_a,
\end{equation}
where $\hat{\bG}_a=\on{Spf} k[[a_0]]$ is the formal additive group. In addition, $\Aut^{++}(D)=\Spec\ k[a_2,\cdots]$ is the pro-unipotent radical of $\Aut^+(D)$. Write $\bG_m^{\on{rot}}=\Spec k[a_1^{\pm 1}]$, usually called the rotation torus. Then
\begin{equation}\label{decomp2}
\Aut^+(D)=\Aut^{++}(D)\rtimes \bG_m^{\on{rot}}.
\end{equation}
\begin{rmk}
If $F$ is a general local field without a preferred choice of the uniformizer, then there is no canonical semidirect product decomposition \eqref{decomp1} and \eqref{decomp2}. But the corresponding short exact sequences of group ind-schemes are still well-defined.
\end{rmk}

Now, if $\underline G$ is a group scheme over $k[[t]]$ equipped with an action of $\Aut(D)$ that lifts the canonical action on $k[[t]]$ (e.g. $\underline G=G\otimes_k k[[t]]$ is constant), then is the an action
\begin{equation}\label{auto of disc}
\Aut(D)\times \Gr_{\underline G}\to \Gr_{\underline G},\quad (g, (\mE,\beta))\mapsto (g^*\mE,g^*\beta).
\end{equation}
In particular, the rotation torus $\bG_m^{\on{rot}}$ acts on $\Gr_{\underline G}$.

\subsection{Beauville-Laszlo's theorem.}
There is another important moduli interpretation of $\Gr_{\underline G}$ due to Beauville-Laszlo.

Let $X$ be a reduced connected curve (but not necessarily smooth) over $k$ and let $x\in |X|$ be a smooth closed point. Let $\underline G$ be a smooth affine group scheme over $X$. 
Let $\Gr_{\underline G,x}$ be the presheaf defined as
\begin{equation}\label{aff glob}
\Gr_{\underline G,x}(R)=\left\{(\mE,\beta)\ \left|\ \begin{split}&\mE \mbox{ is a } \underline G\mbox{-torsor on } X_R , \mbox{ and }\\
& \beta:\mE|_{X_R^*}\cong \mE^0|_{X_R^*}\mbox{ is a trivialisation}.\end{split}\right.\right\}
\end{equation}
Let $\underline G_x=\underline G\otimes_X \mO_x$. There is a natural morphism 
$$\on{res}:\Gr_{\underline G,x}\to \Gr_{\underline G_x}$$ 
induced by the restriction of a $\underline G$-torsor on $X_R$ to $D_{x,R}$.

\begin{thm}\label{glob aff}
The above morphism is an isomorphism.
\end{thm}
Again, the actual proof of the theorem will be technical but the idea is simple, so we first give an informal account. Namely, if $\mE$ is a $\underline G_x$-torsor on $D_x$ with a trivialisation $\beta$ on $D_x^*$, then one can glue $\mE$ and the trivial $\underline G|_{X^*}$-torsor $\mE^0$ on $X^*$ by $\beta$. This construction then induces a map inverse to $\on{res}$. 
To turn this idea into a rigorous proof, we need to first discuss Beauville-Laszlo's descent theorem (cf. \cite{BL2}). 
\begin{thm}\label{descent lemma}
Let $p: A\to \tilde{A}$ be a homomorphism of commutative rings, and $f\in A$. Assume that $f$ (resp. $p(f)$) is not a zero divisor in $A$ (resp. $\tilde{A}$) and that the induced map $A/f\to \tilde{A}/p(f)$ is an isomorphism. Let $\mM_f(A)$ denote the category of $f$-torion free $A$-modules and $\mM_{p(f)}(\tilde{A})$ the similar category for $(\tilde{A},p(f))$. Denote by $\mM$ the category of triples $(M_1,M_2,\varphi)$ where $M_1$ is an $A_f$-module, $M_2\in \mM_{p(f)}(\tilde{A})$ and $\varphi$ is an isomorphism between $M_1\otimes_{A_f}\tilde{A}_{p(f)}$ and $M_2\otimes_{\tilde{A}}\tilde{A}_{p(f)}$. Then
\begin{enumerate}
\item the functor $M\mapsto M\otimes_A\tilde{A}$ maps $\mM_f(A)$ to $\mM_{p(f)}(\tilde{A})$, so we have the functor $F : \mM_f(A)\to \mM$ that sends $M\in\mM_f(A)$
to $(M_f, M\otimes_A\tilde{A}, \varphi)$ where
$\varphi$ is the natural isomorphism between $M_f\otimes_{A_f}\tilde{A}_{p(f)}$ and $(M\otimes_A\tilde{A})_{p(f)}$.
\item $F : \mM_f(A)\to \mM$ is an equivalence of categories.
\item $M\in\mM_f(A)$ is a finite (resp. flat, resp. finite projective) $A$-module if and only if both $M_f$ and $M\otimes_A\tilde{A}$ have this property.
\end{enumerate}
\end{thm}
In \cite{BL2}, $\tilde{A}$ was taken to be $\hat A$, the completion of $A$ with respect to $f$. But the same arguments apply.
\begin{rmk}
The content of this theorem is descending modules to $A$ from its covering $(A_f, \tilde A)$. It
does not follow from the usual faithfully flat descent, since $A\to \tilde A$ is not assumed to be flat. For example, if $A$ is non-Noetherian, $A\to \hat A$ is typically non-flat. But even in the Noetherian case, the isomorphism $\varphi$ does not give the full descent datum, i.e. the isomorphism over $\hat R\otimes_R\hat R$ is not included.
\end{rmk}

Here is the consequence of the theorem.
\begin{cor}\label{global interpretation of loop}
The loop group $L\underline G_x$ represents the functor that associates every $k$-algebra $R$ the set of triples $(\mE,\al,\beta)$, where $\mE$ is a $\underline G$-torsor on $X_R$, $\al$ is a trivialisation of $\mE|_{D_{x,R}}$, and $\beta$ is a trivialisation of $\mE|_{X^*_R}$. 
\end{cor}
\begin{proof}We refer to \cite[Lemma 6.1]{PZ} for a detailed argument. Here is the idea. 

We can assume that $X=\Spec A$ is affine, and $x$ is defined by $t=0$. Let $R$ be a $k$-algebra and $g\in L\underline G_x(R)$. 
Let $H_1$ (resp. $H_2$) be the ring of regular functions of $\underline G\times_X D_{x,R}$ (resp. $\underline G\times_XX^*_R$), which is a flat $R[[t]]$-algebra (resp. $(A[t^{-1}]\otimes R)$-algebra). The element $g$ induces an algebra isomorphism  $\phi:H_1\otimes_{R[[t]]} R((t))\cong H_2 \otimes_{(A[t^{-1}]\otimes R)} R((t))$.
Applying Theorem \ref{descent lemma} to the triple $(H_1,H_2,\phi)$, one obtains a flat $(A\otimes R)$-algebra $H$. Then $\mE=\Spec H$ will be the $\underline G$-torsor on $X_R$, equipped with the canonical trivialisations $\al$ and $\beta$. 

Conversely, given $(\mE,\al,\beta)$, we obtain $g$ as $(\beta|_{D_{x,R}^*})(\al|_{D_{x,R}^*})^{-1}$.
\end{proof}

Now Theorem \ref{glob aff} follows by comparing Proposition \ref{I:quotient} and Corollary \ref{global interpretation of loop}. Alternatively, one can use a similar argument as above to directly construct a map inverse to $\on{res}$. We leave the details to readers.

\begin{rmk}Recall that the category of fpqc sheaves on $\Aff_k$ is equivalent to the category of fpqc sheaves on $\on{Sch}_k$. Under this equivalence, one can naturally extend the moduli description of $\Gr_{\underline G_x}$ as in Theorem \ref{glob aff} to all $k$-schemes. However, it would be awkward to try to extend the moduli description given in \eqref{loc aff} to all $k$-schemes since for a non-affine scheme $S$, $S\hat{\times} \Spec k((t))$ does not really make sense.
\end{rmk}

\subsection{The determinant line bundle.}\label{determinant line}
Our next topic is to construct a very ample line bundle on $\Gr_{\underline G}$. Recall from the proof of Theorem \ref{rep}, there is the following commutative diagram
\[\begin{CD}
\Gr^{(N)}@>>>\Gr^{(N+1)}\\
@VVV@VVV\\
\Gr(2nN)@>>>\Gr(2n(N+1))
\end{CD}\]
On the usual Grassmannian variety $\Gr(m)$, there is the determinant line bundle $\mL$ that associates a subspace $L\subset k^m$ the top exterior power of the quotient space $k^m/L$. This is in fact the ample generator of the Picard group (of each connected component) of $\Gr(m)$. These line bundles are compatible under the embedding $\Gr(m)\subset \Gr(m+1)$, and therefore defines an ample line bundle $\mL_{\det}$ on $\Gr_{\GL_n}$. We prefer to give a more canonical construction.

We start with an important notion.

\begin{dfn} A topological vector space is called \emph{linearly
    compact} if it is the topological dual of a discrete vector space
  (i.e., the usual dual with a basis of open neighbourhoods of $0$
  consisting of annihilators of finite dimensional subspaces). A
  topological space is called \emph{linearly locally compact}, or {\em
    Tate}, if it admits a basis of neighbourhoods of $0$ consisting of
  linearly compact subspaces. A \emph{lattice} in a Tate vector space
  $V$ is a linearly compact open subspace of $V$. 
\end{dfn}

Any two lattices $L_1, L_2$ in a Tate vector space are {\em
commensurable} with each other; that is, the quotients $L_1/(L_1
\cap L_2)$ and $L_2/(L_1 \cap L_2)$ are finite-dimensional.

\begin{rmk}\label{exact str} 
Tate vector spaces form a category, whose Hom's are the continuous
linear maps. This is an exact category in the sense of Quillen, where admissible monomorphisms are closed embeddings, and admissible epimorphisms are open surjective maps.

There is also a categorical approach to Tate vector spaces without invoking the topology. Recall that for an exact category $\mC$, Beilinson constructed the category $\underleftrightarrow{\lim} \mC$ of ``locally compact objects'' as a full subcategory of pro-ind-objects of $\mC$ (\cite[Appendix]{Be}). If we regard the category $\on{Vect}_k$ of finite dimensional vector spaces over $k$ as an exact category, 
then the category Tate vector spaces is exactly Beilinson's category $\underleftrightarrow{\lim}\on{Vect}_k$.  This approach has the advantage of avoiding using the topology, and can be generalised to other situations.
\end{rmk}

Let $V$ be a Tate vector space over $k$. Then for a $k$-algebra $R$, we define a topological $R$-module
\[V\hat\otimes_k R= \underleftarrow\lim_L (V/L)\otimes_kR,\]
where the inverse limit runs over the set of lattices in $V$. We can define a presheaf of groups $\GL(V)$ whose $R$-points are the continuous $R$-linear automorphisms of $V\hat\otimes_k R$.

\begin{dfn}An $R$-family of lattices in $V$ is an open bounded submodule of $\Lambda\subset V\hat\otimes_k R$ such that $V\hat\otimes_k R/\Lambda$ is a projective $R$-module.
\end{dfn}

\begin{dfn}
The Sato Grassmannian $\Gr(V)$ is the presheaf that assigns every $k$-algebra $R$ the set of $R$-families of lattices in $V$.
\end{dfn}

Exactly as Theorem \ref{rep},
\begin{thm}
The Sato Grassmannian $\Gr(V)$ is represented by an ind-projective scheme.
\end{thm}
In fact, by similar (and simpler) arguments, there is a natural isomorphism
$$\Gr(V)\cong \underrightarrow\lim_{L_1\subset L_2}\Gr(L_2/L_1),$$ 
where $L_1\subset L_2$ are two lattices so $L_2/L_1$ is a finite dimensional $k$-vector space, and $\Gr(L_2/L_1)$ denotes the usual Grassmannian variety classifying finite dimensional subspaces of $L_2/L_1$.

 Now we introduce the determinant line bundle on $\Gr(V)\times \Gr(V)$ (see \S~\ref{Intro:sp} for the definition of line bundle on a space).
First, recall that for a finite projective $R$-module $M$ its top exterior power, denoted by $\det(M)$, is an invertible $R$-module.
For $\Lambda_1,\La_2\in\Gr(V)(R)$, define the ``relative determinant line'' for $(\La_1,\La_2)$ as
\[\det(\La_1|\La_2)=\det(L\hat\otimes R/\La_1)\otimes \det(L\hat\otimes R/\La_2)^{-1},\]
where $L$ is some lattice in $V$ such that both $(L\hat\otimes R)/\La_1$ and $(L\hat\otimes R)/\La_2$ are projective $R$-modules. This is independent of the choice of $L$ up to a canonical isomorphism. By varying $R$, we obtain a line bundle $\mL_{\det}$ on $\Gr(V)\times\Gr(V)$. The following lemma, although elementary, summarises important properties of the relative determinant lines.
\begin{lem}~
\begin{enumerate} 
\item For any $g\in  \GL(V)(R)$, and $\La_1,\La_2\in\Gr(V)(R)$, there is a canonical isomorphism $$\det(g\La_1|g\La_2)\cong\det(\La_1|\La_2),$$ such that  for $g,g'$, the isomorphism $$\det(gg'\La_1|gg'\La_2)\cong\det(g'\La_1|g'\La_2)\cong\det(\La_1|\La_2)$$ coincides with $\det(gg'\La_1|gg'\La_2)\cong\det(\La_1|\La_2)$. In other words, the diagonal action of $\GL(V)$ on $\Gr(V)\times\Gr(V)$ lifts to an action on $\mL_{\det}$;

\item For any $\La_1,\La_2,\La_3$, there is a canonical isomorphism \[\ga_{123}:\det(
\La_1|\La_2)\otimes\det(\La_2|\La_3)\cong\det(\La_1|\La_3)\] such that for any $\La_1,\La_2,\La_3,\La_4$, $\ga_{134}\ga_{123}=\ga_{124}\ga_{234}$.
In other words, let $\mL_{\det}^\times$ denote the total space of $\mL_{\det}$. Then we have a groupoid 
$$\mL_{\det}^\times\rightrightarrows\Gr(V).$$
\end{enumerate}
\end{lem}
\begin{rmk}
The groupoid defines a $\bG_m$-gerbe 
\[\mD_V=[\Gr(V)/ \mL_{\det}^\times],\] 
which is usually called the determinant gerbe of $V$ (the terminology was first introduced by Kapranov).
\end{rmk} 

\begin{rmk}
The above construction generalises in families. Namely, for a commutative ring $R$, there is the category $\on{Tate}_R$ of Tate $R$-modules, defined as the idempotent completion of $\underleftrightarrow{\lim}\on{Proj}_R$, where $\on{Proj}_R$ is the category of finite projective $R$-modules  (see \cite{Dr} for a topological approach). For a Tate $R$-module $M$, there is the corresponding Sato Grassmannian $\Gr(M)\to \Spec R$ (see \emph{loc. cit.}), which is represented by an ind-algebraic space. In addition, there is the determinant gerbe $\mD_M$ over $\Spec R$. 
\end{rmk}

\begin{rmk}\label{graded line}
In fact, the above construction should be upgraded to a graded version. Namely, for a finite projective $R$-module, we can define $\det(M)$ as a graded invertible $R$-module by assigning the degree of $\det(M)$ to be the rank of $M$ (which is a locally constant function on $\Spec R$). Then the above defined $\det(\La_1\mid\La_2)$ becomes a graded line bundle on $\Spec R$.
Remembering the grading is crucial for the factorisation structure on determinant lines (see Remark \ref{graded factorization} and Remark \ref{det factorization2}), which in turn is important for applications to symbols and reciprocity laws (cf. \cite{BBE, OZ1, OZ2}). 
\end{rmk}

Let $V=k((t))^n$, endowed with the usual $t$-adic topology. Then it is a Tate vector space, containing $\Lambda_0=k[[t]]^n$ as a lattice. It follows by Lemma \ref{projectivity} that  there is a canonical closed embedding $i:\Gr_{\GL_{n}}\to\Gr(V)$.
Therefore, via pullback we obtain a line bundle $\mL_{\det}|_{\{\La_0\}\times\Gr_{\GL_n}}$, still denoted by $\mL_{\det}$. Note that this is a very ample line bundle on $\Gr_{\GL_n}$.

In general, let $\rho: \underline G\to \GL_n$ be a faithful $\mO$-linear representation as in Proposition \ref{lin rep}. Then
by Proposition \ref{gen G}, it induces a locally closed embedding
$f_\rho: \Gr_{\underline G}\to \Gr_{\GL_n}$ and the pullback of $\mL_{\det}$ defines a very ample line bundle on $\Gr_{\underline G}$.

\begin{rmk}\label{Pfaffians}
We assume $\on{char} k\neq 2$  (to remove this assumption see \cite[\S\ 4.2.16]{BD}). 
Let $V$ be a Tate vector space $k$. A continuous symmetric bilinear form $B$ of $V$ is called non-degenerate if the induced map $V\to V^*$ is an isomorphism, where $V^*$ denotes the topological dual of $V$. In this case, given a lattice $\Lambda$, one can define its orthogonal complement $\Lambda^\perp$. A lattice $\Lambda$ is called a Lagrangian if $\Lambda=\Lambda^\perp$. 
The orthogonal Sato Grassmannian $\OGr(V,B)\subset \Gr(V)$ is the closed subspace classifying Lagrangian lattices. As constructed in \cite[\S\ 4]{BD}, the restriction of determinant line bundle $\mL_{\det}$ to $\OGr(V,B)\times \OGr(V,B)$ admits a canonical square root $\mL_{\on{Pf}}$, called the Pfaffian line bundle.
\end{rmk}

\subsection{Affine Grassmannians over complex number.}\label{top loop group}
In this subsection, we discuss affine Grassmannians over complex numbers. So we assume that $k=\bC$. We assume $\underline G=G\otimes_k \mO$ and write $\Gr$ for $\Gr_{\underline G}$. Then we can regard $\Gr$ as an infinite dimensional complex analytic space (and in particular a CW complex).

Let $K\subset G$ be a maximal compact subgroup and let $S^1$ be the unit circle. Let $\Omega K$ be the
space of polynomial maps $(S^1,1)\to (K,e_K)$. More precisely, let
us parameterise $S^1$ by $e^{i\phi}, \phi\in\bR$, and let
$K\subset \on{SO}(n,\bR)$ be an embedding. Then $\Omega K$ is the
space of maps from $(S^1,1)\to (K,e_K)$ such that when composed
with $K\subset \on{SO}(n,\bR)$, the matrix entries of the maps are given by
Laurent polynomials of $e^{i\phi}$. We regard $\Omega K$ as a topological group.
We have the following theorem.
\begin{thm}\label{top triv}~
\begin{enumerate}
\item The canonical map $\Omega K\to LG$ induces a homeomorphism $\Omega K\cong \Gr$. Equivalently, we have the factorisation $LG=\Omega K\cdot L^+G$.

\item There is an isomorphism $\Omega K\times\Omega K\cong \Gr\tilde\times\Gr$ making the following commutative diagram
\[\begin{CD}
\Omega K\times \Omega K@>m>>\Omega K\\
@VVV@VVV\\
\Gr\tilde\times\Gr @>m>>\Gr,
\end{CD}\]
where the top row is the usual multiplication of the loop group $\Omega K$, and the bottom row is the convolution map \eqref{conv m0}.
\end{enumerate}
\end{thm}
(1) is due to Pressley-Segal (cf. \cite[\S\ 8.3]{PS}). Strictly speaking, only $G=\GL_n$ was proved in \emph{loc. cit.} However, the general case can be deduced from a much more general theorem of Nadler (cf. \cite[\S\ 4]{Na}).
Namely, $\Omega K$ can be regarded as a ``real form'' of the infinite dimensional complex group $G(\bC[t,t^{-1}])$.  Nadler gave a parameterisation of the orbits on $\Gr$ under the action of some real forms of $G(\bC[t,t^{-1}])$. When specialised to $\Omega K$, his theorem implies (1).
Part (2) follows from (1) directly. 
\begin{rmk}
In particular, we see that $\Gr$ has the homotopy type as an $H$-space. In addition, the $L^+G$-torsor $LG\to \Gr$ is trivial, and $\Gr\tilde\times\Gr\cong \Gr\times\Gr$ as topologically spaces  (do not confuse this isomorphism with \eqref{I: isom}). As we shall see in \S\ \ref{Lecture V}, for the affine Grassmannian of $G$ over an arbitrary field $k$, all these facts have avatars at the level of cohomology (in fact at the level of \'etale homotopy type).
\end{rmk}

\subsection{Affine Grassmannians for $p$-adic groups.}\label{p-adic}
We briefly discuss the affine Grassmannian for a group over a $p$-adic field $F$, i.e. a finite extension of $\bQ_p$. 
 As in the case $F=k((t))$, the first step is to make sense a family of lattices of $V=\bQ_p^n$ parameterised by $\Spec R$ for an $\bF_p$-algebra $R$. In particular, one needs to make sense of the disc $D_R$. Although the fiber product $\Spec \bZ_p``\times''\Spec R$ does not literally make sense, it is widely-known that such $D_R$ should be defined as 
$$D_R:=\Spec W(R),\quad D^*_R:=\Spec W(R)[1/p],$$ 
where $W(R)$ is the ring of ($p$-typical) Witt vectors of $R$. This is a ring whose elements are sequences $(r_0,r_1,r_2,\ldots)\in R^{\bN}$, with the addition and the multiplication given by certain (complicated) polynomials. The projection to the first component $W(R)\to R, \ (r_0,r_1,\ldots)\mapsto r_0$ is a surjective ring homomorphism, with a multiplicative (but \emph{non-additive}) section $R\to W(R), \ r\mapsto [r]=(r,0,0,\ldots)$, called the Teichm\"uller lifting of $r$. If $R$ is a perfect ring, every element in $W(R)$ can be uniquely written as $\sum_{i\geq 0} [a_i]p^i$ so $W(R)$ can be regarded as a ``power series ring in variable $p$ and with coefficients in $R$''.
For example, $W(\bF_p)=\bZ_p$.  We refer to \cite[\S\ II.5]{Se2} for a general introduction.

In general, let $F$ be a $p$-adic field, with $\mO$ its ring of integers, and $k$ its residue field. Then one can similarly define $D_{F,R}$ and $D^*_{F,R}$ using the so-called ramified Witt vectors introduced by Drinfeld. If $R$ is perfect, they are simply given by
\[D_{F,R}= \Spec W(R)\otimes_{W(k)}\mO,\quad D^*_{F,R}= \Spec W(R)\otimes_{W(k)}F,\]
where $W(k)\to \mO$ is the unique ring homomorphism that induces the identity map of the residue fields.
Now let $\underline G$ be a smooth affine group scheme over $\mO$. Then one can define a presheaf on $\on{Aff}_k$ exactly as \eqref{loc aff}, with one caveat. The ring $W(R)$ is pathological if $R$ is not perfect. E.g. $p$ could be a zero divisor of $W(R)$ if $R$ is non-reduced. On the other hand, note that
\begin{enumerate}
\item if $R$ is a perfect $k$-algebra, then $W(R)$ is well-behaved, e.g. every element in $W(R)$ has a ``power series expansion''.
\item The $R$-points of a scheme $X$ for perfect rings $R$ determine $X$ up to perfection\footnote{The category of perfect $k$-schemes is a full subcategory of the category of presheaves on the category of perfect $k$-algebras, see \cite[Lemma A.10]{Z14}.}. 
\item The (\'{e}tale) topology of a scheme (e.g. the $\ell$-adic cohomology) does not change when passing to the perfection.
\end{enumerate}

Therefore, we restrict the naive moduli problem as in \eqref{loc aff} to the category of perfect $k$-algebras. This defines a presheaf on this category. The \'etale topology and the fpqc topology make sense on this category and one can do algebraic geometry in this setting (see \cite[Appendix A]{Z14} for an introduction).
Then the best question one can ask is: whether this functor is represented by an inductive limit of perfect $k$-schemes. 
The following difficult theorem,  proved in \cite{BS}, gives a positive answer to this question.
\begin{thm}The functor $\Gr_{\underline G}$ is represented by an ind-perfectly projective scheme, called the Witt vector affine Grassmannian.
\end{thm}
The idea is to construct an ample ``determinant'' line bundle on $\Gr_{\GL_n}$. The difficulty is that for an $R$-family lattice $\La\subset W(R)^n$, the quotient $W(R)^n/\La$ is \emph{not} an $R$-module so  the construction given in \S\ \ref{determinant line} does not work. In addition, there is no similar symmetry $\Aut(D)$ on $\Gr_{\underline G}$ and of course no analogue of Theorem \ref{glob aff}  in $p$-adic setting.

Previously, it was proved in \cite[\S~1]{Z14} that $\Gr_{\underline G}$ is represented by an ind-perfectly proper algebraic space.
This is a much easier theorem, which suffices for most applications to topological problems, such as the geometric Satake equivalence which will be discussed in \S\ \ref{Lecture V}.

\section{Lecture II: More on the geometry of affine Grassmannians}\label{Lecture II}
We study some basic  geometry of affine Grassmannians. Let $F=k((t))$ and $\mO=k[[t]]$. Unless otherwise noted, we assume that $k$ is algebraically closed and $\underline G=G\otimes_k \mO$ is a constant reductive group scheme and write $\Gr$ for $\Gr_{G}$ if no confusion will arise. We refer to \cite{PR} and \cite[\S\ 6.3]{Z11} for generalisations of some statements to non-constant groups, and to \cite{Z14} for some corresponding statements for $p$-adic groups.

\subsection{Schubert varieties.}\label{II:Sch var}
We use notations in \S\ \ref{Intro:group}. There is a canonical bijection
\begin{equation}\label{Cartan decomp2}
G(\mO)\backslash G(F)/G(\mO)\cong \xcoch(T)/W\cong \xcoch(T)^+,
\end{equation}
which can be concretely realised as follows. In the rest of this subsection, we fix embeddings $T\subset B\subset G$.
Let $\mu\in\xcoch(T)$ be a coweight, which defines a map $\mu:F^\times\to T(F)\subset G(F)$. We denote by $t^\mu=\mu(t)$ the image of the uniformizer $t$ in $G(F)$, regarded as a $k$-point of $LG$. It depends on the choice of $t$. Its image in $\Gr$ is still denoted by $t^\mu$, which is independent of the choice of the uniformizer. Then we have the Cartan decomposition (cf. \cite[\S\ 3.3.3]{Ti})
\begin{equation}\label{Cartan decomp}
G(F)=\bigsqcup_{\mu\in\xcoch(T)^+} G(\mO)t^\mu G(\mO).
\end{equation}
Note that the double coset $ G(\mO)t^\mu G(\mO)$ does not depend on the uniformizer $t$ nor the embedding $T\subset G$. Therefore \eqref{Cartan decomp} induces \eqref{Cartan decomp2}.

Now let $\mE_1$ and $\mE_2$ be two $G$-torsors over $D$, and let 
\[\beta: \mE_1|_{D^*}\cong \mE_2|_{D^*}\]
be an isomorphism over $D^*$. By choosing isomorphisms $\phi_1:\mE_1\cong \mE^0$ and $\phi_2:\mE_2\cong \mE^0$, one obtains an automorphism of the trivial $G$-torsor $\phi_2\beta\phi_1^{-1}\in \Aut(\mE^0|_{D^*})$ and therefore an element in $G(F)$. Different choices of $\phi_1$ and $\phi_2$ will modify this element by left and right multiplication by elements from $G(\mO)$. Therefore,
$\beta$ gives rise to a well-defined element
\[\inv(\beta)\in G(\mO)\backslash G(F)/G(\mO)\cong \xcoch(T)^+,\]
sometimes called the relative position of $\beta$.
\begin{rmk}\label{non-algebraically closed}
Let $K$ be a not necessarily algebraically closed field containing $k$, $\mE_1, \mE_2$ be two $G$-torsors over $D_K$, and $\beta: \mE_1|_{D^*_K}\cong \mE_2|_{D^*_K}$ be an isomorphism. Then one can still define $\inv(\beta)\in \xcoch(T)^+$ as follows. After passing to an algebraic closure $\bar K/K$, the above construction gives 
$$\inv(\beta_{\bar K})\in G(\bar K[[t]])\backslash G(\bar K((t))/G(\bar K[[t]])\cong G(\mO)\backslash G(F)/G(\mO),$$
which is independent of the choice of the algebraic closure $\bar K$. Therefore, $\inv(\beta)=\inv(\beta_{\bar K})$ is well-defined.
\end{rmk}

Now let $\mE_1$ and $\mE_2$ be two $G$-torsors over $D_R$, and 
$\beta: \mE_1|_{D_R^*}\cong \mE_2|_{D_R^*}$.
Then for each $x\in \Spec R$, by base change we obtain 
$$(\mE_1|_{D_{k(x)}},\mE_2|_{D_{k(x)}},\beta_{k(x)}: \mE_1|_{D_{k(x)}^*}\simeq \mE_2|_{D_{k(x)}^*}).$$ Let $\inv_x(\beta):=\inv(\beta_{k(x)})$.

\begin{prop}\label{HP}
Let $X=\Spec R$. Given $\mu\in\xcoch(T)^+$,
the set of points 
$$X_{\leq \mu}:=\left\{x\in X\mid \inv_x(\beta)\leq \mu\right\}$$
is Zariski closed in $X$.
\end{prop}
\begin{proof}We sketch the proof. Let $\rho_\chi: G\to \GL(V_\chi)$ be a finite dimensional highest weight representation of $G$ of highest weight $\chi$ (e.g. $V_\chi$ is the Weyl module of $G$, of highest weight $\chi$), and let $V_{\chi,\mE_i}=\mE_i\times^GV_{\chi}, i=1,2$ denote the induced vector bundles. Then by definition the subset in question is contained in the set 
$$X_{V_\chi,\leq \mu}:=\left\{x\in X\mid\rho_\chi(\beta_x)(V_{\chi,\mE_1})\subset t^{-(\chi,\mu)}(V_{\chi,\mE_2})\right\},$$
which is easily seen to be closed in $X$. Indeed, working Zariski locally on $X$ we can assume that both $V_{\chi,\mE_i}, i=1,2$ are trivial (by the argument as in Lemma \ref{et loc triv}). So $\rho_\chi(\beta)$ is given by a square matrix  with entries $a_{ij}=\sum_h a_{ij,h}t^h$  in $R((t))$. Then $X_{V_\chi,\leq \mu}$ is defined by the equations $a_{ij,h}=0$ for $h<-(\chi,\mu)$.

If the derived group $G_\der$ of $G$ is simply-connected, then one has
$$X_{\leq \mu}=\bigcap_{V_\chi} X_{V_\chi,\leq \mu},$$ where the intersection is taken over over all finite dimensional highest weight representations of $G$. Therefore, $X_{\leq \mu}$ is closed.

In general, one can always find a central extension $\tilde G$ of $G$ by a torus $D$ such that the derived group of $\tilde G$ is simply-connected (the so-called $z$-extension, cf. \cite[Proposition 3.1]{MS}). Using Lemma \ref{et loc triv}, \'etale locally on $X$ we can lift $\beta$ to a map $\tilde \beta$ of $\tilde G$-torsors and lift $\mu$ to a coweight $\tilde \mu$ of $\tilde G$. Then we reduce to the previous case.
\end{proof}

We define the (spherical) Schubert variety $\Gr_{\leq \mu}$ as the closed subset
$$\Gr_{\leq \mu}=\left\{(\mE,\beta)\in \Gr\mid \inv(\beta)\leq \mu\right\},$$
of $\Gr_G$, endowed with the reduced scheme structure. By Proposition \ref{HP}, 
\[\Gr_\mu:=\left\{(\mE,\beta)\in \Gr\mid \inv(\beta)=\mu\right\}=\Gr_{\leq \mu}\setminus \cup_{\la<\mu}\Gr_{\leq \la}\]
is an open subset of $\Gr_{\leq \mu}$, called a Schubert cell.

\begin{prop}\label{geom of Sch}~
\begin{enumerate}
\item $\Gr_{\mu}$ forms a single $L^+G$-orbit and is a smooth quasi-projective variety of dimension $(2\rho,\mu)$.

\item $\Gr_{\leq \mu}$ is the Zariski closure of $\Gr_\mu$, and so is a projective variety.
\end{enumerate}
\end{prop}
\begin{proof}
For (1), note that the stabiliser of $t^\mu$ for the action of $L^+G$ is $L^+G\cap t^\mu L^+G t^{-\mu}$. The induced map
$$L^+G/(L^+G\cap t^\mu L^+G t^{-\mu}) \to LG/L^+G,\quad g \mapsto gt^\mu$$ then is a locally closed embedding. By the Cartan decomposition \eqref{Cartan decomp}, the image is exactly $\Gr_\mu$. Therefore, $\Gr_\mu$ is smooth and the tangent space of $\Gr_\mu$ at $t^\mu$ is identified with $\frakg(\mO)/\frakg(\mO)\cap \Ad_{t^\mu}\frakg(\mO)$. But it is easy to see that
\[\frakg(\mO)/\frakg(\mO)\cap \Ad_{t^\mu}\frakg(\mO)=\bigoplus_{(\al,\mu)\geq 0} \frakg_\al(\mO)/t^{(\al,\mu)}\frakg_\al(\mO),\]
where the sum is taken over all positive roots $\al$ of $G$ and $\frakg_\al$ denotes the corresponding root space. Therefore, by a simple calculation $\dim \Gr_\mu=(2\rho,\mu)$.

For  (2), note that if $\la\leq \mu$, then there exists a positive coroot $\al$ such that $\mu-\al$ is dominant and $\la\leq \mu-\al\leq \mu$ (e.g. \cite[Lemma 2.3]{R3}). Therefore, it is enough to show that $t^{\mu-\al}$ is contained in the Zariski closure of $\Gr_\mu$. To prove this, it is enough to construct a curve $C$ in $\Gr_{\leq\mu}$ such that $t^{\mu-\al}\in C$ and $C-\{t^{\mu-\al}\}\subset \Gr_\mu$. This reduces to an $\SL_2$-calculation as follows. 

First, for every integer $m$, let $t^{\la_m}:=\begin{pmatrix} t^m & 0\\ 0 &1\end{pmatrix}$, regarded as an element in $\on{PGL}_2(F)$. Let 
$$K_m=\Ad_{t^{\la_m}}(L^+\SL_2)\subset L\SL_2.$$  Note that the Lie algebra of $K_m$ as an $\mO$-module is spanned by $\{t^me, h, t^{-m}f\}$, where $\{e,h,f\}$ is the standard $\fraks\frakl_2$-triple. Also note that 
$$\sigma_m:=\begin{pmatrix}0&-t^m\\t^{-m}&0\end{pmatrix}=\begin{pmatrix} t^m & 0\\ 0 &1\end{pmatrix}\begin{pmatrix}0&-1\\1&0\end{pmatrix}\begin{pmatrix} t^{-m} & 0\\ 0 &1\end{pmatrix}\in K_m.$$
Let $L^{>0}\SL_2$ be the kernel of the evaluation map $\on{ev}:L^+\SL_2\to \SL_2$ (cf. Example \ref{hyperspacial and Iwahori2}), and let $K_m^{(1)}=\Ad_{t^{\la_m}}(L^{>0}\SL_2)$. Then $K_m/K_m^{(1)}\cong \SL_2$.

Now let $i_\al: \SL_2\to G$ be the canonical homomorphism corresponding to $\al$. Let $m=(\mu,\al)-1$ and consider the orbit 
\begin{equation}\label{curve C}
C_{\mu,\al}:=Li_\al(K_{m})t^\mu.
\end{equation} 
Note that $K_{m}^{(1)}\subset L^+G\cap t^\mu L^+G t^{-\mu}$ so $C_{\mu,\al}$ is a homogenous space under $K_{m}/K_m^{(1)}=\SL_2$. It is easy to see that $C_{\mu,\al}\cong \bP^1$, and that $(L^+G\cap Li_\al(K_m))t^\mu\cong \bA^1\subset \bP^1$. In addition, $\bP^1\setminus\bA^1=i_\al(\sigma_m)t^\mu$.
From the famous identity
\[\begin{pmatrix}0&-1\\1&0\end{pmatrix}=\begin{pmatrix}1&-t^{-1}\\0&1\end{pmatrix}\begin{pmatrix}1&0\\ t&1\end{pmatrix}\begin{pmatrix}1&-t^{-1}\\0&1\end{pmatrix}\begin{pmatrix}t^{-1}&0\\0&t\end{pmatrix},\]
we see that $i_{\al}(\sigma_m)t^\mu=(t^{\mu-\al} \mod L^+G)$. Therefore, $C_{\mu,\al}$ is the desired curve.
\end{proof}

\begin{rmk}We assume (for simplicity) that $G_\der$ is simply connected and $\cha\ k=0$. Then one can define a closed subscheme of $\Gr'_{\leq \mu}\subset \Gr$ whose $R$-points classify those $(\mE,\beta)$ such that for every irreducible highest weight representation $V_\chi$, $\beta(V_{\chi,\mE})\subset t^{-(\chi,\mu)}(V_{\chi}\otimes R[[t]])$. From the proof of Proposition \ref{HP}, we know that $\Gr_{\leq \mu}$ is the reduced subscheme of $\Gr'_{\leq \mu}$. But it is an open question to whether $\Gr'_{\leq \mu}=\Gr_{\leq \mu}$. This is closely related to \cite[Conjecture 2.14, Conjecture 2.20]{KWWY}. 
\end{rmk}

\begin{ex}\label{momega1}
Let $G=\GL_n$ and $\mu=(r,0,\ldots,0)$ (under the standard identification of $\xcoch(T)\cong \bZ^n$). Then using the lattice description of $\Gr_{\GL_n}$ as in \S\ \ref{aff GLn},
$$\Gr_{\leq \mu}(R)=\left\{ \Lambda\subset \Lambda_0:=R[[t]]^n\mid \on{rk} \Lambda_0/\Lambda=r\right\}.$$
Now assume that $r=n$.
Let $\mN_n$ denote the variety of $n\times n$ nilpotent matrices. There is a natural map
\begin{equation}\label{nil to Gr}
\mN_n\to \Gr_{\leq \mu},\quad A\mapsto \La=(t-A)\La_0.
\end{equation}
This is well-defined since $\det(t-A)=t^n$. We will show in Lemma \ref{Lus obs} that this is an open embedding so in particular, $\Gr_{\leq \mu}$ gives a compactification of $\mN_n$.
At the level of topological spaces, this was first observed by Lusztig (cf. \cite{Lu0}). 
Note that giving a dominant coweight $\la\leq \mu=(n,0,\ldots,0)$ is equivalent to giving a partition of $n$. So each $\la\leq \mu$ gives a nilpotent orbit $\mO_\la$ in $\mN_n$. We leave it as an exercise to show that \eqref{nil to Gr} maps $\mO_\la$ to $\Gr_\la$.
\end{ex}

We continue the general theory. There is a map 
\begin{equation}\label{parity map}
p:\xcoch(T)\to \bZ/2, \mu\mapsto (-1)^{(2\rho,\mu)}
\end{equation}
which factors through $\xcoch(T)\to \pi_1(G)\to \bZ/2$ and therefore induces a map $p:\pi_0(\Gr_G)\to \bZ/2$
by Theorem \ref{components}. 
\begin{cor}\label{parity}
The Schubert cell $\Gr_{\mu}$ is in the even (resp. odd) components, i.e. $p(\Gr_{\mu})=1$ (resp. $p(\Gr_{\mu})=-1$) if and only if $\dim \Gr_{\mu}$ is even (resp. odd).
\end{cor}

In the sequel, let $(LG)_{\leq \mu}$ (resp. $(LG)_\mu$) denote the preimage of $\Gr_{\leq \mu}$ (resp. $\Gr_\mu$) in $LG$. Note that $(LG)_{\leq \mu}$ and $(LG)_{\mu}$ are schemes.

For a coweight $\mu$, let $P_\mu$ denote the parabolic subgroup of $G$ corresponding to $\mu$, i.e. the group generated by the root subgroups $U_\al$ of $G$ for those roots $\al$ satisfying $\langle\al,\mu\rangle\leq 0$. Let us denote by $g\mapsto \bar{g}:=\on{ev}(g)$ for the evaluation map (see Example \ref{hyperspacial and Iwahori2}).
Then there is a natural projection
\[p_\mu: \Gr_\mu\cong L^+G/(L^+G\cap t^\mu L^+G t^{-\mu})\to G/P_\mu,\quad (gt^\mu \mod L^+G)\mapsto (\bar{g}\mod P_\mu).\]
The fibers are isomorphic to affine spaces. 
On the other hand, there is also a closed embedding
\[i_\mu: G/P_\mu\cong G\cdot t^\mu\subset \Gr_\mu\]
by regarding $G$ as a subgroup of $L^+G$ of ``constant loops''. Note that  
the action of  $\Aut^+(D)$ on $\Gr_G$ preserves $\Gr_\mu$\footnote{However, the action of $\Aut(D)$ on $\Gr_G$ does not preserve Schubert varieties as $gt^\mu\not\in \Gr_{\leq \mu}$ for $g\in \Aut(D)$ sending $t\to a_0+t$, where $a_0\in R$ is nilpotent.} and the fixed point subscheme under the rotation torus $\bG_m^{\on{rot}}$ is $G\cdot t^\mu$.

\begin{ex}We assume that $G=\GL_n$ and let $\La_0=k[[t]]^n$. Let $\mu=(m_1,\ldots,m_n)\in \bZ^n$ be a dominant coweight (so $m_i\geq m_{i+1}$). Then the map $p_\mu: \Gr_\mu\to G/P_\mu$ sends a lattice $\Lambda\in Gr_\mu$ to the following decreasing filtration $V=\La_0/t\La_0\cong k^n$. 
\[\on{Fil}^iV=t^{-m_i}\La\cap \La_0/t^{-m_i}\La\cap t\La_0.\]
Conversely, given a decreasing filtration $\on{Fil}^\bullet$ on $V$, the map $i_\mu$ sends it to the lattice
\[\Lambda = \sum\on{Fil}^iV\otimes  t^{m_i} \mO\subset V\otimes_kF\cong F^n.\]
\end{ex}

We give more explicit descriptions of $\Gr_{\leq \mu}$ for small $\mu$.

Recall that a dominant coweight $\mu$ is called minuscule if $\mu\neq 0$ and for every positive root $\al$, $(\al,\mu)\leq 1$. Note that minuscule coweights are minimal elements in $\xcoch(T)^+$ under the partial order $\leq$. Therefore,
\begin{lem}\label{minusch}
If $\mu$ is minuscule, then $\Gr_{\leq\mu}=\Gr_\mu\cong G/P_\mu$.
\end{lem}
So minuscule Schubert varieties in $\Gr_G$ are exactly those partial flag varieties classifying parabolic subgroups of $G$ corresponding to minuscule coweights. Examples include the usual Grassmannians $\Gr(r,n)$, smooth quadrics in projective spaces, etc.

Recall that a dominant coweight $\mu$ is called quasi-minuscule if $\mu\neq 0$, $\mu$ is not minuscule, and for every positive root $\al$, $(\al,\mu)\leq 2$. For simplicity, we assume that $G_\der$ is simple.
Then $\mu$ is the short dominant coroot, denoted by $\theta$, and is the minimal element in $\xcoch(T)^+-\{0\}$ under the partial order $\leq$.  The root corresponding to $\theta$ is the highest root $\theta^\vee$. Let $\mL_\theta=G\times^{P_\theta} k_{\theta^\vee}$ be the very ample line bundle on $G/P_\theta$, where $P_\theta$ acts on the $1$-dimension space $k_{\theta^\vee}$ by the character $\theta^\vee$. 
Note that $0\in\xcoch(T)^+$ is the unique element that is strictly less than $\theta$. Therefore, $\Gr_{\leq \theta}=\Gr_\theta\sqcup \Gr_0$.
\begin{lem}\label{quasiminusch}
Assume (for simplicity) that $G_\der$ is simple. If $\mu=\theta$ is quasi-minuscule, then via the projection $p_\theta$, $\Gr_\theta$ is isomorphic to the total space of the line bundle $\mL_\theta$,
and $\Gr_{\leq \theta}$ is isomorphic to the projective cone of the projective embedding of $G/P_\theta$ by $\mL_\theta$. Blowing-up at the origin $\widetilde{\Gr}_{\leq \theta}= \bP(\mL_{\theta}\oplus \mO)$ gives a resolution of singularities . 
\end{lem}
\begin{proof}See \cite[\S\ 7]{NP}.
\end{proof}
\begin{rmk}
In fact, $\widetilde{\Gr}_{\leq \theta}$ is a kind of Demazure resolution of $\Gr_{\leq \theta}$. See \cite[Lemma 2.12]{Z14}.
\end{rmk}
\begin{ex}If $G=\SL_2$, then $\theta=\al$ is the positive coroot. In this case $\Gr_{\leq \theta}$ is isomorphic to the projective cone of the quadratic curve $\{(x,y,z)\mid x^2=yz\}\subset \bP^2$. 
\end{ex}

Given $\mu_1,\mu_2\in \xcoch^+$ of $G$, one can define the twisted product of $\Gr_{\leq \mu_1}$ and $\Gr_{\leq \mu_2}$ using the $L^+G$-torsor $(LG)_{\leq \mu_1}\to \Gr_{\leq \mu_1}$. Alternatively, one can define it as
\[\Gr_{\leq \mu_1}\tilde{\times}\Gr_{\leq \mu_2}=\left\{(\mE_1,\mE_2,\beta_1,\beta_2)\in\Gr\tilde\times\Gr\mid \inv(\beta_1)\leq \mu_1, \inv(\beta_2)\leq \mu_2\right\},\]
which is closed in $\Gr\tilde\times\Gr$ and therefore is representable. Similarly,  if $\mmu=(\mu_1,\ldots,\mu_n)$ is a sequence of dominant coweights of $G$, one can define
$$\Gr_{\mmu}:=\Gr_{\mu_1}\tilde\times\cdots\tilde\times\Gr_{\mu_n}\subset \Gr_{\leq\mmu}:=\Gr_{\leq \mu_1}\tilde\times\cdots\tilde\times\Gr_{\leq \mu_n}\subset\Gr\tilde\times\cdots\tilde\times\Gr.$$ 
Let $|\mmu|=\sum \mu_i$, then the convolution map \eqref{conv m} induces
\begin{equation}\label{conv m2}
m: \Gr_{\leq \mmu}\to \Gr_{\leq |\mmu|}, \quad (\mE_\bullet,\beta_\bullet)\mapsto (\mE_n,\beta_1\cdots\beta_n).
\end{equation}
This map is analogous to Demazure resolutions (except that $\Gr_{\leq \mmu}$ is not smooth in general). The geometry of these maps is very rich, as we shall see in \S~\ref{Lecture V}. Here we give one example.

\begin{ex}\label{conv v.s. springer}
We use notations from Example \ref{momega1}.  Now, let $\omega_1=(1,0,\ldots,0)$. We have the convolution map $m:\Gr_{\leq(\omega_1,\ldots,\omega_1)}\to\Gr_{\leq \mu}$. On the other hand, there is the classical Springer resolution $p:\tilde \mN_r\to\mN_r$ (cf. \cite[\S~1]{Yu}). These maps fit into the following commutative diagram
\begin{equation}\label{sat and springer}
\begin{CD}
\Gr_{\leq(\omega_1,\ldots,\omega_1)}@<<< \Gr_{\leq (\omega_1,\ldots,\omega_1)}^{\square}@>>>\tilde \mN_r\\ 
 @VmVV@VVV@VVpV\\
 \Gr_{\leq \mu}@<\pi<<\Gr_{\leq \mu}^{\square}@>\phi>> \mN_r,
\end{CD}
\end{equation}
where $\Gr_{\leq \mu}^{\square}(R)$ classify triples $(\mE,\beta,\epsilon)$ where $(\mE,\beta)\in\Gr_{\leq \mu}(R)$ and $\epsilon: \La_0/\La\cong R^r$ is an isomorphism of $R$-modules. Note that both squares in the diagram are Cartesian. Indeed, a point of $\Gr_{\leq(\omega_1,\ldots,\omega_1)}$ over $\La\in \Gr_{\leq \mu}$ gives a chain $\La=\La_r\subset\La_{r-1}\subset\cdots\subset \La_0=k[[t]]^n$. Via the framing $\epsilon$, such a chain gives a full flag of $k^r$ and vice versa.

Now we assume $r=n$. Then the formula \eqref{nil to Gr} in fact gives a section $\mN_n\to \Gr_{\leq \mu}^{\square}$ of $\phi$ (since $(t-A)\La_0$ is canonically trivial). Therefore, we obtain the following Cartesian diagram
\begin{equation}\label{sat and springer2}
\begin{CD}
\tilde\mN_n@>>>\Gr_{\leq(\omega_1,\ldots,\omega_1)}\\
@VpVV@VVmV\\
\mN_n@>>>\Gr_{\leq \mu}.
\end{CD}
\end{equation}
In other words, the convolution map $m$ in this case extends the Springer resolution for $\GL_n$.
\end{ex}

We continue the general discussion. The following fundamental theorem summarises the basic facts of the singularities of Schubert varieties. 
\begin{thm}\label{singularity}
Let $p$ denote the characteristic exponent of $k$. If $p\nmid |\pi_1(G_\der)|$, then $\Gr_{\leq \mu}$ is normal, Cohen-Macaulay, Gorenstein, and has rational singularities. 
\end{thm}

\begin{proof}
This is a difficult theorem. 
Except the Gorenstein property, all statements were proved in \cite[Theorem 8]{Fa}, see also \cite[Theorem 0.3]{PR}. The key step is to prove that $\Gr_{\leq \mu}$ is normal.

That $\Gr_{\leq \mu}$ is Gorenstein  follows from \cite[Equation (241)]{BD}. Strictly speaking, the argument of \emph{loc. cit.} only works when the characteristic of $k$ is zero (or is very good for $G$). We refer to \cite[Theorem 6.11]{Z10} for more details and the extension to the case $p\nmid |\pi_1(G_\der)|$. See also Remark \ref{symplectic sing} for related discussions.
\end{proof}

\begin{rmk}
We very briefly discuss the parallel story when $\underline G$ is an Iwahori group scheme of $G$. See \cite{Fa,PR} for details. Write $I=\underline G(\mO)$. In this case, the Cartan decomposition is of the form
\[I\backslash G(F)/I\simeq \widetilde W,\]
where $\widetilde W\simeq \xcoch(T)\rtimes W$ is called the Iwahori Weyl group. This is a (quasi-)Coxeter group equipped with a Bruhat order $\leq$ and a length function $\ell:\widetilde W\to\bZ$.
For $w\in \widetilde{W}$, let $S_w$ denote the corresponding $L^+\underline G$-orbits on $\Gr_{\underline G}=\Fl_G$ and let $S_{\leq w}$ denote its closure. Then $S_w$ is isomorphic to an affine space of dimension $\ell(w)$, and $S_{\leq w}=\cup_{w'\leq w}S_w$. All $S_{\leq w}$'s are normal, Cohen-Macaulay, and with rational singularities. But they are not Gorenstein in general.
\end{rmk}

\subsection{Digression: Some sub-ind-schemes in $\Gr$.}\label{digression}
We make a digression to explain a construction of some important subvarieties of $\Gr$. This subsection is not used in the rest of the notes.  

Let $\varphi: \Gr\to \Gr$ a morphism defined over $k$. For $\ga\in LG$ and $\mu$,
we define a closed sub ind-scheme of $\Gr$ as
\[X(\mu,\ga\varphi)=\left\{ x\in \Gr\mid \inv(x,\ga\varphi(x))\leq \mu\right\}.\]
Explicitly,
write $x=g L^+G\in \Gr$. Then
\[X(\mu,\ga\varphi)= \left\{g L^+G\mid g^{-1}\ga\varphi(g) \in (LG)_{\leq \mu}\right\}/L^+G.\]

Here are a few equivalent definitions.
\begin{enumerate}
\item[(1)] $X(\ga\varphi,\mu)$ is defined by the following Cartesian diagram
\[\begin{CD}
X(\mu,\ga\varphi)@>>> \Gr\tilde\times\Gr_\mu\\
@VVV@VVV\\
\Gr@>1\times \ga\varphi>> \Gr\times \Gr.
\end{CD}\]

\item[(2)] Assume that $\varphi$ is induced by a morphism $\varphi:LG\to LG$ that preserves $L^+G$.
We consider the conjugate action of $LG$ twisted by $\varphi$. I.e. $\on{Int}(g)\cdot^\varphi g'= g^{-1}g' \varphi(g)$. Consider the morphism of stacks
\[\pi: [(LG)_\mu/^\varphi L^+G]\mapsto [LG/^\varphi LG].\]
Then $X(\mu,\ga\varphi)= \pi^{-1}(\ga)$.
\end{enumerate}

\begin{rmk}
In many cases, the quotient $[LG/^\varphi LG]$, although not well-behaved, should be thought as the moduli space of certain important algebraic structures.
\end{rmk}

Here are some concrete examples of the above construction.
\begin{enumerate}
\item[(i)] Let $\varphi=\id$ be the identity map. Then $X(\ga,\mu)$ is a group version affine Springer fiber (cf. \cite[\S~2.7.3]{Yu}). In this case, $[LG/LG]$ classifies conjugacy classes of $G$. There is also a Lie algebra analogue of this construction, which gives the (usual) affine Springer fibers (cf. \cite[\S~2.2]{Yu}).

\item[(ii)] Assume that $k=\bF_q$ and let $\varphi=\sigma$ be the $q$-Frobenius of $\Gr$. Then $X(\mu,\ga\varphi)$ is usually denoted by $X(\mu,b)$ and is called the affine Deligne-Lusztig variety (first introduced in \cite{Ra}). In this case $[LG/^{\sigma}LG]$ classifies $\sigma$-conjugacy classes of $G(\bar{k}((t)))$, which are closely related to Drinfeld Shtukas.

\item[(iii)] We continue to assume that $k=\bF_p$. Let $\varphi:\Gr\to \Gr$ be the morphism induced by the absolute Frobenius $k[[t]]\mapsto k[[t]], \sum a_it^i\mapsto \sum a_i^pt^{ip}$. Write $\ga\varphi$ by $\Phi$. Then $X(\Phi,\mu)$ is called the Kisin variety, and  $[LG/^{\sigma}LG]$ classifies Galois representations $\Gal(\overline F/ F)\to G(\bF_p)$ (see \cite{PR2}). 

\item[(iv)] Assume that $k=\bC$ and let $\varphi$ be induced by the loop rotation $t\mapsto qt$, for some $q\in \bG_m^{\on{rot}}$. Then $[LG/^{\varphi}LG]$ is closely related to $G$-bundles on the elliptic curve $E_q=\bC^\times/q^\bZ$ (Looijenga, Baranovsky and Ginzburg \cite{BG}).

\item[(v)] Instead of taking $\varphi\in \bG_m^{\on{rot}}$, one can let $\varphi\in \Aut^{++}(D)$. A Lie algebra analogue of  this construction was studied in \cite{FZ}.
\end{enumerate}

In each case, some basic questions to ask are
\begin{enumerate}
\item When is the variety finite dimensional?
\item If it is finite dimensional, is there a dimension formula? Are the irreducible components equidimensional?
\item How to parameterise the connected components? 
\item How to parameterise the irreducible components?
\end{enumerate}

These questions are better understood for Case (i) and (ii) (e.g. see \cite{Gortz, Yu} and the references cited there for a summary of known results), but remains widely untouched in other cases. 

\begin{rmk}Thanks to \S\ \ref{p-adic}, affine Springer fibers and affine Deligne-Lusztig varieties now are also defined for $p$-adic groups.\end{rmk}

\subsection{Opposite Schubert ``varieties'' and transversal slices.}
Opposite Schubert ``varieties'' are very different from Schubert varieties $\Gr_{\leq \mu}$ introduced in \S~\ref{II:Sch var}. In fact, they are not really varieties and are infinite dimensional (but finite codimensional). 

First, let $L^-G$ be the presheaf of groups defined as
\[L^-G(R)=G(R[t^{-1}]).\]
In \S~\ref{Kac-Moody}, we shall see that $LG$ is closely related a Kac-Moody group. From this point of view, $L^-G$ then is a ``parabolic opposite'' to $L^+G$.
There is another point of view. Namely, $L^-G$ is a special case of  the group $\underline G^{X^*}$ which will be introduced in \eqref{Gout} (corresponding to the case $X=\bP^1$ and $x=0$). In particular, by Lemma \ref{rep of Gout} below, the group $L^-G$ is represented by an ind-scheme.

We fix $T\subset B\subset G$ as in \S~\ref{II:Sch var}.
Then there is the Birkhoff decomposition of $G(F)$
\begin{equation}\label{Birk decom}
G(F)=\bigsqcup_{\mu\in \xcoch(T)^+} G(k[t^{-1}])t^\mu G(\mO),
\end{equation}
which induces
\begin{equation}\label{Birk decom 2}
L^-G(k)\backslash G(F)/G(\mO)\cong \xcoch(T)^+.
\end{equation}
Similar to the Cartan decomposition \eqref{Cartan decomp2}, \eqref{Birk decom 2} is independent of any choice.
We refer to \cite[Lemma 4]{Fa} for a proof of this decomposition for the general $G$.

We define opposite Schubert ``cells'' as $L^-G$-orbits on $\Gr$. By the Birkhoff decomposition \eqref{Birk decom 2}, they are of the form
\[\Gr^\mu=L^-G\cdot t^\mu \subset \Gr,\]
and are parameterised by $\xcoch(T)^+$. 
Let 
$$\Gr^{\geq \mu}=\bigsqcup_{\la\geq \mu}\Gr^\la,$$ called the opposite Schubert variety. The terminology is justified by the following proposition.
\begin{prop}\label{geom of opp Sch}~
\begin{enumerate}
\item $\Gr^\mu$ is a locally closed sub-ind-scheme of $\Gr$.

\item $\Gr_\mu\cap \Gr^\la\neq \emptyset$ if and only if $\mu\geq \la$. In addition, $\Gr_\mu\cap \Gr^\mu= G\cdot t^\mu\cong G/P_\mu$, which is fixed by the rotation torus $\bG_m^{\on{rot}}$.

\item $\Gr^{\geq \mu}$ is Zariski closed, and contains $\Gr^\mu$ as an open dense subset.

\item The codimension of $\Gr^{\geq \mu}$ is $(2\rho,\mu)-\dim G/P_\mu$.
\end{enumerate}
\end{prop}
\begin{proof}
(1) The stabiliser of $t^\mu$ in $L^-G$ is $L^-G\cap t^\mu L^+G t^{-\mu}$, which is a finite dimension subgroup. By writing $L^-G=\underrightarrow\lim K_i$ such that $K_i\supset L^-G\cap t^\mu L^+G t^{-\mu}$, we see that each $K_it^\mu$ is locally closed in $\Gr$ and $\Gr^\mu=\underrightarrow\lim K_it^\mu$ is an ind-scheme. We also note that 
\begin{equation}\label{dim opp Sch}
\dim(L^-G\cap t^\mu L^+G t^{-\mu})=(2\rho,\mu)+\dim P_\mu,
\end{equation} 
which follows by calculating the dimension of its Lie algebra as in the argument for Proposition \ref{geom of Sch}.

(2) We leave it as an exercise.

(3) We first show that $\Gr^0$ is open in the neutral connected component of $\Gr$. Note that there is an evaluation map
\[L^-G\to G,\quad g\mapsto g\mod t^{-1}.\]
We define $L^{<0}G=\ker(L^-G\to G)$ as the kernel of the above map. Note that $\Gr^0= L^{<0}G\cdot e$.
\begin{lem}\label{big cell}
Then the multiplication map
\[L^{+}G\times L^{<0}G\to LG\]
is an open embedding. In particular, $L^{<0}G\cong L^{<0}G\cdot e\subset \Gr$ is an open immersion.
\end{lem}
\begin{proof}Note that $L^{<0}G\cap L^+G=\{1\}$ and at the level of Lie algebras $L\frakg=L^+\frakg\oplus L^{<0}\frakg$, where $L^+\frakg=\frakg(\mO)$ and $L^{<0}\frakg=t^{-1}\frakg\otimes k[t^{-1}]$.
Therefore $L^{+}G\times L^{<0}G\to LG$ is an open embedding. 
\end{proof}
Now consider $L^-G\cdot t^\mu\cdot \Gr^0$. It is open and is a union of $L^-G$-orbits. It is easy to see that $t^\la$ belongs to it  if and only if $\Gr^\la\cap \Gr_\mu\neq \emptyset$ and therefore by (2) if and only if $\la\leq \mu$. Therefore, $\sqcup_{\la\leq \mu}\Gr^\la=L^-G\cdot t^\mu \cdot \Gr^0$ is open. It follows that $\Gr^{\geq \mu}$ then is closed. It remains to show that $\Gr^\la$ is in the closure of $\Gr^\mu$ for every $\la>\mu$. As argued in Proposition \ref{geom of Sch}, we can assume that $\la=\mu+\al$ for some positive root $\al$. Let $C_{\mu+\al,\al}$ be the curve defined by \eqref{curve C}. Then it is easy to see that $t^{\mu+\al}\in C_{\mu+\al,\al}$ and the complementary $\bA^1\subset \Gr^\mu$. This finishes the proof.

(4) follows from (3) and \eqref{dim opp Sch}.
\end{proof}
\begin{rmk} (i) The open subset $\Gr^0$ is usually called the ``big open cell'' of $\Gr$.

(ii) As we just see, the roles played by $L^+G$ and $L^-G$ are different. So it is also natural to consider the quotient $LG/L^-G$. By Theorem \ref{uniformization for P1} below, it can be interpreted as the moduli of $G$-torsors on $\bP^1$ together with a trivialisation on $D_0$ (the disc around $0\in \bP^1$). Sometimes in literature, $LG/L^-G$ is called the thick affine Grassmannian of $G$. Note that unlike $\Gr$, it is of global nature (i.e. its definition involves a global curve) and is a scheme (of infinite type).  
\end{rmk}

The best way to understand opposite Schubert cells/varieties is via the following theorem. Let $\Bun_G(\bP^1)$ denote the moduli stack of $G$-bundles on $\bP^1$. This is an algebraic stack locally of finite presentation over $k$ (e.g. see \cite{So}).

\begin{thm}\label{uniformization for P1}
Assume that $G$ is connected reductive over $k$. There is a canonical isomorphism 
\[ [L^-G\backslash \Gr]\simeq \Bun_G(\bP^1).\]
In fact, the map $\Gr\to \Bun_G(\bP^1)$ exhibits $\Gr$ as an $L^-G$-torsor on $\Bun_G(\bP^1)$ in \'etale topology.
\end{thm}
\begin{proof}
The statement is equivalent to saying that after passing to an \'etale covering map $R\to R'$ every $G$-bundle on $\bP^1\times \Spec R$ is trivial on $(\bP^1-\{0\})\times \Spec R'$. This follows from: (i) every $G$-bundle on $X_R$ admits a reduction to the Borel after a possible \'etale base change $R\to R'$ (\cite[Theorem 1]{DS}), and (ii) every $B$-bundle on $(\bP^1-\{0\})\times \Spec R$ is trivial after a possible \'etale base change $R\to R'$ since $B$ has a filtration by subgroups with the associated graded being $\bG_a$s and $\bG_m$s.
\end{proof}

At the level of $k$-points, the theorem says that
\[ \Bun_{G}(\bP^1)(k)\cong G(k[t^{-1}])\backslash G(k((t))/ G(k[[t]]).\]
In particular, Grothendieck's theorem for vector bundles over $\bP^1$ is equivalent to the Birkhoff decomposition \eqref{Birk decom} (for $G=\GL_n$).
Explicitly, the $G$-bundle corresponding to the double coset for $t^\mu$ is the one induced from the line bundle $\mO_{\bP^1}(1)$ by pushout along $\bG_m\stackrel{\mu}{\to}T\subset G$. We denote this $G$-bundle by $\mO(\mu)$, and let $[\mO(\mu)]$ denote the corresponding point on $\Bun_G(\bP^1)$.
Then Proposition \ref{geom of opp Sch} can be reformulated as
\begin{prop}\label{geom of strata}~
\begin{enumerate}
\item The point $[\mO(\la)]$ is in the closure of $[\mO(\mu)]$ if and only if $\la\geq \mu$; 
\item The closure of $[\mO(\mu)]$ is a closed substack of dimension $-(2\rho,\mu)-\dim P_\mu$.
\end{enumerate}
\end{prop}

As an application of Proposition \ref{geom of opp Sch},  we can also introduce the transversal slices.
\begin{prop}$L^{<0}G\cdot t^\la$ is a transverse slice of $\Gr_\la$ in $\Gr$. More precisely, for every $\la\leq \mu$, let $S^\la_{\leq \mu}=L^{<0}G\cdot t^\la\cap \Gr_{\leq \mu}$. Then $S^\la_{\leq \mu}$
intersects $\Gr_\nu$ transversally for every $\la\leq \nu\leq \mu$.
\end{prop}
\begin{proof}It follows from Lemma \ref{big cell}.
\end{proof}

\begin{rmk}\label{symplectic sing}
The variety $S^\la_{\leq \mu}$ is an example of a variety with symplectic singularities (cf. \cite[Theorem 2.7]{KWWY}). We refer to \cite[Definition 1.1]{Bea} for the definition of this notion. According to \cite[Proposition 1.3]{Bea}, symplectic singularities are rational Gorenstein, which is consistent with Theorem \ref{singularity}.
\end{rmk}
We discuss a few examples.  
\begin{ex}\label{slice GLn}
We consider $G=\GL_n$. Let $\mu=(n,0,\ldots,0)$. First, we have
\begin{lem}\label{Lus obs}
The map \eqref{nil to Gr} is an open embedding.
\end{lem}
\begin{proof}Indeed, the map $\mN_n\to L^{<0}\GL_n,\ A\mapsto 1-At^{-1}$ is a closed embedding and therefore by Lemma \ref{big cell}, the map $\mN_n\to \Gr_{\GL_n},\ A\mapsto (t-A)\La_0$ is a locally closed embedding with image contained in $\Gr_{\leq \mu}$. Therefore, \eqref{nil to Gr} itself is a locally closed embedding. Since $\Gr_{\leq \mu}$ is irreducible and reduced and  $\dim \mN_n=\dim \Gr_{\leq \mu}=n(n-1)$, \eqref{nil to Gr} must be an open embedding. 
\end{proof}
Now let $\la=(1,\ldots,1)$. Note that $\Gr_\la$ is a point. Then Lemma \ref{Lus obs} is equivalent to saying that the map \eqref{nil to Gr} realises the transversal slice $S_{\leq \mu}^\la\subset \Gr_{\leq \mu}$. 
As a consequence $S_{\leq \mu}^\nu$ is isomorphic to the Slodowy slice for $\mO_\nu$ where $\la\leq \nu\leq \mu$.
\end{ex}
\begin{ex}
We assume that $G_\der$ is simple.. Let $\mu=\theta$ be quasi-minuscule. By the same construction as above, 
$S^0_{\leq \theta}$ is isomorphic to the closure of the minimal nilpotent orbit $\overline{\mO}_{\on{min}}$ of $\frakg$.
\end{ex}
\begin{rmk}
Note that Theorem \ref{singularity} implies
nilpotent orbits of $\GL_n$ and the minimal nilpotent orbit of a general semisimple Lie algebra $\frakg$ are normal, Cohen-Macaulay, Gorenstein, and have rational singularities.
This is not true for the general nilpotent orbits in semisimple Lie algebras, which implies that the above relation between transversal slices an nilpotent orbits does not extend to the general case.
\end{rmk}

\subsection{The Picard group.}\label{The Picard group}
We assume that $G$ is simple and simply connected (for simplicity). Let $\theta$ be the short dominant coweight, which we recall is a minimal element in $\xcoch(T)^+-\{0\}$. By Proposition \ref{geom of opp Sch}, $\Theta:=\Gr^{\geq \theta}$ is of codimension one, and $\Gr=\Gr^0\sqcup \Theta$. We show that it is the vanishing loci of a global section of an ample line bundle on $\Gr$.

\begin{prop}\label{big cell 2}
There is an ample line bundle $\mL$ and a global section $\sigma\in\Ga(\Gr,\mL)$ such that $\Gr^0=\{x\in\Gr \mid \sigma(x)\neq 0\}$.
\end{prop}
\begin{proof}We explain the ideas and refer to \cite[Lemma 2, Corollary 3]{Fa} for details.
We first consider $G=\SL_n$. In fact $L^{<0}G\cdot e$ is the non-vanishing loci of a global section $\sigma$ of $\mL_{\det}$ constructed as follows. Let $L_0=(t^{-1}k[t^{-1}])^n$ and consider 2-term complex on $\Spec R$
\[0\to R\otimes L_0\oplus\Lambda\stackrel{\Upsilon}{\to} R\otimes L_0\oplus R[[t]]^n\to 0,\quad \Upsilon(v,w)=(v,0)+\beta(w),\]
where $\beta: \La\to R((t))=R\otimes L_0\oplus R[[t]]$ is the natural inclusion.
By definition, the determinant cohomology of this complex (or more precisely the complex obtained by dividing both terms by $ R\otimes L_0\oplus t^NR[[t]]^n$ for $N\gg 0$) is exactly the determinant line bundle $\mL_{\det}$ constructed in \S~\ref{determinant line}, and $\sigma=\det(\Upsilon)$ defines a global section of $\mL_{\det}$. Let us represent a point $x\in \Gr$ as $g\La_0$ for some $g\in LG$. Then by definition, $\sigma(g\La_0)\neq 0$ if and only if $\Upsilon$ is an isomorphism at this point. Note that the latter condition in turn is equivalent to $g^{-1} L_0=L_0$ (cf. \cite[Lemma 2]{Fa}). This finishes the proof for $\SL_n$.

Next, let $\rho:G\to \GL(V)$ be a representation such that there is a vector $v\in V$ whose stabiliser in $\GL(V)$ is $G$. Such a representation exists since $G$ is semisimple. Let $f_\rho: \Gr_G\to \GL_{\SL_n}$ be the induced map.
One can show that $f_\rho^{-1}(L^{<0}\SL_n\cdot e)=L^{<0}G\cdot e$. This implies that $\Theta_G$ is the pullback of $\Theta_{\SL_n}$ (as closed subsets). 
\end{proof}

Recall the definition of line bundles on a prestack (cf. \S~\ref{Intro:sp}).
Let $e\in X(k)$ be a point. We define the Picard groupoid $\Pic^e(X)$ of line bundles on $X$ rigidified at $e$ as the groupoids consisting of $(\mL,\epsilon)$ where $\mL$ is a line bundle on $X$ and $\epsilon: e^*\mL\simeq k$ is a trivialisation.
Note that if $X$ is a proper algebraic variety, then $\Pic^e(X)$ is discrete and reduces to the usual Picard group of $X$.

\begin{thm}\label{Pic of Gr}
Assume that $G$ is simple and simply-connected.
Endow $\Theta$
with the reduced closed ind-scheme structure. Then $\Theta$ is an effective Cartier divisor, and the corresponding line bundle $\mO(1):=\mO(\Theta)$ is ample. In addition, $\Pic^e(\Gr)=\bZ\mO(1)$.
\end{thm}

\begin{rmk}\label{generator}
If $\rho: G\to \GL_n$ is a representation, the pullback of $\mL_{\det}$ along $f_\rho:\Gr_G\to \Gr_{\GL_n}$ then is a positive multiple of $\mO(1)$, i.e. $f_\rho^*\mL_{\det}=\mO(d_\rho)$ for some positive integer $d_\rho$, called the Dynkin index of the representation $\rho$. A general formula of $d_\rho$ can be found in \cite[Lemma 5.2]{KNR}, see also \cite[Lemma 10.6.1]{So} and \cite[Lemma 4.2]{YZ} (note that although only the case $\cha\ k=0$ was considered in all the above references, the argument of \cite[Lemma 4.2]{YZ} works for arbitrary characteristic).

For example, if $\rho$ is the adjoint representation, then $d_\rho=2h^\vee$, where $h^\vee$ is the dual Coxeter number of $G$. If $\rho$ is the standard representation of $\SL_n$ or $\on{Sp}_{2n}$, then $d_\rho=1$. If $G$ is an orthogonal group $\on{SO}_n$ (say $\cha\ k>2$) and $\rho$ is the standard representation, or $G=G_2$ and $\rho$ is the unique $7$-dimensional representation, then $d_\rho=2$ and the unique square root can be obtained by pullback of the pfaffian line bundle $\mL_{\on{Pf}}$ on $\on{OGr}(V)$ (see Remark \ref{Pfaffians}). 

From a different point of view, the above discussions essentially give the construction of $\mO(1)$ when $G$ is a classical group or $G=G_2$.
However, if $G=E_8$, $60\mid d_\rho$ for any $\rho$. So it is an interesting question to find an explicit construction of $\mO(1)$ in this case.
\end{rmk}

\begin{rmk}\label{square root of adjoint}
It follows that for the adjoint representation $\Ad: G\to \GL(\frakg)$, $f_{\Ad}^*\mL^{-1}_{\det}$ admits a unique rigidified square root, denoted by $\mL_{c}$. This is the so-called critical line bundle on the affine Grassmannian, which plays an important role in the geometric Langlands program (cf. \cite{BD}). 
\end{rmk}

One way to prove this theorem is to identify $\Gr$ with a partial flag variety of the corresponding Kac-Moody group (see \S~\ref{Kac-Moody} below) and then to invoke the results from \cite{Ma}. Here is a more direct approach. Namely, $\mO(1)$ should be the line bundle associated to the divisor $\Theta$. However, since $\Gr$ is not finite dimensional, and in general cannot be written as an increasing union of smooth finite dimensional subvarieties, its really definition relies on the following lemma.
\begin{lem}\label{O1 on Bun}
Assume that $G$ is simple and simply-connected.
There is an open substack $\mathring{\Bun}_G(\bP^1)$ of $\Bun_G(\bP^1)$ classifying trivialisable $G$-torsors on $\bP^1$, and the boundary $\Bun_G\setminus \mathring{\Bun}_G$ is pure of codimension one. Therefore, it is an effective Cartier divisor, defining a line bundle $\mO(1)$ on $\mathring{\Bun}_G(\bP^1)$.
\end{lem}
\begin{proof}We have $\mathring{\Bun}_G(\bP^1)= [L^-G\backslash L^{<0}G\cdot e]$, and $\Bun_G\setminus \mathring{\Bun}_G=[L^-G\backslash \Theta]$ is pure of codimension one. Alternatively, one can prove this using deformation theory. \end{proof}

Then the pullback defines a line bundle on $\Gr$, temporally denoted by $\mL(\Theta)$. Note that $\mL(\Theta)$ is ample since its certain multiple is the pullback of the determinant line bundle  by Proposition \ref{big cell}.
 
Next, it is not hard to show that the rational curve $C_{\theta,\theta}$ constructed in \eqref{curve C} intersects with $\Theta$ transversally at $t^\theta$ (e.g. by looking at the tangent spaces), where as before $\theta$ is the short dominant coroot. In addition, one can show that the restriction to $C_{\theta,\theta}$ induces an injective map of Picard groups (using Demazure resolutions). It follows that there are inclusions $\bZ\mL(\Theta)\subset \Pic^e(\Gr)\subset\Pic(\bP^1)$ which send $\mL(\Theta)$ to $\mO_{\bP^1}(1)$.
 The theorem then follows. See \cite[Corollary 12]{Fa} and \cite[\S~10.2]{So} for more details.

We leave the following statement as an exercise.
\begin{lem}\label{lifting Aut action}
The action of $\Aut^+(D)$ on $\Gr$ lifts to a unique action on $\mO(1)$ which restricts to the trivial action on the fiber of $\mO(1)$ at $e$.
\end{lem}

\subsection{Central extensions, and affine Kac-Moody algebras.}\label{Kac-Moody}
We continue to assume that $k$ is algebraically closed and $F=k((t))$.
Recall that affine Dynkin diagrams classify two different types of groups.
Let $\Gamma$ be a connected affine Dynkin diagram, and $S$ be a subset of vertices of $\Gamma$. The work of Bruhat-Tits attaches to $\Ga$ a simple and simply-connected reductive group $G_\Gamma$ over $F$ and to $S$ a parahoric group scheme $\underline G_S$ over $\mO$ (cf. \cite{Ti} for a summary and in particular \S 4.2 of \emph{ibid.} for the classification). On the other hand, attached to $\Ga$ there is a Kac-Moody algebra $\frakg_\Ga$ (\cite[Chap. 4]{Kac}) and the corresponding Kac-Moody group $G_{\Ga}$ and to $S$ a parabolic subgroup $P_S$ of $G_{\Ga}$\footnote{In fact, the situation is much more complicated. There are various versions of infinite dimensional groups (e.g. minimal, formal, etc.)  attached to $\frakg_\Ga$ due to the work of many people (Moody-Teo, Marcuson, Garland, Kac-Perterson, Slodowy, Tits, Mathieu,...). See \cite{Ti3} for a summary and references cited there.}.

Now we explain the relation between the Bruhat-Tits theory and the Kac-Moody theory by explaining the dotted arrow in the following diagram.
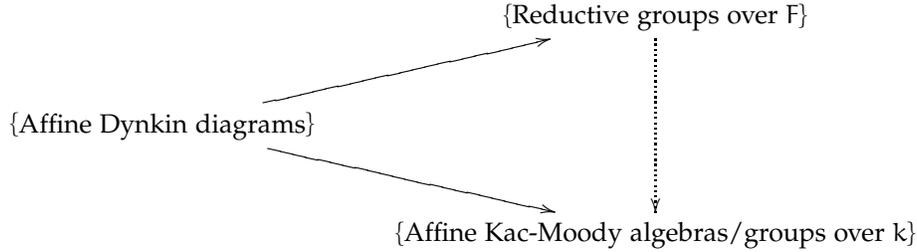
\[\xymatrix{
&\left\{\textrm{Reductive groups over } F\right\}\ar@{.>}[dd]\\
\left\{\textrm{Affine Dynkin diagrams}\right\} \ar[ur]\ar[dr]&\\
&\left\{\textrm{Affine Kac-Moody algebras/groups over } k\right\}
}\]

We will assume that $\Ga$ is an untwisted affine Dynkin diagram\footnote{The twisted cases are more complicated, in particular when $k$ has small characteristic.}. These are the affine Dynkin diagrams listed in \cite[Chap. 4, TABLE Aff 1]{Kac}. They can be obtained from the finite Dynkin diagrams (listed in \cite[Chap 4, TABLE Fin]{Kac}) by adding a distinguished node $v_0$. Reductive groups over $F$ corresponding to these Dynkin diagrams are those base change of simple and simply-connected reductive groups from $k$. If $S=\{v_0\}$, then the corresponding parahoric group scheme is hyperspecial, and is of the form $G\otimes_kk[[t]]$ (see Example \ref{hyperspacial and Iwahori}).

In this case, the line bundle $\mO(1)$ on $\Gr_G$ allows one to define a central extension of $LG$ by $\bG_m$
\begin{equation}\label{central extension}
1\to\bG_m\to \widehat{LG}\to LG\to 1.
\end{equation}
Namely, an $R$-point of $\widehat{LG}$ is an $R$-point $g\in LG(R)$ together with an isomorphism
\[g^*\mO(1)\simeq \mO(1).\] 
The action of the rotation torus canonically lifts to an action on $\widehat{LG}$ (by Lemma \ref{lifting Aut action}).
Let 
\[\widetilde{LG}:=\widehat{LG}\rtimes\bG_m^{\on{rot}}.\]
\begin{thm}The ind-group $\widetilde{LG}$ is isomorphic to the (formal version of) Kac-Moody group $G_\Ga$ constructed from the corresponding untwisted affine Dynkin diagram.
\end{thm}
This can be deduced from the following more precise theorem. Let $\underline G$ be a parahoric group scheme of $G$, corresponding to a subset $S$ of vertices of the affine Dynkin diagram of $G$. Let $P_S\subset G_\Ga$ be the corresponding parabolic subgroup in the affine Kac-Moody group. Then
\begin{thm}
The affine Grassmannian $\Gr_{\underline G}$ is isomorphic to the partial flag variety $G_\Ga/P_S$ as defined in \cite{Ma,Ku}.
\end{thm}
\begin{proof}See \cite{BL} for the case $G=\SL_n$ and \cite{Fa} for general $G$ (see also \cite[\S 10]{PR}). The idea is to first prove that the Schubert varieties in  $\Gr_{\underline G}$ are isomorphic to the (a priori different) Schubert varieties arising from Kac-Moody theory.
\end{proof}

\begin{rmk}\label{2-cocycle}
The Lie algebra $\widehat{L\frakg}$ of $\widehat{LG}$ is a central extension of $\frakg(F)$ by $k$. If $\cha\ k=0$, it can be explicitly described as $\frakg\otimes F\oplus k\mathbf{1}$ with the Lie bracket given by
\[[X\otimes f,\mathbf{1}]=0,\quad [X\otimes f,Y\otimes g]=[X,Y]\otimes fg- (\Res_{t=0}fdg) (X,Y)\mathbf{1},\]
where $X,Y\in \frakg,\ f,g\in F$, $(-,-): \frakg\times\frakg\to k$ is the normalised invariant bilinear form on $\frakg$ and $\Res: Fdt\to k$ is the residue map. We refer to \cite[\S~7]{Kac} for details. 
\end{rmk}

Combined with main results of \cite{Ku,Ma}, one obtains a Borel-Weil type theorem. For simplicity, we assume that $\cha\ k=0$ and state it for untwisted affine Dynkin diagram $\Ga$ and $S=\{v_0\}$.
\begin{thm}\label{Borel-Weil}
The dual of the space of global sections $\Ga(\Gr_G,\mO(m))^*$ is isomorphic to $L_m(\frakg_\Ga)$, the level $m$ integrable representation of $\frakg_{\Ga}$.
\end{thm}
Integrable representations of Kac-Moody groups are the analogues of finite dimensional representations of reductive algebraic groups and can be constructed purely by representation theory method.
We briefly recall the definition of $L_m(\frakg_\Ga)$. We follow notations from \cite{Kac}. Let $\Lambda_0$ be the fundamental weight of $\frakg_\Ga$ as defined in \emph{ibid.} \S\ 6, and let $M(m\Lambda_0)$ be the Verma module of $\frakg_\Ga$ of highest weight $m\Lambda_0$ as defined in \emph{ibid.} \S\ 9. Then $L_m(\frakg_\Ga)$ is the quotient of $M(m\Lambda_0)$ by the submodule generated by those vectors given in \emph{ibid.} Formula (10.1.1).

\section{Lecture III: Beilinson-Drinfeld Grassmannians and factorisation structures}\label{Lecture III}
In this section, we introduce a version of affine Grassmannians over curves, known as Beilinson-Drinfeld Grassmannians. We also discuss some of their basic geometric properties, in particular the factorisation property.

\subsection{Beilinson-Drinfeld Grassmannians.}
A great discovery of
Beilinson-Drinfeld is that the interpretation of the affine Grassmannian via Theorem \ref{glob aff} admits vast generalisations.

First, we can allow the point $x$ in Theorem \ref{glob aff} to move. For a point $x\in X(R)$, let $\Ga_x$ denote the graph of the map $x:\Spec R\to X$. Then one can define the prestack
\[\Gr_{\underline G, X}(R)=\left\{(x,\mE,\beta)\ \left|\ \begin{split}&x\in X(R),\quad \mE \mbox{ is a } \underline{G}\mbox{-torsor on } X_R, \\
& \beta: \mE|_{X_R-\Ga_x}\cong \mE^0|_{X_R-\Ga_x} \mbox{ is a trivialisation}\end{split}\right.\right\}.\]
But in fact we can do more by allowing multiple points. Let $I$ be a finite non-empty set. For a point $x=\{x_i,i\in I\}\in X^I(R)$, let $$\Ga_x:=\bigcup_{i\in I} \Ga_{x_i}$$ denote the union of the graphs of the maps $x_i: \Spec R\to X$ (as closed subschemes of $X_R$).

\begin{dfn}
We define the presheaf $\Gr_{\underline G,X^I}$ over the self-product of curves $X^I$ as
\begin{equation}\label{BD Grass def}
\Gr_{\underline G, X^I}(R)=\left\{(x,\mE,\beta)\ \left|\ \begin{split}&x\in X^I(R),\quad \mE \mbox{ is a } \underline{G}\mbox{-torsor on } X_R, \\
& \beta: \mE|_{X_R-\Ga_x}\cong \mE^0|_{X_R-\Ga_x} \mbox{ is a trivialisation}\end{split}\right.\right\}.
\end{equation}
Let $q_I:\Gr_{\underline G,X^I}\to X^I$ denote the structural map.
\end{dfn}

\begin{thm}\label{representability of BD}
The presheaf $\Gr_{\underline G, X^I}$ is represented by an ind-scheme, ind-of finite type over $X^I$. It is ind-projective if $\underline G$ is reductive.
\end{thm}
\begin{proof}
Again, one reduces to prove the theorem in the case when $\underline G=\GL_n$ by arguments similar to Proposition \ref{lin rep} and \ref{gen G}. Then it is enough to prove that for positive integer $N$,
\[\Gr_{\GL_n,X^I}^{(N)}=\left\{(x,\mE,\beta)\in\Gr_{\GL_n,X^I}\mid \mO_{X_R}^n(-N\Ga_x)\subset \mE\subset \mO_{X_R}^{n}(N\Ga_x)\right\}\]
is represented by a projective scheme over $X^I$. By an argument as in Lemma \ref{projectivity}, the map $\mE\mapsto \mF:=\mO_{X_R}^{n}/\mE(-N\Ga_x)$ identifies it with the functor that assigns $R$ the pairs $(x,\mF)$ where $x\in X(R)$ and $\mF$ is an $R$-flat $\mO_{X_R}$-module that is a quotient of $\mO_{X_R}^n/\mO_{X_R}^n(-2N\Ga_x)$.
Its representability then follows from the representability of quot schemes (in place of Grassmannians used in Lemma \ref{Embedding}). 
\end{proof}
For simplicity, write $\Gr_{\underline G,X^I}$ by $\Gr_{X^I}$ if the group $\underline G$ is clear for the context.

\begin{rmk}\label{projectivity of BD}
The last statement of Theorem \ref{representability of BD} holds in a more general situation.
We call a smooth affine group scheme $\underline G$ over $X$ a Bruhat-Tits group scheme (or a parahoric group scheme) if its restriction to every $D_x, x\in |X|$ is a parahoric group scheme in the sense of Bruhat-Tits. Using Proposition \ref{indproper}, an argument as in \cite{PZ} shows that $\Gr_{X^I}$ is ind-proper over $X^I$.
\end{rmk}

One can similarly define the multi-point version of the jet groups and the loop groups. For $x\in X^I(R)$, let $\hat{\Ga}_x$ denote the formal completion of the graphs $\Ga_x$ in $X_R$ (which is a formal scheme). We have the canonical map $\pi:\hat{\Ga}_x\to X_R$.
\begin{equation}
(L^+\underline G)_{X^I}=\left\{(x,\beta)\mid x\in X^I(R), \beta\in
\underline{G}(\hat{\Ga}_x)\right\}.
\end{equation}
\begin{prop}
$(L^+\underline{G})_{X^I}$ is represented by a scheme affine over $X^I$.
\end{prop}

To rigorously define the global loop groups, we need the following lemma. 
\begin{lem}The formal scheme $\hat{\Ga}_x$ is ind-affine. Let $A_x$ be the topological algebra of its ring of functions and let $\hat{\Ga}'_x=\Spec A_x$. There is the canonical map $p:\hat{\Ga}_x\to\hat{\Ga}'_x$. Then there is a unique morphism $i: \hat{\Ga}'_x\to X_R$ such that $\pi=ip:\hat{\Ga}_x\to X_R$.
\end{lem}
\begin{proof}
See \cite[Proposition 2.12.6]{BD}.
\end{proof}
Then it makes sense to define $\hat{\Ga}_x^\circ=\hat{\Ga}'_x\setminus \Ga_x$.

\begin{equation}
(L\underline G)_{X^I}=\left\{(x,\beta)\mid x\in X^I(R), \beta\in
\underline G(\hat{\Ga}^\circ_x)\right\}.
\end{equation}

By Theorem \ref{descent lemma} and similar arguments as in Corollary \ref{global interpretation of loop}, we have
\begin{prop}
As $k$-spaces, $\Gr_{\underline G,X^I}\cong[(L\underline G)_{X^I}/(L^+\underline G)_{X^I}]$. In particular,
\begin{enumerate}
\item  $(L\underline G)_{X^I}$ is represented by an ind-scheme over $X^I$; 

\item $(L\underline G)_{X^I}$ naturally acts on $\Gr_{\underline G,X^I}$.
\end{enumerate}
\end{prop}

Beilinson-Drinfeld Grassmannians are truly remarkable objects. To wit, let us first analyse the case $I=\{1\}$ and $I=\{1,2\}$. 

Let $I=\{1\}$ and first assume that $\underline G=G\times X$ is constant. 
For every $x\in X$, by choosing an isomorphism $\mO_x\simeq \mO$, we obtain a canonical isomorphism
\[\Gr_{X}|_x=\Gr_{\underline G, x}\cong \Gr_{G}\]
by Theorem \ref{glob aff}.
However, globally
there is no canonical identification $\Gr_X\neq \Gr_G\times X$ in general. Rather, $\Gr_X$ can be regarded as a twisted product $X\tilde\times\Gr_G$ (cf. \S~\ref{Intro:torsor} for the definition of the twisted product). Namely, there is an $\Aut(D)$-torsor $\hat{X}\to X$ over $X$, classifying $x\in X$ together with $D_x\cong D$. Precisely
\[\hat{X}(R)=\left\{(x,\al)\mid x\in X(R), \al: \hat{\Ga}'_x\cong \Spec R[[t]].\right\}\]
 Likewise, there is the $\Aut^+(D)$-torsor $\hat{X}^+$ consisting of those $\al$ that restricts to the canonical isomorphism $\Ga_x\simeq \Spec R$. Recall that $\Aut(D)$ acts on $\Gr_G$ (as defined by \eqref{auto of disc}).
Then there are isomorphisms
\begin{equation}\label{III: twisted constr}
\Gr_X\cong \hat{X}\times^{\Aut(D)}\Gr_G\cong \hat{X}^+\times^{\Aut^+(D)}\Gr_G.
\end{equation}
The first isomorphism in fact implies that there exists a connection on $\Gr_X$ (see also Proposition \ref{III:property}).
The second isomorphism  allows one to define
\begin{equation}\label{III:SV over X}
\Gr_{\leq \mu,X}: =X\tilde\times \Gr_{\leq \mu}:=\hat{X}^+\times^{\Aut^+(D)}\Gr_{\leq \mu}
\end{equation}
as a closed subscheme. Alternatively, $\Gr_{\leq \mu,X}$ can be identified with the closed subset (with the reduced scheme structure) of $\Gr_X$ classifying those points $(x,\mE,\beta)$ such that $\inv_x(\beta)\leq \mu$.

\begin{rmk}
One can also describe $(L^+G)_X$ and $(LG)_X$ using the above twisting procedure.
\end{rmk}

We still assume that $I=\{1\}$, but that $\underline G$ is not necessarily constant. For example, one can construct a group scheme $\underline G$ over $X$ such that $\underline G|_{X-x}\simeq G\times (X-x)$ and $\underline G|_{D_x}$ is an Iwahori group scheme of $G$. Then $\Gr_{\underline G, X}$ is a deformation from the affine Grassmannian to the affine flag variety of $G$ (see Remark \ref{terminology} for the terminology). I.e. the fiber of $\Gr_{\underline G, X}$ over a point $y\neq x$ on $X$ is isomorphic to $\Gr_G$ and its fiber over $x$ is isomorphic to $\Fl_G$. This deformation plays a crucial role in the construction of central sheaves by Gaitsgory \cite{Ga}. Similar deformations over the ring $p$-adic integers were used in \cite{PZ} to define the local model of Shimura varieties.

Next, we turn to $\Gr_{X^2}$. Let $\Delta:X\to X^2$ be the diagonal.
\begin{prop}\label{III:Gr2}
There are canonical isomorphisms
\[\Delta:\Gr_X\cong \Gr_{X^2}\times_{X^2,\Delta}X,\quad c:\Gr_{X^2}|_{X^2-\Delta}\cong (\Gr_X\times\Gr_X)|_{X^2-\Delta}.\]
In addition, there is a natural $S_2$-action on $\Gr_{X^2}$ such that $c$ is $S_2$-equivariant, where $S_2$ acts on the $\Gr_X\times\Gr_X$ by the permutation of two factors.
\end{prop}
\begin{proof}The first isomorphism is clear. To construct the second one, let $(x,y)\in (X^2-\Delta)(R)$. Then $(X_R-\Ga_x)$ and $(X-\Ga_y)$ form a Zariski cover of $X_R$. We define a map $$\Gr_{X^2}\times_{X^2}\Spec R\to \Gr_X\times_{X,x}\Spec R$$ by sending $(\mE,\beta)$ to $(\mE_x,\beta_x)$, where $\mE_x$ is obtained by gluing $\mE|_{X_R-\Ga_y}$ and $\mE^0|_{X_R-\Ga_x}$ via $\beta$ and therefore is equipped with a trivialisation $\beta_x: \mE|_{X_R- \Ga_x}\cong \mE^0|_{X_R-\Ga_x}$. Similarly, we have a map $\Gr_{X^2}\times_{X^2}\Spec R\to \Gr_X\times_{X,y}\Spec R$. Together, they define a map $\Gr_{X^2}|_{X^2-\Delta}\to (\Gr_X\times\Gr_X)|_{X^2-\Delta}$.

Conversely, if we have $(\mE_x,\beta_x)\in \Gr_X\times_{X,x}\Spec R$ and $(\mE_y,\beta_y)\in \Gr_X\times_{X,y}\Spec R$, we can construct $\mE$ by gluing $\mE_{x}|_{X_R- \Ga_y}$ and $\mE_y|_{X_R-\Ga_x}$ by $\beta^{-1}_y\beta_x$, which by definition is naturally equipped with a trivialisation $\beta$ on $X_R-\Ga_x-\Ga_y$.

Finally, the non-trivial element $\sigma\in S_2$ acts on $\Gr_{X^2}$ by sending $(x,y, \mE,\beta)$ to $(y,x,\mE,\beta)$. It is an exercise to check that under $c$, this action corresponds to the flip of the two factors of $\Gr_X\times\Gr_X$.
\end{proof}

These isomorphisms contradict to our intuition. Namely, the fibers of $\Gr_{X^2}\to X^2$ over the diagonal are ``smaller'' than its general fibers! This is purely an infinite dimensional phenomenon, as we all know that the fiber dimension of a family of algebraic varieties goes up under specialisation. 

In fact, there is no contradiction. We have the following result.
\begin{prop}\label{III:closure}
Assume that $\underline G=G\times X$ is constant and $p\nmid \pi_1(G_\der)$. Regard $\Gr_{\leq \mu,X}\times \Gr_{\leq \la,X}|_{X^2-\Delta}\subset \Gr_{X^2}|_{X^2-\Delta}$ via the above identification and let $\Gr_{\leq (\mu,\la)}$ denote the closure in $\Gr_{X^2}$. Then $\Gr_{\leq (\mu,\la)}|_{\Delta}= \Gr_{\leq \mu+\la,X}$.
\end{prop}
So as soon as we restrict to finite dimensional subvarieties, we restore our intuition.  A slightly  weaker version of this result was proved in \cite{Z07} (see also \cite{Z10}) but more or less the same arguments apply here.

Next we consider line bundles. Let $\mL$ be a line bundle on the prestack $\Bun_{\underline G}$ of $\underline G$-torsors on $X$, rigidified at the trivial $\underline G$-torsor. Note that we have the natural forgetful map
\begin{equation}\label{III:I-un}
u_{X^I}: \Gr_{X^I}\to \Bun_{\underline G}.
\end{equation}
Let $\mL_{X^I}$ denote the pullback of $\mL$ to $\Gr_{X^I}$.

\begin{prop}\label{III:line}
Assume that $\underline G=G\times X$ is constant and $G$ simply-connected. Then under the identification given in Proposition \ref{III:Gr2}, there are canonical isomorphisms
\[\Delta^*\mL_{X^2}\cong \mL_{X}, \quad \mL_{X^2}|_{X^2-\Delta}\cong c^*(\mL_X\boxtimes\mL_X|_{X^2-\Delta}).\]
\end{prop}
\begin{proof}
The first isomorphism is clear. We explain the second. Note that the following diagram is commutative by the proof of Proposition \ref{III:Gr2}
\[\xymatrix{
(\Gr_X\times X)|_{X^2-\Delta}\ar^{c^{-1}(\id\times e)}[rr]\ar[dr]&&\Gr_{X^2}|_{X^2-\Delta}\ar[dl]\\
&\Bun_G&
}\]
Therefore, the restriction of $\mL_{X^2}|_{X^2-\Delta}$ to $\Gr_X\times X|_{X^2-\Delta}$ is canonically isomorphic to $\mL_X\boxtimes\mO_{X}|_{X^2-\Delta}$. Using the fact that $\Pic^e(\Gr_G)\simeq \bZ$  which is discrete (so every line bundle algebraically equivalent to zero is trivial) and $\Gr_X$ is connected, we see that for every point of $s$ of $\Gr_X$, the restriction of $\mL_{X^2}|_{X^2-\Delta}$ to $\Gr_X\times\{s\}$ is isomorphic to $\mL_X$. By the same argument, the restriction of $\mL_{X^2}|_{X^2-\Delta}$ to $\{s\}\times \Gr_X$ is also isomorphic to $\mL_X$.

Since $\Gr_X$ is ind-proper over $X$, we can apply the See-Saw principle (cf. \cite[\S 5]{Mum}) to conclude that $\mL_{X^2}|_{X^2-\Delta}$ is canonically isomorphic to $\mL_X\boxtimes\mL_X|_{X^2-\Delta}$, compatible with the rigidifications.
\end{proof}

Note that if $G$ is not simply-connected, Proposition \ref{III:line} cannot hold for arbitrary line bundle on $\Bun_G$\footnote{Although probably one can generalise this proposition to not necessarily simply-connected groups by requiring $\mL$ on $\Bun_G$ to be ``multiplicative'' in appropriate sense.}. In addition, the argument does not work for non-constant group scheme $\underline G$ even the generic fiber of $\underline G$ is simply-connected,  since we crucially use the fact that $\Gr_{X^I}\to X^I$ is ind-proper (in order to apply the See-Saw principle). For most applications, Proposition \ref{det factorization} below can be served as a replacement.

Namely, the construction of the determinant line bundle on $\Gr_{\GL_n}$ in \S~\ref{determinant line} immediately generalises to give determinant line bundles $\mL_{\det, X^I}$ on $\Gr_{\GL_n,X^I}$.

\begin{prop}\label{det factorization}
There are similar isomorphisms as in Proposition \ref{III:line} for the determinant line bundle $\mL_{\det,X^I}$. 
\end{prop}
In particular, given a representation $\rho:\underline G\to \GL_n$, via pullback along $f_\rho: \Gr_{\underline G,X^I}\to\Gr_{\GL_n,X^I}$, one can construct line bundles on $\Gr_{\underline G,X^I}$ together with isomorphisms as in Proposition \ref{III:line}.
\begin{proof}The existence of the first isomorphism is clear. The second isomorphism follows from the canonical isomorphism
\begin{equation}\label{det functor}
\det(M\oplus N)\cong \det(M)\otimes_R \det(N),
\end{equation}
where $M$ and $N$ are two finite projective $R$-modules.
\end{proof}

\begin{rmk}\label{graded factorization}
There is, however, a crucial difference between the isomorphism $\mL_{X^2}|_{X^2-\Delta}\cong c^*(\mL_X\boxtimes\mL_X|_{X^2-\Delta})$ in Proposition \ref{III:line} and the corresponding isomorphism in Proposition \ref{det factorization}. Namely, first note that both $\mL_{X^2}$ and $\mL_{\det,X^2}$ admit a natural $S_2$-action. While the isomorphism in Proposition \ref{III:line} is $S_2$-equivariant, the isomorphism in Proposition \ref{det factorization} is only $S_2$-equivariant up to a sign. The sign appears
because under \eqref{det functor}, the canonical isomorphism $\det(M\oplus N)\cong\det(N\oplus M)$ corresponds to the canonical isomorphism $\det(M)\otimes\det(N)\cong\det(N)\otimes\det(M)$ up to the sign $(-1)^{\on{rk}M\on{rk}N}$. This reflects the fact that taking the determinant line is a symmetric monoidal functor from the category $\on{Proj}_R$ of finite projective $R$-modules (with the tensor product given by direct sums) to the category $\Pic^\bZ_R$ of \emph{graded} invertible $R$-modules. We refer to \cite{OZ1} for further elaborations on this point. 
\end{rmk}

We list two more basic properties of Beilinson-Drinfeld Grassmannians.
\begin{prop}\label{III:property}~
\begin{enumerate}
\item There exists a canonical connection on $\Gr_{X^I}$. I.e. Let $x,y:S\to X^I$ be two points such that $x|_{S_{red}}=y|_{S_{red}}$. Then there is a canonical isomorphism $\Gr_{X^I}\times_{X^I,x}S=\Gr_{X^I}\times_{X^I,y}S$.

\item Beilinson-Drinfeld Grassmannians commute with \'etale base change. I.e., let $U\to X$ be an \'etale morphism. Then there is a canonical isomorphism $\Gr_{\underline G,X^I}\times_{X^I}U^I\cong\Gr_{\underline G_U, U^I}$.
\end{enumerate}
\end{prop}
\begin{proof} For (1), see \cite[\S 2.8]{BD}.
For (2), see \cite[Lemma 3.2]{Z10}.
\end{proof}

We finish this subsection by introducing some variants of Beilinson-Drinfeld Grassmannians. 
First we have  convolution Grassmannians over $X^J$. 

Let $\phi:J\twoheadrightarrow \{1,2,\ldots,n\}$ be a surjective map of finite sets. Let $J_i=\phi^{-1}(i)$. For $x\in X^J$, let $x_i\in X^{J_i}$ be the corresponding component.
We can define the convolution Grassmannian
\begin{equation}\label{III:conv}
\Gr_{\underline G, \phi}=\left\{(x,(\mE_{i},\beta_i)_{i=1,2,\ldots,n})\ \left|\ \begin{split}&x\in X^J(R),\quad \mE_i \mbox{ are } \underline{G}\mbox{-torsor on } X_R, \\
& \beta_i: \mE_{i}|_{X_R-\Ga_{x_i}}\cong \mE_{i-1}|_{X_R-\Ga_{x_i}} \mbox{ where } \mE_0=\mE^0.\end{split}\right.\right\}.
\end{equation}
Sometimes, it is also denoted by $\Gr_{X^{J_1}}\tilde\times\Gr_{X^{J_2}}\tilde\times\cdots\tilde\times\Gr_{X^{J_n}}$ since there are projections
\[\pr_i:\Gr_{X^{J_1}}\tilde\times\Gr_{X^{J_2}}\tilde\times\cdots\tilde\times\Gr_{X^{J_i}}\to \Gr_{X^{J_1}}\tilde\times\Gr_{X^{J_2}}\tilde\times\cdots\tilde\times\Gr_{X^{J_{i-1}}}\]
realising $\Gr_{X^{J_1}}\tilde\times\Gr_{X^{J_2}}\tilde\times\cdots\tilde\times\Gr_{X^{J_i}}$ as the twisted product of $\Gr_{X^{J_1}}\tilde\times\Gr_{X^{J_2}}\tilde\times\cdots\tilde\times\Gr_{X^{J_{i-1}}}$ and $\Gr_{X^{J_i}}$ (see \S\ \ref{Intro:torsor} for the definition of the twisted product). Namely, there is an $(L^+\underline G)_{X^{J_n}}$-torsor  
\begin{equation}\label{torsor E}
\bE\to \Gr_{X^{J_1}}\tilde\times\Gr_{X^{J_2}}\tilde\times\cdots\tilde\times\Gr_{X^{J_{n-1}}}\times X^{J_n}
\end{equation} classifying a point $(x,(\mE_{i},\beta_i)_{i=1,2,\ldots,n-1})\in \Gr_{X^{J_1}}\tilde\times\Gr_{X^{J_2}}\tilde\times\cdots\tilde\times\Gr_{X^{J_{n-1}}}$, a point $y\in X^{J_n}$, and a trivialisation of $\mE_{n-1}$ along $\hat{\Ga}_y$. Then the twisted product is constructed using this torsor. 

There is also the convolution map
\[m_{\phi}: \Gr_{\underline G, \phi}\to \Gr_{X^J}\]
sending $(x,(\mE_{i},\beta_i)_{i=1,2,\ldots,n})$ to $(x, \mE_n, \beta_1\cdots\beta_n)$. Using an argument similar to Proposition \ref{III:Gr2}, one shows that
\begin{lem}\label{generic conv}
The map $m_\phi$ restricts to an isomorphism over $X^{(\phi)}$.
\end{lem}

Next we have a symmetrized version of Beilinson-Drinfeld Grassmannians. We refer to \cite[Tag01WQ, Tag056P]{St}  for the definition of (relative) effective Cartier divisors.
Recall that given $d$, the $d$th symmetric product $X^{(d)}=X^d/\!\!/S_d$ classifies effective Cartier divisors on $X$. I.e. for every $k$-algebra $R$, $X^{(d)}(R)$ is the set of relative effective Cartier divisor on $X_R$. 
 Then we can define the symmetrized version of Beilinson-Drinfeld Grassmannians $\Gr^{(d)}\to X^{(d)}$ as follows.
\begin{equation}\label{sym BD Grass}
\Gr^{(d)}=\left\{(D,\mE,\beta)\ \left|\ \begin{split}& D\subset X_R \mbox{ a relative effective Cartier divisor of degree } d, \\
&  \mE \mbox{ a } \underline{G}\mbox{-torsor on } X_R, \beta: \mE|_{X_R-D}\cong \mE^0|_{X_R-D} \mbox{ a trivialisation}\end{split}\right.\right\}.
\end{equation}
It is easy to see that the natural map
\[
\Gr_{X^d}\to \Gr^{(d)}\times_{X^{(d)}}X^d
\]
is an isomorphism.
Given a coweight $\mu$, we have a closed subvariety $\Gr^{(d)}_{\leq \mu}$ and the natural map
\[
\Gr_{\leq (\mu,\ldots,\mu)}\to\Gr^{(d)}_{\leq \mu}\times_{X^{(d)}}X^d
\]
is an isomorphism.
The the convolution map induces $\Gr_{\leq \mu}\tilde\times\cdots\tilde\times\Gr_{\leq \mu}\to \Gr_{\leq \mu}^{(d)}$.

\begin{ex}\label{momega1BD}
Let $G=\GL_n$, and $\omega_1=(1,0,\ldots,0)$. Then
\[\Gr^{(m)}_{\leq \omega_1}=\left\{ \mE\subset \mO_{X_R}^n\mid \mO_{X_R}^n/\mE \mbox{ is a projective } R\mbox{-module of rank } m \right\}.\]
We further assume that $X=\bA^1=\Spec k[t]$. Then Example \ref{momega1} and \ref{conv v.s. springer} have the global counterparts. First, Formula \eqref{nil to Gr} in fact defines
\[\gl_n\to \Gr_{\leq \omega_1}^{(n)}.\]
We leave it as an exercise to construct the counterpart of \eqref{sat and springer} in the global setting. In addition, the global analogue of
\eqref{sat and springer2} is
\[\begin{CD}
\widetilde{\gl}_n@>>> \Gr_{\omega_1,X}\tilde\times\cdots\tilde\times\Gr_{\omega_1,X}@>>> X^n\\
@VVV@VVV@VVV\\
\gl_n@>>>\Gr^{(m)}_{\leq \omega_1}@>>>X^{(n)},
\end{CD}\]
where $\widetilde{\gl}_n\to\gl_n$ is the Grothendieck-Springer resolution (\cite[\S~1]{Yu}), and the left square is Cartesian. We can identify $X^n\cong \bA^n$ with the standard Cartan subalgebra $\frakt_n$ of $\gl_n$ and $X^{(n)}$ with $\frakt_n/\!\!/S_n$. Then the composition of the two maps in the bottom row can be identified with the Chevalley map that sends a matrix to its characteristic polynomial. This diagram first appeared in \cite{MVy}.
\end{ex}

\begin{rmk}
There also exists a version of Beilinson-Drinfeld affine Grassmannians for $p$-adic groups, as recently constructed by P. Scholze \cite{Sch}. However, this is beyond the scope of our notes.
\end{rmk}

\subsection{Factorisation property.}
Let us summarise the most important properties of Beilinson-Drinfeld Grassmannians $\{\Gr_{X^I}\}$ into the following theorem, which vastly generalises Proposition \ref{III:Gr2}. But the proof remains similar.
\begin{thm}\label{III:factorization} Let $\underline G$ be a smooth affine group scheme over $X$.

(i) For every $I$, there is a unit section $e_I: X^I\to \Gr_{X^I}$ given by the trivial $\underline G$-torsor.

(ii) For every $\phi: J\twoheadrightarrow I$, let $\Delta(\phi):X^I\to X^J$ be the corresponding diagonal embedding sending $\{x_i, i\in I\}$ to $\{y_j, j\in J\}$ with $y_j=x_{\phi(j)}$. Then
there is a canonical isomorphism
\begin{equation}\label{BD restriction}
\Delta(\phi):\Gr_{X^I}\cong \Gr_{X^J}\times_{X^J,\Delta(\phi)}X^I
\end{equation}
By abuse of notation, the composition $\Gr_{X^I}\cong \Gr_{X^J}\times_{X^J,\Delta(\phi)}X^I\to \Gr_{X^J}$ is also denoted by $\Delta(\phi)$.

(iii) Given $\phi: J\twoheadrightarrow  I$, let $J=\sqcup_{i\in I}J_i$ denote the partition of $J$, and $X^{(\phi)}\subset X^J$ denote the open subset of those $\{x_j, j\in J\}$ such that $x_{j}\cap x_{j'}=\emptyset$ if $\phi(j)\neq \phi(j')$. Then there is a canonical isomorphism (usually called the factorization isomorphism)
\begin{equation}\label{BD factorization}
c_\phi: \Gr_{X^J}\times_{X^J}X^{(\phi)}\cong (\prod_{i\in I}\Gr_{X^{J_i}})\times_{X^J}X^{(\phi)}.
\end{equation}
These maps satisfy the following conditions:
\begin{enumerate}
\item $\Delta(\phi)\circ e_I=e_J|_{\Delta(\phi)}$ and $c_\phi (e_J|_{X^{(\phi)}})=(\prod_{i\in I} e_{J_i})|_{X^{(\phi)}}$.
\item Given $K\stackrel{\phi}{\twoheadrightarrow}  J\stackrel{\psi}{\twoheadrightarrow}  I$, $\Delta(\phi)\Delta(\psi)=\Delta(\psi\phi)$.
\item Given $K\stackrel{\phi}{\twoheadrightarrow}  J\stackrel{\psi}{\twoheadrightarrow}  I$, the following diagram is commutative
\[\begin{CD}
\Gr_{X^J}|_{X^{(\psi)}}@>\Delta(\phi)|_{X^{(\psi)}}>>\Gr_{X^K}|_{X^{(\psi\phi)}}\\
@Vc_{\psi} VV@VV c_{\psi\phi}V\\
(\prod_i\Gr_{X^{J_i}})|_{X^{(\psi)}}@>(\prod_{i}\Delta(\phi_i))|_{X^{(\psi)}}>> (\prod_i \Gr_{X^{K_i}})|_{X^{(\psi\phi)}},
\end{CD}\] 
where $\phi_i:K_i\to J_i$ is the map over $i\in I$.
\item Given $K\stackrel{\phi}{\twoheadrightarrow}  J\stackrel{\psi}{\twoheadrightarrow}  I$, $(\prod_{i\in I}c_{\phi_i})(c_{\psi\phi}|_{X^{(\phi)}})=c_{\phi}$.
\end{enumerate}
\end{thm}
\begin{ex}\label{FA ex}
Let us see why this theorem encodes all the information contained in Proposition \ref{III:Gr2}. First, applying \eqref{BD restriction} to the map $\phi:\{1,2\}\to \{1\}$ recovers the isomorphism $\Delta: \Gr_X\cong \Gr_{X^2}|_{\Delta}$, and applying it to the flip $\sigma:\{1,2\}\to \{1,2\}$ gives the canonical isomorphism $\Delta(\sigma):\Gr_{X^2}\to \Gr_{X^2}$ covering the flip map of $X^2\to X^2$.
On the other hand, applying \eqref{BD factorization} to the map
$\id: \{1,2\}\to \{1,2\}$, 
recovers the isomorphism $c:\Gr_{X^2}|_{X^2-\Delta}\cong \Gr_X\times\Gr_X|_{X^2-\Delta}$. 
Applying Property (3) to $\{1,2\}\stackrel{\sigma}{\to}\{1,2\}\stackrel{\id}{\to}\{1,2\}$ and Property (4) to $\{1,2\}\stackrel{\id}{\to}\{1,2\}\stackrel{\sigma}{\to}\{1,2\}$ implies that the isomorphism $c$ is compatible with $\Delta(\sigma)$ and the natural flip of $\Gr_{X}\times\Gr_X$, as claimed in Proposition \ref{III:Gr2}.
\end{ex}

The generalisation of Proposition \ref{III:line} is
\begin{thm}\label{III:factor line}
Assume that $G$ is simply-connected and $\underline G=G\times X$ is constant.
Let $\mL$ be a rigidified line bundle on $\Bun_{G}$, and let $\mL_{X^I}$ denote its pullback to $\Gr_{X^I}$. Then under the isomorphisms in Theorem \ref{III:factorization}, there are canonical isomorphisms,

(i) $e_I^*\mL_{X^I}\cong \mO_{X^I}$; (ii) $\Delta(\phi)^*\mL_{X^J}\cong \mL_{X^I}$; (iii) $\mL_{X^J}|_{X^{(\phi)}}\cong c_\phi^*(\boxtimes \mL_{X^{J_i}}|_{X^{(\phi)}})$.

In addition, these isomorphisms are compatible in the sense that the analogue statements (1)-(4) in Theorem \ref{III:factorization} hold. 
\end{thm}

The above structures and properties of Beilinson-Drinfeld Grassmannians are also shared by some other geometric objects so it is worth to summarise them into the following definition.
\begin{dfn}\label{FA space}
A factorization space $\mY$ over $X$ is an assignment for each non-empty  finite set $I$ a space $\mY_I\to X^I$ together with a collection of maps as in (ii)-(iii) of Theorem \ref{III:factorization}, satisfying the compatibility conditions (2)-(4) of Theorem \ref{III:factorization}. A factorization space is called unital if there exist a collection of maps as in (i) of Theorem \ref{III:factorization}, satisfying the compatibility condition (1) of Theorem \ref{III:factorization}.

A line bundle $\mL$ on a unital factorization space $\mY$ is an assignment for each $I$  a line bundle $\mL_I$ on $\mY_I$ together with a collection of isomorphisms as in (i)-(iii) of Theorem \ref{III:factor line} satisfying the natural compatibility conditions as in Theorem \ref{III:factor line}.
\end{dfn}

\begin{rmk}
Factorization spaces (resp. unital factorization spaces) are called chiral semigroups (resp. chiral monoids) in \cite[\S\ 3.10.16]{BD0}. We refer to \emph{loc. cit.} for  further discussions and examples.
\end{rmk}

\begin{rmk}\label{det factorization2}
One can also generalise Proposition \ref{det factorization} to say that the collection $\mL_{\det, X^I}$ form a graded factorization line bundle on $\Gr_{\GL_n, X^I}$. There are also a collection of isomorphisms as in (i)-(iii) of Theorem \ref{III:factor line} between them. But the compatibility conditions between these isomorphisms must be modified. More precisely, the analogous Statement (3) in Theorem \ref{III:line} must be modified. 
\end{rmk}

\subsection{The Ran space.}
It will be conceptually concise to organise all $\{\Gr_{X^I}\}$ into a single object and reformulate the above theorem using this object. 

\begin{dfn}
We define the Ran space $\Ran(X)$ of $X$ as the presheaf that assign every $R$ the set of finite non-empty subsets of $X(R)$.
\end{dfn}

Note that for any finite non-empty set $I$, there is a tautological map $X^I\to \Ran(X)$, which induces an isomorphism
\[\underrightarrow{\lim}_I X^I\to \Ran(X),\]
where the colimit is taken in the category of presheaves over the category of all finite non-empty sets with surjective transition maps.  Since the colimit is taken over a non filtered category, $\Ran(X)$ is not a space (although each $X^I$ is). In fact, it is not a sheaf even for the \'etale topology.

For $x=\{x_i,i\in I\}\in \Ran(X)(R)$, let $\Ga_x$ denote the union of all graphs $\Ga_{x_i}$ (as closed subschemes in $X_R$).

\begin{dfn}\label{dfn Ran Grass}
We define the Ran Grassmannian $\Gr_{\underline G,\Ran(X)}$ as
\[\Gr_{\underline G,\Ran(X)}=\left\{(x,\mE,\beta)\ \left|\ \begin{split}&x\in \Ran(X)(R),\quad \mE \mbox{ is a } \underline{G}\mbox{-torsor on } X_R, \\
& \beta: \mE|_{X_R-\Ga_x}\cong \mE^0|_{X_R-\Ga_x} \mbox{ is a trivialisation}\end{split}\right.\right\}.\]
\end{dfn}
Note that we do not claim any representability of $\Gr_{\underline G,\Ran(X)}$ as it is \emph{not} representable. However, 
there is a natural projection
\[q_{\Ran}: \Gr_{\underline G,\Ran(X)}\to \Ran(X),\]
and the unit section
\[e_{\Ran}:\Ran(X)\to  \Gr_{\underline G,\Ran(X)},\]
that are relatively representable.
In fact, we have
\[\Gr_{\underline G,X^I}=X^I\times_{\Ran(X),q_{\Ran}}\Gr_{\underline G,\Ran(X)}.\]
If the curve $X$ and the group scheme $\underline G$ are clear from the context, we write $\Gr_{\underline G,\Ran(X)}$ by $\Gr_{\Ran(X)}$ or by $\Gr_{\Ran}$ for brevity.

Here is a concise way to encode the factorization structure of the Beilinson-Drinfeld Grassmannians. First, the Ran space $\Ran(X)$ has a semigroup structure given by union of points
\[\on{union}: \Ran(X)\times \Ran(X)\to \Ran(X),\quad (\{x_i\},\{x_j\})\mapsto \{x_i,x_j\}.\]
Let $(\Ran\times \Ran)_{disj}\subset \Ran\times\Ran$ denote the open subpresheaf consisting of those $\{x_i\}$ and $\{x_j\}$ with $\{x_i\}\cap \{x_j\}=\emptyset$, and similarly one defines $(\Ran\times \Ran\times \Ran)_{disj}$. Then Theorem \ref{III:factorization} can be reformulated as
\begin{thm}\label{Ran Grass factorization}
There is a canonical isomorphism
\[a: \Gr_{\Ran}\times_{\Ran}(\Ran\times\Ran)_{disj}\simeq (\Gr_{\Ran}\times\Gr_{\Ran})|_{(\Ran\times\Ran)_{disj}},\]
compatible with the unit section $e_{\Ran}:\Ran\to\Gr_{\Ran}$ and satisfying a natural cocycle condition over $(\Ran\times\Ran\times \Ran)_{disj}$. 
\end{thm}
We leave it as an exercise to verify that Theorem \ref{III:factorization} and Theorem \ref{Ran Grass factorization} are equivalent.

\subsection{Rigidified line bundles on $\Gr_{\Ran}$.}\label{line bundle on BD Grass}
The Ran Grassmannian $\Gr_{\Ran}$ is a fascinating geometric object. As we shall see in the next section, it encodes many information of the moduli of $\underline G$-bundles on $X$. Therefore, it is important to study the geometry of $\Gr_{\Ran}$.
Here we illustrate some methods by one example, namely, we will calculate the Picard groupoid of rigidified line bundles on $\Gr_{\underline{G}, \Ran(X)}$, where $\underline{G}\to X$ is a simple, simply-connected group scheme over $X$. 

For a morphism $X\to S$ of prestacks equipped with a section $e:S\to X$, let $\Pic^e(X/S)$ denote the Picard groupoid of line bundles on $X$, equipped with a trivialisation along $e$.
Here is the theorem.
\begin{thm}\label{Rigid line}Assume that $G$ is simple and simply-connected and $\underline G=G\times X$ is constant.
Then for any point $x\in X$, the map $\Gr_x\to \Gr_{\Ran}$ induces an isomorphism
$\Pic^e(\Gr_{\Ran}/\Ran)\simeq \Pic^e(\Gr_x)$.
\end{thm}

Note that we have
\[\Pic^e(\Gr_{\Ran}/\Ran)=\underrightarrow\lim\ \Pic^e(\Gr_{X^I}/X^I),\]
and the Picard groupoid $\Pic^e(\Gr_{X^I}/X^I)$ is discrete. In addition, it is the group of global sections of an \'etale sheaf on $(X^I)_{et}$
\[\mP ic^e(\Gr_{X^I}/X^I)(U)= \Pic^e((\Gr_{X^I})_U/U).\]
We will first study $\mP ic^e(\Gr_{X^I}/X^I)$.

\begin{lem}
Let $X$ be a smooth curve, and $\underline G$  be a connected reductive group scheme over $X$ that is fiberwise connected, and simply-connected.
Then there is a canonical isomorphism 
\[\bZ_X\simeq \mP ic^e(\Gr_X/X),\]
where $\bZ_X$ is the constant sheaf $\bZ$ on $X$.
In particular, for every point $x\in X$, the stalk $\mP ic^e(\Gr_X/X)_x=\Pic^e(\Gr_x)=\bZ$.
\end{lem}
\begin{proof}
We first treat the case when $\underline G=G\times X$ is constant and there is an \'etale morphism $X\to \bA^1$. Then since affine Grassmannians commute with \'etale base change (see Proposition \ref{III:property}), we reduce to treat the case $X=\bA^1$. In this case, there is a canonical isomorphism $\Gr_{\bA^1}\cong \Gr_0\times {\bA^1}$. Namely, there exists a global coordinate on $\bA^1$ so the $\Aut(D)$-torsor $\widehat{\bA^1}\to \bA^1$ is trivial. Then the lemma follows. Note that the isomorphism $\bZ\cong \Pic^e(\Gr_{\bA^1})$ is canonical since there is an ample generator $\mO(1)$. 

In general, we have an \'etale map $f:U\to X$ such that $\underline G\times_XU\cong G\times U$ for some constant group $G$ and there is an \'etale map $U\to \bA^1$. The canonicity of the isomorphism $\bZ_U\cong \mP ic^e(\Gr_U/U)$ implies that it descends to a canonical isomorphism $\bZ_X\to \mP ic^e(\Gr_X/X)$.
\end{proof}

We write by $\mO(n)_X$ the line bundle on $\Gr_X$ corresponding to $n\in \bZ_X$ under the above isomorphism.

The crucial step is to understand $\mP ic^e(\Gr_{X^2}/X^2)$. Note that since $\Gr_X\to X$ is ind-projective, we have the canonical isomorphism
\[ \mP ic^e(\Gr_X/X)\boxtimes \mP ic^e(\Gr_X/X)\cong\mP ic^e(\Gr_X\times\Gr_X/X\times X)\]
by the See-Saw principle (see the proof of Proposition \ref{III:line}). In particular, 
\[\mP ic^e(\Gr_{X^2}/X^2)|_{X^2-\Delta}\cong \mP ic^e(\Gr_X/X)\boxtimes \mP ic^e(\Gr_X/X)|_{X^2-\Delta}\cong(\bZ_X\boxtimes\bZ_X)|_{X^2-\Delta}.\]
\begin{lem}\label{Pic X2}
The above isomorphism extends to an exact sequence
\[0\to \mP ic^e(\Gr_{X^2}/X^2)\to \bZ_X\boxtimes\bZ_X\to \Delta_*\bZ_X\to 0, \]
where the map $\bZ_X\boxtimes\bZ_X\to \Delta_*\bZ_X$ is given by $(a,b)\mapsto a-b$.
\end{lem}
\begin{proof}
First, we show that $ \mP ic^e(\Gr_{X^2}/X^2)\to  \bZ_X\boxtimes\bZ_X$ is injective. I.e. it does not admit non-zero sections supported along the diagonal. Let $(x,x)$ be a point on the diagonal. As above, by working \'etale locally on $X$ we can assume that $\underline G=G\times X$ is constant, and that $\Gr_{X}\to X$ is isomorphic to $\Gr\times X\to X$.  Let $V\to X\times X$ be an \'etale open neighbourhood of $(x,x)$. For brevity, we write $V\cap\Delta$ by $\Delta$ in the sequel.
Let $\mL$ be a rigidified line bundle on $\Gr_V:=\Gr_{X^2}\times_{X^2}V$, trivial on $\Gr_{V-\Delta}:=\Gr_V|_{V-\Delta}$. We need to show that $\mL$ is trivial. It is enough to show that its restriction to the Schubert variety $\Gr_{\leq(\la,\mu)}|_V\subset \Gr_V$ is trivial for all $(\la,\mu)$ (see Proposition \ref{III:closure} for the definition of $\Gr_{\leq(\la,\mu)}$). For brevity, we write $\mL|_{\Gr_{\leq(\la,\mu)}}$ by $\mL$.

By the definition of the rigidified line bundle, $\mL$ is equipped with a trivialisation $\theta: \mO_{V}\cong e^*\mL$. In addition, $\mL|_{\Gr_{V-\Delta}}$ is trivial. Since the pullback along  the map $\Gr_{\leq(\la,\mu)}|_{V-\Delta}\to V-\Delta$ induces an isomorphism of spaces of regular functions, there is a unique trivialisation 
$$s:  \mO_{\Gr_{\leq(\la,\mu)}|_{V-\Delta}}\cong \mL|_{V-\Delta},$$ 
such that $e^*(s)=\theta|_{V-\Delta}$. By Proposition \ref{III:closure} and  Theorem \ref{singularity}, the scheme $\Gr_{\leq(\la,\mu)}$ is Cohen-Macaulay. Therefore, $\mL((n-1)\Delta)$ is a subsheaf of $\mL(n\Delta)$. Then
there is an integer $n$ such that 
$$s\in \Gamma(\Gr_{\leq(\la,\mu)},\mL(n\Delta))\setminus \Gamma(\Gr_{\leq(\la,\mu)},\mL((n-1)\Delta)).$$ Then we regard $s$ as a regular section of $\mL(n\Delta)$. 
Note that then $e^*(s)$ coincides with $\mO_V\simeq e^*\mL\to e^*\mL(n\Delta)$ as a regular section of $(e^*\mL)(n\Delta)$. In particular, $n\geq 0$. But if $n>0$, $s$ would vanish at $e(\Delta)$ and therefore must vanish on the entire $\Gr_{\leq (\la,\mu)}|_\Delta$ (since the zero loci of $s$ is pure of codimension one and is contained in $\Gr_{\leq (\la,\mu)}|_\Delta$ which is an integral scheme again by Proposition \ref{III:closure}). This would imply that $s\in \Gamma(\Gr_{\leq(\la,\mu)},\mL((n-1)\Delta))$, which is contradiction. Therefore $n=0$. By the same reasoning $s$ does not vanish along $e(\Delta)$ and therefore does not vanish anywhere. Therefore, $s$ induces a trivialisation of $\mL$. This finishes the proof of the injectivity.

Now let $\mL$ be a rigidified line bundle on $\Gr_{X^2}$, whose restriction to $\Gr_{X^2}|_{X^2-\Delta}$ is $\mO(a)_X\boxtimes\mO(b)_X|_{X^2-\Delta}$. Then again by the See-Saw principle,  its pullback $m^*\mL$ along the convolution map $m: \Gr_X\tilde\times\Gr_X\to \Gr_{X^2}$ is of the form $\pr_1^*\mO(a)_X\otimes \mL'$, where $\mL'$ is some line bundle on $\Gr_X\tilde\times\Gr_X$ whose restriction to each $\pr_1^{-1}(y)(\simeq \Gr_X)$ is isomorphic to $\mO(b)_X$. Over a point on the diagonal, the global convolution map restricts to the local convolution map (see \eqref{conv m0}). Then an easy computation of the pullback of line bundles on $\Gr$ along \eqref{conv m0} implies that $a=b$. This shows that the composition $\Pic^e(\Gr_{X^2}/X^2)\to \bZ_X\boxtimes\bZ_X\to \Delta_*\bZ_X$ is zero.

Finally, we need to show that $\Pic^e(\Gr_{X^2}/X^2)$ maps surjectively to $\ker(\bZ_X\boxtimes\bZ_X\to\Delta_*\bZ_X)$. Again, it is enough to consider the stalks along the diagonal. So as before we can pass to an \'etale neighbourhood of $(x,x)$ and therefore assume that $X=\bA^1\subset \bP^1$ and $\underline G=G\times X$ is constant. Then we can conclude by pulling back of the line bundle $\mO(1)$ on $\Bun_G(\bP^1)$ (see Lemma \ref{O1 on Bun}) to $\Gr_{X^2}$. 
\end{proof}

By the same argument, one can show that $\mP ic^e(\Gr_{X^I}/X^I)$ fits into the following exact sequence
\[0\to \mP ic^e(\Gr_{X^I})\to \boxtimes_{i\in I}\bZ_X\to \bigoplus_{\phi:I\twoheadrightarrow J/\sim}\Delta(\phi)_*\boxtimes_{j\in J}\bZ_X.\]
Here $|J|=|I|-1$ and $\phi\sim \phi'$ if there is an isomorphism $J\cong J'$ intertwining $\phi$ and $\phi'$. In other words, the sum is taken over all partitions of $I$ into $|I|-1$ subsets. The map $d_{\phi}:\boxtimes_{i\in I}\bZ_X\to \Delta(\phi)_*\boxtimes_{j\in J}\bZ_X$ is given as follows: one can write $I=I'\cup\{1,2\}\to J=I'\cup\{1\}$ and the map $d_\phi$ is the identity on $\boxtimes_{i\in I'}\bZ_X$-factor and is the map $\bZ_X\boxtimes\bZ_X\to \Delta_*\bZ_X$ in Lemma \ref{Pic X2} on the factors corresponding to $\{1,2\}\to\{1\}$.

\begin{cor}For every point $x\in X$, and every $\phi: J\twoheadrightarrow I$, the maps $\Pic^e(\Gr_X)\to \Pic^e(\Gr_x)$ and $\Pic^e(\Gr_{X^J}/X^J)\to\Pic^e(\Gr_{X^I}/X^I)$ are an isomorphisms. 
\end{cor}
Theorem \ref{Rigid line} then follows.

\section{Lecture IV: Applications to the moduli of $G$-bundles}\label{Lecture IV}
In this section, we explain how to study moduli of $G$-bundles via affine Grassmannians. The basic tools are the
uniformization theorems. We assume that $k$ is algebraically closed throughout this section.

\subsection{One point uniformization.}\label{one point uniformization}
Let $X$ be a smooth projective curve over $k$. Let $\underline G$ be a fiberwise connected smooth affine group scheme over $X$. Then the moduli space $\Bun_{\underline G}$ of $\underline G$-torsors on $X$ is an algebraic stack, locally of finite presentation over $k$ (e.g. see \cite[Proposition 1]{He}). 
\begin{ex} Let $\underline G|_{X-\{x_1,\ldots,x_n\}}\cong G\times (X-\{x_1,\ldots,x_n\})$ and $\underline G|_{D_{x_i}}$ be an appropriate parahoric group scheme at every $x_i$, then $\Bun_{\underline G}$ is the moduli stack of parabolic bundles of $G$ (with parabolic structures at $\{x_1,\ldots,x_n\}$).
\end{ex}

By the definition of $\Gr_{\underline G,x}$ (cf. \eqref{aff glob}), we have a natural morphism
\begin{equation}\label{one pt uni map}
u_x: \Gr_{\underline G,x}\to \Bun_{\underline G}.
\end{equation}
Let $X^*=X-\{x\}$, and let $\underline G^{X^*}$ denote the space of sections. I.e. 
\begin{equation}\label{Gout}
\underline G^{X^*}(R)=\Gamma(X_R^*,\underline G).
\end{equation}
\begin{lem}\label{rep of Gout}
The space $\underline G^{X^*}$ is represented by an group ind-scheme.
\end{lem}
\begin{proof}Again, one reduces to prove it for $\underline G=\GL_n$, which is easy (e.g. see \cite[(1.9)]{BL}).
\end{proof}
 By restriction to $D_x^*\subset X^*$, we may regard $\underline G^{X^*}$ as a subgroup of $L\underline G_x$. 
Then $\underline G^{X^*}$ acts transitively along the fibers of $u_x$ and therefore $u_x$ factors through
\begin{equation}\label{1pt uniform}
[\underline G^{X^*}\backslash\Gr_{\underline G,x}]\to \Bun_{\underline G}.
\end{equation}

This map is not necessarily surjective, as can be seen in the case $G=\bG_m$ (not every line bundle on the open curve $X^*$ is trivial). However, it is surjective if $\underline G=G\otimes X$ is semisimple. In fact, we have the following stronger result.
\begin{thm}
Let $\underline G$ be a smooth fiberwise connected affine group scheme over $X$
with semisimple generic fiber  over the rational function field $k(X)$. Then \eqref{1pt uniform} is an isomorphism.
\end{thm}

At the level of $k$-points, it says that 
$$\Bun_{\underline G}(k)= \Gamma\backslash G(F_x)/G(\mO_x),$$ where $\Gamma=\underline G^{X^*}(k)=\Gamma(X^*,\underline G)$. This is due to Harder. 
The $\SL_n$ case is due to Beauville-Laszlo (cf. \cite{BL}). 
For a constant semisimple group scheme, this is due to Drinfeld-Simpson (cf. \cite{DS}). It was proved in \emph{loc. cit.} that if $p\nmid|\pi_1(G)|$, then the $\underline G^{X^*}$-torsor $\Gr_{\underline G}\to \Bun_{\underline G}$ is \'etale locally trivial. In its full generality, the theorem is due to Heinloth (cf. \cite{He}).
Because of this theorem, we can think \eqref{one pt uni map} as a uniformization map.

\begin{ex}We illustrate the theorem in the simplest example $G=\SL_2$. In this case, $\Bun_G$ classifies rank two vector bundles on $X$ with the trivial determinant. We explain why every such bundle $\mE$ on $X$ becomes trivial on $X^*$. First, there exists a line bundle $\mL$ (negative enough) fitting into the following short exact sequence $0\to \mL\to \mE\to \mL^{-1}\to 0$. Now, since $X^*$ is affine, $\on{H}^1(X^*,\mL^{-2})=0$, so $\mE|_{X^*}\cong (\mL\oplus\mL^{-1})|_{X^*}$. 
Next, we can find a large enough integer $n$, a global section $s_1$ of $\mL(nx)$, and $s_2$ of $\mL^{-1}(nx)$ such that $s_1$ and $s_2$ do not have common zeros on $X^*$. Let $s=(s_1,s_2)|_{X^*}:\mO_{X^*}\to (\mL\oplus\mL^{-1})|_{X^*}$. Then $s$ vanishes nowhere on $X^*$, and therefore we have $0\to \mO_{X^*}\to (\mL\oplus\mL^{-1})|_{X^*}\to \mO_{X^*}\to 0$. Again, by vanishing of $\on{H}^1(X^*,\mO_{X^*})$, we have $\mE|_{X^*}\cong (\mL\oplus\mL^{-1})|_{X^*}\cong \mO_{X^*}^2$.
\end{ex}

For simplicity, we assume that $\underline G=G\times X$ in the rest of this subsection. In this case, write $\underline G^{X^*}$ by $G^{X^*}$.
\begin{cor}
There is a canonical isomorphism $c:\pi_0(\Bun_G)\cong \pi_1(G)$. In particular, if $G$ is simply-connected, then $\Bun_G$ is connected.
\end{cor}
\begin{ex} When $G=\GL_n$, the map $\Bun_G\to \pi_0(\Bun_G)\cong\pi_1(G)\cong \bZ$ is given by attaching a vector bundle on $X$ its Chern class.\end{ex}
\begin{proof}
The point is to show that $G^{X^*}$ is contained in the neutral connected component of $LG$, and then it follows $\pi_0(\Bun_G)\cong \pi_0(\Gr_{x})$ so we can use Theorem \ref{components}. We refer to \cite[(1.2)]{BLS} for the proof of the claim.
\end{proof}
\begin{rmk}(i) Note that $G^{X^*}$ itself may not be connected. In fact, it was proved in \cite{BLS} that $\pi_0(G^{X^*})\simeq H^1(X^*,\pi_1(G))$ (in the case $\cha\ k=0$). 

(ii) For the generalisation of this result to non-constant group schemes, we refer to \cite{He}.
\end{rmk}

\subsection{Line bundles and conformal blocks.}
The uniformization theorem is in particular useful for the study of line bundles on $\Bun_G$. Here is the main theorem. 
\begin{thm}\label{uniformize line}
Assume that $G$ is simply-connected. Then the pullback induces an isomorphism $u_x^*:\Pic^e(\Bun_G)\to \Pic^e(\Gr_x)$.
\end{thm}
In particular, if $G$ is simple, then $\Pic^e(\Bun_G)\cong\bZ$. The ample generator $\mO(1)$ of $\Gr_G$ descends to a generator of $\Pic^e(\Bun_G)$, still denoted by $\mO(1)$.

The injectivity of the map is easy. By the uniformization theorem, rigidified line bundles on $\Gr_G$ are the same as the $G^{X^*}$-equivariant line bundles on $\Gr_G$. Therefore, the corollary is equivalent to the statement that any character $G^{X^*}\to \bG_m$ is trivial. This last statement was proved in \cite[Corollary 5.2]{LS}. However, since the geometry of $G^{X^*}$ is still very complicated, it is not easy to deduce the surjectivity part of the theorem from the one point uniformization theorem. Theorem \ref{uniformize line} was proved in \cite{LS} for classical groups and $G_2$ by explicitly constructing the line bundle on $\Bun_G$ whose pullback is $\mO(1)$ (similar to the discussion in Remark \ref{generator}). Then the full theorem was proved by Sorger \cite{So0} when $\cha\ k=0$ (using some non-vanishing result of the space of conformal blocks) and by Faltings \cite[Theorem 17]{Fa} in general. 
In the next subsection, we will deduce this theorem from another uniformization theorem.

Now assume that $ k=\bC$ (or any algebraically closed field of characteristic zero). Let $\mL$ be a line bundle on $\Bun_G$. Global sections of $\mL$ are usually called non-abelian theta functions (because when $G=\bG_m$ they are classical theta functions). Let $\frakg^{X^*}=\frakg\otimes\mO_X(X^*)$.  This is the Lie algebra of $G^{X^*}$.
\begin{lem}
The Kac-Moody central extension (i.e. the Lie algebra of \eqref{central extension}) splits uniquely over $\frakg^{X^*}$ and therefore $\frakg^{X^*}$ embeds into the Kac-Moody algebra $\frakg_\Ga$ (cf. \S~\ref{Kac-Moody}) as a subalgebra.
\end{lem}
\begin{proof} By the explicit formula in Remark \ref{2-cocycle} and the residue theorem, the central extension splits over $\frakg^{X^*}$. Since $[\frakg,\frakg]=\frakg$, such splitting is unique.
\end{proof}

In physics, the integrable representation $L_m(\frakg_\Ga)$  is the space of vacuum states of the WZW conformal field theory, and $\Hom_{\frakg^{X^*}}(L_m(\frakg_\Ga),k)$ is called the space of conformal blocks. Combining the above uniformization theorem and the Borel-Weil type Theorem \ref{Borel-Weil}, we obtain
\begin{cor}Assume that $G$ is simple and simply-connected.
The space of level $m$ non-abelian theta functions on $\Bun_G$ is isomorphic to the dual of the conformal block. I.e.
\[\Gamma(\Bun_G,\mO(m))^*\simeq L_m(\frakg_\Ga)/\frakg^{X^*} L_m(\frakg_\Ga),\]
where $\mO(m)=\mO(1)^{\otimes m}$.
\end{cor}
The dimension of the space of conformal blocks is calculated by the so-called Verlinde formula. As a consequence, one also obtains the dimension of the space of non-abelian theta functions on $\Bun_G$. For example, we have
\begin{cor}Let $G=\SL_n$. Then
\[\dim \Gamma(\Bun_G,\mO(m))= (\frac{n}{n+m})^g \sum_{\substack{S\subset \{1,\ldots,m+n\}\\|S|=n}}\prod_{\substack{s\in S\\ t\not\in S}}\left|2\sin\pi\frac{s-t}{m+n}\right|^{g-1}.\]
\end{cor}
In particular, the right hand side is a positive integer! See \cite[Corollary 8.6]{BL}.
\begin{cor}Assume that $G$ is of $A,D,E$ type, then $\dim \Ga(\Bun_G,\mO(1))=|Z(G)|^g$, where $Z(G)$ is the center of $G$, and $g$ is the genus of $X$.
\end{cor}
In particular, if $G=E_8$, there is a unique divisor on $\Bun_G$ that defines $\mO(1)$. It would be very interesting to find this divisor geometrically (this question was asked by Faltings, see \cite{Fa2}).

For further discussions and applications of the one point uniformization theorem, we refer to \cite{BL, LS, BLS, So0, So, He}. See also \cite{KNR, KN}.

\subsection{Ad\`elic uniformization.}
We have seen some important applications of the uniformization of $\Bun_{\underline G}$ by $\Gr_x$. However, in practice there are several limitations of this uniformization theorem. First of all it only works for semisimple group. More seriously, the geometry of the fibers of $u_x$, which 
are $\underline G^{X^*}$-torsors, are still complicated, so some questions about $\Bun_{\underline G}$ cannot be easily accessible by studying $\Gr_x$ (e.g. Theorem \ref{uniformize line}).
In this subsection, we introduce another uniformization of $\Bun_{\underline G}$ to overcome of these difficulties.

Recall the the following result of A. Weil. Assume that $k$ is algebraically closed and let $X$ be a smooth connected projective curve over $k$. Let $|X|$ denote the set of closed points. Let $\bO=\prod_{x\in |X|} \mO_x$ be the ring of integral ad\`eles and $\bA=\prod' F_x$ be the ring of ad\`eles, which is the \emph{restricted} product of $F_x$ over all closed points $x\in |X|$: an element in $\bA$ is a collection of $(f_x\in F_x)_{x\in |X|}$ such that $f_x\in \mO_x$ for all but finitely many $x$. Let $\eta$ denote the generic point of $X$, and $F_\eta$ the fractional field of $X$. Note that the natural embedding $F_\eta\subset F_x$ defines a diagonal embedding $F_\eta\subset \bA$.
Let $\underline G$ be a fiberwise connected smooth affine group scheme over $X$, and write $G=\underline G\otimes F_\eta$. 
Then
\begin{lem} There is a natural bijection
\begin{equation}\label{Weil}
\Bun_{\underline G}(k)\simeq G(F_\eta)\backslash G(\bA)/\underline G(\bO),
\end{equation}
\end{lem}
\begin{proof} Let $\mE$ be a $\underline G$-torsor on $X$. Since $F_\eta$ is a field of cohomological dimension one, the $G$-torsor $\mE\otimes F_\eta$ is trivial by Steinberg's theorem. We choose such a trivialisation $\beta_\eta: \mE\otimes F_\eta\simeq \mE^0\otimes F_\eta$. It extends to a trivialisation over some open subset $U\subset X$.
On the other hand, for every $x\in |X|$, $\mE|_{D_x}$ is also trivial and we fix such a trivialisation $\beta_x: \mE|_{D_x}\simeq \mE^0|_{D_x}$. Then $(\beta|_{D_x^*})\circ(\beta_x^{-1}|_{D_x^*})$ is an automorphism of the trivial $\underline G$-torsor on $D_x^*$ and therefore defines an element $g_x\in G(F_x)$. Note that for $x\in |U|$, $g_x\in G(\mO_x)$ and therefore, from the triple $(\mE,\beta_\eta,\{\beta_x\}_{x\in |X|})$ we construct an element $(g_x)\in G(\bA)$. Different choices of $\beta_\eta$ and $\beta_x$ will modify this element by left and right multiplication by elements from $G(F_\eta)$ and from $G(\bO)$. Therefore,
$\mE$ gives rise to a well-defined element in $G(F_\eta)\backslash G(\bA)/\underline G(\bO)$. This defines a map $\Bun_{\underline G}(k)\to G(F_\eta)\backslash G(\bA)/\underline G(\bO)$.
Conversely, given an element $g=(g_x)\in G(\bA)$, there exists an open subset $U\subset X$ such that $g_x\in G(\mO_x)$ for $x\in |U|$. Then
one can construct $\mE$ on $X$ by gluing the trivial torsor on $U$ and on $D_x, x\in |X|\setminus |U|$ using $g_x\in G(F_x)=\Aut(\mE^0|_{D_x^*})$ as transition maps. This construction gives the inverse map. 
\end{proof}

We will describe an algebro-geometric version of \eqref{Weil}, which will be called the ad\`elic uniformization of $\Bun_{\underline G}$. The first question is to endow  the set $G(\bA)/\underline G(\bO)$ with an algebro-geometric structure. In fact, there are two natural solutions to this question.

The first solution uses the Ran Grassmannian $\Gr_{\Ran}$ (Definition \ref{dfn Ran Grass}).  While an element in $G(\bA)/\underline G(\bO)$ can be interpreted as a pair $(\mE,\beta_\eta)$, where $\mE$ is a $\underline G$ on $X$ and $\beta_\eta$ is section of $\mE$ over $\eta$ (or equivalently over some open subset $U\subset X$), a $k$-point of $\Gr_{\Ran}$ also specify a set of points on $X$ at which $\beta_\eta$ may fail to be a regular map. One can get rid of the redundancy as follows (following \cite[\S~2]{Ga3}).

Recall that the Ran space $\Ran=\Ran(X)$ has a semigroup structure by union of points. There is an ``action'' map 
\[\on{act}: \Ran\times \Gr_{\Ran}\to \Gr_{\Ran}\] sending $(x, (y,\mE,\beta))\in \Ran(R)\times\Gr_{\Ran}(R)$ to $(x\cup y, \mE, \beta|_{X_R-\Ga_{x\cup y}})$. The ``quotient'' presheaf, denoted by $\Gr_{\Ran,\on{indep}}$, then is defined as the coequalizer of the diagram
\[\Ran\times\Gr_{\Ran}\substack{\on{act}\\  \xrightarrow{\hspace*{.6cm}} \\[-0.6ex]  \xrightarrow{\hspace*{.6cm}} \\ \pr_2 } \Gr_{\Ran}.\]  
Concretely, $\Gr_{\Ran,\on{indep}}(R)$ is the quotient of $\Gr_{\Ran}(R)$ by the equivalence relation: $(x,\mE,\beta)\sim (x',\mE',\beta')$ if there is an isomorphism $\al: \mE\cong \mE'$ of $\underline G$-torsors on $X_R$ compatible with $\beta$ and $\beta'$ when restricted to $X_R-\Ga_{x\cup x'}$.
It follows by definition that 
$$\Gr_{\Ran,\on{indep}}(k)= G(\bA)/\underline G(\bO).$$ However $\Gr_{\Ran,\on{indep}}$ is in general not a space unless $\underline G$ is the trivial group (in which case $\Gr_{\Ran,\on{indep}}$ is a point). Fortunately, the fpqc sheafification of $\Gr_{\Ran,\on{indep}}$ exists and can be described explicitly as follows.

We say an open subscheme $U\subset X_R$ very dense if it is fiberwise dense, i.e. the composition $U\to X_R\to \Spec R$ is faithfully flat. Note that finite intersections of very dense open subsets are very dense.
We define $\Gr_{\underline G,\Rat(X)}$ as the prestack that assigns every $k$-algebra $R$ the following groupoid: objects are triples $(\mE,U,\beta)$ where $\mE$ is a $\underline G$-torsor on $X_R$, $U\subset X_R$ is a very dense open subscheme, and $\beta$ is a trivialisation of $\mE$ on $U$; morphisms from $(\mE,U,\beta)$ to $(\mE',U',\beta')$ are isomorphisms $\alpha:\mE\cong \mE'$ such that $\beta'\alpha|_{U\cap U'}=\beta|_{U\cap U'}$. 

We write $\Gr_{\underline G,\Rat(X)}$ by $\Gr_{\Rat}$ for simplicity if $\underline G$ and $X$ are clear from the context. We need 
\begin{lem}\label{rat sheaf}
The above groupoid is discrete, and $\Gr_{\Rat}$ is in fact a sheaf.
\end{lem}
\begin{proof}We need to show that there is no non-trivial automorphism for any object $(\mE,U,\beta)$. I.e., let $R$ be a $k$-algebra and $\mE$ be a $\underline G$-torsor over $X_R$. Let $\al: \mE\to \mE$ be an automorphism. If $\al|_U=\id$, then $\al=\id$. By embedding $\mE$ into some affine space over $X_R$, it reduces to show that if $f$ is a regular function on $X_R$, such that its restriction to some very dense open subset is zero, then $f=0$. Note that this is not completely trivial since we do not assume that $R$ is reduced. It follows from the following lemma. 
\end{proof}

\begin{lem}\label{very dense vs Cartier}
Let $X$ be a quasi-projective variety, and $Z\subset X$ be a finite (possibly empty) set of closed points, regarded as a closed subscheme.
Let $U\subset X_R$ be an open dense subset such that: (i) $Z_R\subset U$; and (ii) $U\to \Spec R$ is surjective. Then there is an open cover of $\Spec R$ by $\Spec R_i$ and relative Cartier divisors $D_i\subset X_{R_i}$ such that $Z_R\cap D_i=\emptyset$, and $X_{R_i}\setminus D_i\subset U\otimes_RR_i$. 
\end{lem}
We recall that on an affine space $X=\Spec A$, a Cartier divisor $D$ is by definition a closed subscheme defined by one equation $f=0$ for some $f\in A$ that is \emph{not} a zero divisor (cf. \cite[Tag01WQ]{St}). Then from the definition the localisation map $A\to A_f$ is injective and therefore any function that vanishes on $X\setminus D$ will vanish on the entire $X$.
\begin{proof}We reduce to the following situation: $X=\Spec A$ is affine, and $S=\Spec R$,  and $U=\Spec A\otimes_k R\setminus V(\fraka)$, where $\fraka$ is an ideal of $A\otimes_k R$, and $V(\fraka)\cap Z=\emptyset$. 
The subjectivity of $U\to S$ then is equivalent to the condition that $\on{Im}(\fraka\otimes_RK\to A\otimes_kK)\neq 0$, where $K=\on{Frac}(R/\frakp)$ for every prime ideal $\frakp$ of $R$. We fix such a $\frakp$ and are free to shrink $S$ around $x=V(\frakp)$.

Writing $R$ as the union of finitely generated $k$-algebras, and by shrinking $U$, we can assume that $R$ is finitely generated and that the image of $U\to \Spec R$ contains $x$.
Let $\bar{f}$ be such a non-zero element in $\on{Im}(\fraka\otimes_R R/\frakp\to A\otimes_k R/\frakp)$ such that $V(\bar{f})\cap Z_K=\emptyset$, and let $f\in \fraka$ be a lifting of $\bar f$. 
Then by \cite[Tag00ME]{St}, the multiplication map $f: A\otimes_k R_\frakp\to A\otimes_k R_\frakp$ is injective and the quotient is $R_\frakp$-flat. Since $R$ is noetherian, by shrinking $S$ if necessary around $x$, we can assume that $R\to R_\frakp$ is injective. Then $f:A\otimes_k R\to A\otimes_k R$ is injective. Again by shrinking $S$ if necessary, we can assume that $V(f)$ does not contain vertical divisors and $V(f)\cap Z=\emptyset$. Then $f: A\otimes_k R/\frakq\to A\otimes_k R/\frakq$ is injective for every prime ideal $\frakq$ in $R$. Therefore, $(A\otimes_k R/f)_\frakq$ is flat over every $R_\frakq$ and therefore $A\otimes_kR/f$ is flat over $R$. Then $D=V(f)$ 
is the sought-after relative Cartier divisor.
\end{proof}

Clearly, the map $(x,\mE,\beta)\mapsto (\mE,U=X_R-\Ga_x,\beta)$ realises $\Gr_{\Ran,\on{indep}}$ as a subpresheaf of $\Gr_{\on{Rat}}$.
\begin{lem}\label{sheafify Ran}
The space  $\Gr_{\on{Rat}}$ is the fpqc sheafification of $\Gr_{\Ran,\on{indep}}$.
\end{lem}
\begin{proof}It is enough to show that fpqc locally every point $\Spec R\to \Gr_{\on{Rat}}$ can be lifted to a point of  $\Gr_{\Ran,\on{indep}}$. Let $(\mE,U, \beta)$ represent an $R$-point of $\Gr_{\on{Rat}}$. By Lemma \ref{very dense vs Cartier}, Zariski locally on $\Spec R$, we can assume that $U=X_R\setminus D$ for some relative Cartier divisor $D\subset X_R$. Recall that the presheaf that assigns every $R$ the set of relative Cartier divisors on $X$ is represented by symmetric powers of $X$. On the other hand, the symmetrisation map $X^d\to X^{(d)}$ is finite faithfully flat. Therefore, after a finite faithfully flat base change of $R$, we can assume that $U$ is the complement of $\Ga_x$ for some $x:\Spec R\to X^d$.
\end{proof}

Next, we need an algebro-geometric structure on $G(F_\eta)$. Again there are two solutions. We will only introduce 
${\underline G}_{\Rat}:={\underline G}_{\Rat(X)}$, which is the one corresponding to $\Gr_{\on{Rat}}$. It
 assigns every $R$ the groupoid of pairs $(U,\al)$, where $U$ is very dense open on $X_R$ and $\al\in \underline G(U)$. Again, this is a sheaf of groups. The following lemma is easy.
\begin{lem}\label{generic dependence}
${\underline G}_{\Rat(X)}$ depends only on $\underline G\otimes F_\eta$.
We have ${\underline G}_{\Rat(X)}(k)=G(F_\eta)$.
\end{lem}

Now we can state an algebro-geometric version of \eqref{Weil}.
\begin{thm}\label{ran uniformization of BunG}
Assume that the generic fiber of $\underline G$ is reductive.
The morphism $u_{\Rat}:\Gr_{\Rat}\to\Bun_{\underline G}$ is a ${\underline G}_{\Rat}$-torsor and therefore induces an isomorphism 
$$[{\underline G}_{\Rat(X)}\backslash \Gr_{\Rat}]\cong \Bun_{\underline G}.$$ Precisely, for every $\Spec R\to \Bun_{\underline G}$, there is a faithfully flat map $\Spec R'\to \Spec R$ such that $\Spec R'\times_{\Bun_{\underline G}}\Gr_{\Rat}$ is ${\underline G}_{\Rat}$-equivariantly isomorphic to $\Spec R'\times {\underline G}_{\Rat}$. 
\end{thm}
\begin{proof}
The basic ingredient of this proposition is the same as in Theorem \ref{uniformization for P1}. Namely every $\underline G$-torsor on $X_R$ admits a reduction to a Borel over some very dense open subset after a faithfully flat base change $\Spec R'\to\Spec R$. To prove this, by Lemma \ref{generic dependence} we can modify some fibers of $\underline G$ to make it a Bruhat-Tits group scheme (see Remark \ref{projectivity of BD} for the definition) over $X$ and then apply \cite[Corollary 26]{He}. But since $\bG_m$ and $\bG_a$-torsors can be trivialised Zariski locally, the theorem follows.
\end{proof}

One reason to introduce the above ad\`elic uniformization theorem of $\Bun_{\underline G}$ is because the geometry of the fibers of $u_{\Rat}$ are surprisingly ``simple''. Namely,
the group space ${\underline G}_{\Rat(X)}$ in some sense behaves like a point. We justify this assertion by the following theorem. 
Recall that for a $k$-space $\mF$ (or a $k$-prestack), it makes sense to talk about its ring of regular functions, the Picard groupoid of line bundles (see \eqref{Pic line}). In \cite[\S~2]{GL}, it was also explained how to define the ($\ell$-adic) (co)homology of a prestack.

\begin{thm} Let $R$ be a $k$-algebra of finite type.
\begin{enumerate}
\item[(1)] Every regular function on $\underline G_{\Rat}\otimes R$ is the pullback of a regular function on $\Spec R$.
\item[(2)] Every line bundle on  $\underline G_{\Rat}\otimes R$ is the pullback of a line bundle on $\Spec R$.
\item[(3)] The map $\underline G_{\Rat}\otimes R\to \Spec R$ induces an isomorphism of ($\ell$-adic) homology. 
\end{enumerate}
\end{thm}
Part (1) and (2) are due to Beilinson-Drinfeld (see \cite[Proposition 4.3.13]{BD} and \cite[\S~4.9.1]{BD0}). Part (3) is due to Gaitsgory-Lurie (see \cite{Ga3} and \cite[Theorem 3.3.1]{GL}). 
\begin{proof}We refer to \cite{GL} for Part (3), and only explain (1)
in the case when $G=\bG_m$ or $\bG_a$. The general case reduces to this special case (using Lemma \ref{generic dependence}). 

We can choose an ample line bundle $\mL$ on $X$. For $n\gg 0$, let $V_n=\Gamma(X,\mL^n)$, and let $V'_n= V_n\setminus \{0\}$. Then there are maps 
$$p_n:V'_n\times V'_n\to (\bG_m)_{\on{rat}(X)},\quad  p_n:V_n\times V'_n\to (\bA^1)_{\on{rat}(X)}$$ given by $p_n(f,g)\mapsto f/g$. By Lemma \ref{very dense vs Cartier}, every map $S\to (\bG_m)_{\on{rat}(X)}$ (or $S\to (\bA^1)_{\on{rat}(X)}$) Zariski locally on $S$ factors through $p_n$ for some $n$ large enough.
Given a regular function $\varphi$ on $Y_{\on{rat}(X)}\otimes R$, its pullback gives a regular function $p_n^*\varphi$ on $(V'_n\times V'_n)\otimes R$ (or on $(V_n\times V'_n)\otimes R$) which is invariant with respect to the obvious action of $\bG_m$ on $V_n'\times V'_n$ (or on $V_n\times V'_n$). Choose $n$ large enough so that $\dim V_n>1$. Then $p_n^*\varphi$ extends to a $\bG_m$-invariant regular function on $(V_n\times V_n)\otimes R$, which necessarily comes from $R$. This proves (1). One can similarly argue for (2).
 \end{proof}

To make this theorem more useful, we need the following companion statement.
By definition, there are the following sequence of maps
 \[\Gr_{\Ran}\stackrel{w}{\to} \Gr_{\Rat}\stackrel{u_{\on{Rat}}}{\to} \Bun_{\underline G},\]
and $u_{\Ran}=w\cdot u_{\on{Rat}}$. 
We have the following result.

\begin{thm}\label{Ran=Rat}~
\begin{enumerate}
\item[(1)] The pullback along $w$ induces an isomorphism of spaces of regular functions. 
\item[(2)] The pullback along $w$ induces an equivalence of Picard groupoids of rigidified line bundles. 
\item[(3)] The pushforward along $w$ induces an isomorphism of $\ell$-adic homology.
\end{enumerate}
\end{thm}

Note that this theorem is non-trivial even when $\underline G$ is the trivial group. In this case, we state it as  the following proposition.
\begin{prop}\label{Ran pt}~
\begin{enumerate}
\item[(1)] Every regular function on $\Ran X\otimes R$ is the pullback of a regular function on $\Spec R$.
\item[(2)] The natural map $\Ran X\otimes R\to \Spec R$ induces an isomorphism of ($\ell$-dic) homology. 
\end{enumerate}
\end{prop}
\begin{proof}
Again, we refer to \cite{Ga3} and \cite[\S~2]{GL} for (2). 
To prove (1), we make use of the following lemma.
\begin{lem}
Let $f\in R[[x,y,z]]$ be a power series in $x,y,z$ with coefficients in $R$. If $f$ is invariant under permutations of $x,y,z$ and if $f(x,x,y)=f(x,y,y)$, then $f(x,x,x)$ is constant.
\end{lem}
We leave this lemma as an interesting exercise. By factoring $X^I\to \Ran(X)$ as $X^I\stackrel{\Delta}{\to} X^{I\sqcup I\sqcup I}\to \Ran(X)$, one sees that every function on $\Ran\otimes R$, when pullback to $X^I\otimes R$, comes from a regular function on $\Spec R$. (1) then follows.
\end{proof}

Now we explain Part (1) of Theorem \ref{Ran=Rat}. Recall the definition of the space of regular functions on a prestack (cf. \eqref{reg fun}). Since $\bG_a$ is a sheaf, according to Lemma \ref{sheafify Ran}, 
$$\Gamma(\Gr_{\Rat},\mO)=\Hom(\Gr_{\Rat},\bG_a)=\Hom(\Gr_{\Ran,\on{indep}},\bG_a)=\Gamma(\Gr_{\Ran,\on{indep}},\mO)$$
and by definition, the latter is isomorphic to the equaliser of the following sequence
\[\Gamma(\Gr_{\Ran},\mO)\substack{\on{act}^*\\  \xrightarrow{\hspace*{.6cm}}\\[-0.6ex] \xrightarrow{\hspace*{.6cm}} \\ \pr_2^*} \Gamma(\Ran\times \Gr_{\Ran},\mO).\]
Therefore, to prove (1), it is enough to show that $\on{act}^*=\pr_2^*$. There is a natural map $(q_{\Ran},\id):\Gr_{\Ran}\to \Ran\times\Gr_{\Ran}$ sending $(x,\mE,\beta)$ to $(x,(x,\mE,\beta))$. In addition
\[ \on{act}\circ(q_{\Ran},\id)=\pr_2\circ(q_{\Ran},\id)=\id.\]
By Proposition \ref{Ran pt} (1), $\pr_2^*$ is an isomorphism. It follows that $(q_{\Ran},\id)^*$ is an isomorphism which in turn implies that $\on{act}^*=\pr_2^*$. Part (2) of Theorem \ref{Ran=Rat} can be proved similarly.

\begin{cor}\label{Gr=Bun}~
\begin{enumerate}
\item The pullback along $u_{\Ran}$ induces an isomorphism of spaces of regular functions. 

\item It induces an equivalence of Picard groupoids of rigidified line bundles. 

\item The pushforward induces an isomorphism of homology of $\Gr_{\Ran}$ and $\Bun_{\underline G}$.
\end{enumerate}
\end{cor}

\begin{rmk}
(i) Note that by combining Theorem \ref{Rigid line} and Part (2) of Corollary \ref{Gr=Bun}, we obtain a natural proof of Theorem \ref{uniformize line}.

(ii) Part (3) of Corollary \ref{Gr=Bun} is one of the key steps in Gaitsgory-Lurie's proof of Weil's conjecture on Tamagawa number for global function fields (cf. \cite{GL}).
\end{rmk}

\section{The geometric Satake equivalence}\label{Lecture V}
In this lecture, we discuss the geometric Satake equivalence, which is a cornerstone of the geometric Langlands program. 
We will assume that $k$ is algebraically closed unless otherwise stated. Let $\mO=k[[t]]$ and $F=k((t))$, and let $G$ denote a reductive group over $k$. Let $\underline G= G\otimes_k\mO$. For simplicity, we will write $\Gr$ for $\Gr_G$ if the group $G$ is clear.

We fix a prime number $\ell$ different from the characteristic of $k$.
Sheaves will mean $\Ql$-sheaves, and all pushforward and pullback are \emph{derived}.
Let us mention that the geometric Satake equivalence can also be formulated
in two other sheaf contents. (1) If $\cha\ k=0$, one can work with D-modules (see \cite[\S\ 5.3]{BD}); (2) if $k=\bC$, one can work with sheaves of $E$-vector spaces under analytic topology on $\Gr$, where $E$ is field of characteristic zero. The approach given below also applies to these contents.
On the other hand, there are versions of the geometric Satake equivalence for $\bZ_\ell$- and $\bF_\ell$-sheaves, or when $k=\bC$, for sheaves of $\Lambda$-modules, where $\Lambda$ is a commutative noetherian ring of finite global dimension (e.g. $\Lambda=\bZ$). The following approach then is insufficient. One needs to refer to \cite{MV} for the extension to general coefficients.
 
We also mention that there is a geometric Satake equivalence for $p$-adic groups. We refer to \cite{Z14} for details.

\subsection{The Satake category $\Sat_G$.}

First, the action of $L^+G$ on $\Gr$ satisfies Condition (S) in \S~\ref{A:ind}, since each orbit is of the form $\Gr_\mu=L^+G/(L^+G\cap t^{\mu} L^+G t^{-\mu})$ and  $L^+G\cap t^{\mu} L^+G t^{-\mu}$ is connected.
Therefore the construction of \S~\ref{A:ind} gives the category $\on{P}_{L^+G}(\Gr)$ of $L^+G$-equivariant perverse sheaves on $\Gr$. Denote by $\IC_\mu$ the intersection cohomology sheaf on $\Gr_{\leq \mu}$. Then
$\IC_\mu|_{\Gr_\mu}=\Ql[(2\rho,\mu)]$, and its restriction to each stratum $\Gr_\la$ is constant. Again since  $L^+G\cap t^{\mu} L^+G t^{-\mu}$ is connected, irreducible objects of $\on{P}_{L^+G}(\Gr)$ are exactly these $\on{IC}_\mu$'s.

\begin{prop}\label{semisimplicity}
The category $\on{P}_{L^+G}(\Gr)$ is semisimple.
\end{prop}
\begin{proof}We need to show that $\on{Ext}^1(\IC_\la,\IC_\mu)=0$ for any $\la$ and $\mu$. The essential case is $\la\leq \mu$ or $\mu\leq \la$. In this case, the vanishing follows from: (i) purity of the stalk cohomology of $\on{IC}$ sheaves of (affine) Schubert varieties (cf. \cite{KL}); and (ii) parity of the dimension of $L^+G$-orbits in each connected component of $\Gr$ (see Lemma \ref{parity}). We refer to \cite[Proposition 1]{Ga} (see also \cite[Proposition 3.1]{Ri}) for details. 
\end{proof}

Here is a consequence.
\begin{cor}
Objects in $\on{P}_{L^+G}(\Gr)$ are $\Aut^+(D)$-equivariant.
\end{cor}
\begin{proof}Since every IC sheaf must be $\Aut^+(D)$-equivariant (by Lemma \ref{A:forgetful}).
\end{proof}

There is natural a monoidal structure on $\on{P}_{L^+G}(\Gr)$, defined by Lusztig's convolution product of sheaves. Namely, for $\mA_1,\mA_2\in \on{P}_{L^+G}(\Gr)$, we denote by $\mA_1\tilde\boxtimes\mA_2$ the ``external twisted product'' of $\mA_1$ and $\mA_2$ on $\Gr\tilde\times\Gr$ (see \S~\ref{A:twist product} and  \S~\ref{A:ind}). For example, if $\mA_i=\IC_{\mu_i}$, then $\IC_{\mu_1}\tilde\boxtimes\IC_{\mu_2}$ is canonically isomorphic to the intersection cohomology sheaf of $\Gr_{\leq \mu_1}\tilde\times\Gr_{\leq\mu_2}$.
Then the convolution product of $\mA_1$ and $\mA_2$ is defined as
\begin{equation}\label{Lus conv}
\mA_1\star\mA_2:=m_!(\mA_1\tilde\boxtimes\mA_2),
\end{equation}
where $m:\Gr\tilde\times\Gr\to \Gr$ is the convolution map (defined by \eqref{conv m0}).
This is an $L^+G$-equivariant $\ell$-adic complex on $\Gr$. Similarly, one can define the $n$-fold convolution product $\mA_1\star\cdots\star\mA_n=m_!(\mA_1\tilde\boxtimes\cdots\tilde\boxtimes\mA_n)$ using the $n$-fold convolution map \eqref{conv m}. Note that since $m$ is ind-proper, $m_!=m_*$.

\begin{prop}\label{conv prod}
The convolution product $\mA_1\star\mA_2$ is perverse.
\end{prop}
This is called a miraculous theorem in \cite{BD}. It allows us to stay in the world of abelian monoidal categories rather than the more involved notion of triangulated monoidal categories. Another miracle for $\Sat_G$ is Lemma \ref{left=right}.

There are essentially two different approaches to this proposition. 
The first approach is based on some semi-infinite geometry of affine Grassmannians (which will be introduced in \S~\ref{MV semiinfinite} below) in one way or another.  It was used by Lusztig to prove a numerical result of the affine Hecke algebra \cite{Lu} that is equivalent to the proposition (see \cite[Proposition 2.2.1]{Gi1}), and  by Mirkovi\'c-Vilonen to prove the semismallness of the convolution map that is also equivalent to the proposition (see  \cite[Lemma 4.3, 4.4]{MV} and Remark \ref{semismall} below). It was also used in \cite[Corollary 9.7]{NP} (see also \cite[\S\ 2.2]{Z14})  to give a direct proof of this proposition.

The second approach is based on Beilinson-Drinfeld Grassmannians. Using nearby cycles functor Gaitsgory proved a stronger result: Namely in Proposition \ref{conv prod}, $\mA_1$ can be any perverse sheaf on $\Gr$ (note that no equivariance condition on $\mA_1$  is needed in order to define the external twisted product). See Remark \ref{conv perverse} below. 
 
We need a corollary of Proposition \ref{conv prod}.
Let $\mmu=(\mu_1,\ldots,\mu_n)$ be a sequence of dominant coweights. By Proposition \ref{semisimplicity} and \ref{conv prod}, we can write
\[\IC_{\mu_1}\star\cdots\star\IC_{\mu_n}=\bigoplus_{\la} V_{\mmu}^\la\otimes \IC_\la,\]
where $V_{\mmu}^\la=\Hom_{\on{P}(\Gr)}(\IC_\la, \IC_{\mu_1}\star\cdots\star\IC_{\mu_n})$. Let $|\mmu|=\sum \mu_i$.
\begin{cor}\label{Satake fiber}
There is a canonical isomorphism
\[V_{\mmu}^\la\cong \on{H}^{(2\rho,|\mmu|-\la)}_c(\Gr_{\mmu}\cap m^{-1}(t^\la),\Ql),\]
and the latter vector space has a basis given by irreducible components of $\Gr_{\mmu}\cap m^{-1}(t^\la)$ of dimension $(2\rho,|\mmu|-\la)$.
\end{cor}
\begin{proof}Let $d=\dim \Gr_{\mmu}\cap m^{-1}(t^\la)$. By stratifying $\Gr_{\leq \mu_\bullet}$ as $\sqcup_{\mu'_\bullet}\Gr_{\mu'_\bullet}$,
a spectral sequence argument shows that the degree $2d- (2\rho,|\mmu|)$ stalk cohomology of $\IC_{\mu_1}\star\cdots\star\IC_{\mu_n}$ at $t^\la$ is given by
$$\on{H}^{2d}_c(\Gr_{\mmu}\cap m^{-1}(t^\la),\Ql).$$
It follows from the perversity of $\IC_{\mu_1}\star\cdots\star\IC_{\mu_n}$ that $2d- (2\rho,|\mmu|)\leq -(2\rho,\la)$, or equivalently  $d\leq (\rho,|\mmu|-\la)$. In addition, $V_{\mmu}^\la$ equals to the degree $-(2\rho,\la)$ stalk cohomology of $\IC_{\mmu}$ at $t^\la$. The corollary follows.
\end{proof}
\begin{rmk}\label{semismall} 
(i) Note that the argument also implies that 
\[\dim  \Gr_{\leq \mmu}\cap m^{-1}(t^{\la})\leq (\rho,|\mmu|-\la).\]
It follows that the convolution map is semismall. But in general, there might exists irreducible component of $\Gr_{\leq \mmu}\cap m^{-1}(t^{\la})$ of dimension strictly smaller than $(\rho,|\mmu|-\la)$. See \cite[Remark 4.3, 8.3]{Ha} for examples.

(ii) Sometimes these $ \Gr_{\leq \mmu}\cap m^{-1}(t^{\la})$ are called Satake fibers. In the case $G=\GL_n$, sometimes they can be identified with the Springer fibers (see Example \ref{conv v.s. springer}). In particular, it implies the semismallness of the Springer resolution $\tilde\mN_n\to \mN_n$.
\end{rmk}

By identifying $(\mA_1\star\mA_2)\star\mA_3$ and $\mA_1\star(\mA_2\star\mA_3)$ with $\mA_1\star\mA_2\star\mA_3$, one equips the convolution product with natural associativity constraints. We call the monoidal category $(\on{P}_{L^+G}(\Gr),\star)$ the Satake category, and sometimes denoted by $\Sat_G$ for simplicity.

\subsection{$\Sat_G$ as a Tannakian category.}\label{mon str of H}
In this subsection, we will first endow the hypercohomology functor $\on{H}^*: \Sat_G\to \on{Vect}_{\Ql}$ with a monoidal structure following \cite[\S\ 2.3]{Z14}. Then we will explain how to upgrade the monoidal structure on $(\Sat_G,\on{H}^*)$ to a symmetric monoidal structure. The actual proofs will be based on the later interpretation of the convolution product as a fusion product.

\begin{prop}\label{fiber functor}
The hypercohomology functor $\on{H}^*:\Sat_G\to \on{Vect}_{\Ql}$ has a canonical monoidal structure.
\end{prop}

The idea is as follows. First assume that $k=\bC$ so we can regard $\Gr$ and $\Gr\tilde\times\Gr$ as infinite dimensional complex analytic spaces as in \S~\ref{top loop group}. Note that the category $\on{P}_{L^+G}(\Gr_G)$ can also be defined via the analytic topology. It is easy to see that the following holds (see \cite[Proposition 2.1.1]{Gi1}).
\begin{lem}
Under the isomorphisms in Theorem \ref{top triv}, $\mA_1\tilde\boxtimes\mA_2$ on $\Gr\tilde\times\Gr$ is identified with $\mA_1\boxtimes\mA_2$ on $\Omega K\times\Omega K$.
\end{lem}
It follows from the K\"unneth formula that $\on{H}^*(-):=\on{H}^*(\Gr,-):\Sat_G\to \on{Vect}_{\Ql}$ admits a monoidal structure as claimed in Proposition \ref{fiber functor}.

If $k$ is general, Theorem \ref{top triv} and the above lemma do not literally make sense. There are two alternative approaches to make sense of these topological isomorphisms. One is based on a deformation from $\Gr\tilde\times \Gr$ to $\Gr\times \Gr$ via Beilinson-Drinfeld Grassmannians, which will be explained in \S\ \ref{fusion}. Here we explain the other approach using equivariant cohomology, which has the advantage of being purely local (i.e. independent of the choice of a curve).

Recall that for $\mA\in\Sat_G$, its $L^+G$-equivariant cohomology $\on{H}^*_{L^+G}(\mA)$ is a free $R_{G}$-module (cf. \S~\ref{A:char class}). By (informally) thinking $L^+G$-equivariant sheaves on $\Gr$ as sheaves on $L^+G\backslash LG/L^+G$,  there should exist another $R_{G}$-module structure on $\on{H}^*_{L^+G}(\mA)$, which can be literally constructed as follows. Let $L^+G^{(m)}\subset L^+G$ denote the $m$-th congruence subgroup (i.e. the kernel of the projection $L^+G\to L^mG$), and let $\Gr^{(m)}=LG/L^+G^{(m)}$ denote the universal $L^m G$-torsor on $\Gr$. Then $\Gr^{(m)}$ admits an action of $L^+G\times L^mG$ and the projection 
$$\pi_m:\Gr^{(m)}\to \Gr$$ is $L^+G$-equivariant. Then by \eqref{coh free action},
\[\on{H}^*_{L^+G}(\mA)= \on{H}^*_{L^+G\times L^mG}(\pi_m^*\mA),\] 
giving an $R_{G}$-bimodule structure on $\on{H}^*_{L^+G}(\mA)$. Note that this structure is independent of $m$, as soon as $m>0$.

Recall that the category of bimodules of a (unital) algebra $R$ has a standard monoidal structure $(M,N)\mapsto M\otimes_RN$. We have
\begin{lem}\label{soergel monoidal}
For every $\mA_1,\mA_2,\ldots,\mA_n$, there is a canonical isomorphism of $R_{G}$-bimodules
\[\on{H}^*_{L^+G}(\mA_1\star\cdots\star\mA_n)\cong \on{H}^*_{L^+G}(\mA_1)\otimes_{R_{G}}\cdots\otimes_{R_{G}}\on{H}^*_{L^+G}(\mA_n),\]
satisfying the natural compatibility conditions. Therefore, there is a natural monoidal structure on $\on{H}^*_{L^+G}(-):\Sat_G\to (\on{R}_{G}\otimes\on{R}_{G})\on{-mod}$. 
\end{lem}
This is a variant of Soergel's idea to attach equivariant sheaves on (affine) flag varieties to Soergel bimodules. Here we just explain the idea of constructing a map in the case $n=2$ and $\mA_i=\IC_{\mu_i}$. The general case is similar (e.g. see \cite[\S~2.3]{Z14}). We assume that the action of $L^+G$ on $\Gr_{\leq \mu_2}$ factors through $L^mG$ so $\Gr_{\leq \mu_1}\tilde\times\Gr_{\leq \mu_2}= \Gr_{\leq\mu_1}^{(m)}\times^{L^mG}\Gr_{\leq \mu_2}$. By the definition of the convolution product and \eqref{coh free action}
\[\on{H}^*_{L^+G}(\IC_{\mu_1}\star\IC_{\mu_2})=\on{H}^*_{L^+G\times L^mG}(\pi_m^*\IC_{\mu_1}\boxtimes \IC_{\mu_2}).\]
On the other hand by the equivariant K\"unneth formula \eqref{equiv Kunneth}, there is a canonical isomorphism
\[\on{H}^*_{L^+G}(\IC_{\mu_1})\otimes_{R_{G}} \on{H}^*_{L^+G}(\IC_{\mu_2})\cong \on{H}^*_{L^+G\times L^mG}(\pi_m^*\IC_{\mu_1}\boxtimes \IC_{\mu_2}).\]
Putting them together, we obtain the desired isomorphism.

Another ingredient we need is
\begin{lem}\label{left=right}
The two $R_{G}$-structures on $\on{H}^*_{L^+G}(\mA)$ coincide.
\end{lem}
\begin{proof}We explain the ideas and refer to \cite[\S\ 2.3]{Z14} for details. Using Lemma \ref{minusch} and Lemma \ref{quasiminusch}, it is easy to prove the lemma for $\mA=\IC_\mu$ where $\mu$ is minuscule or quasi-minuscule. Then by Lemma \ref{soergel monoidal}, the statement holds for the convolutions of these $\IC_\mu$s. Finally, every $\mA\in\Sat_G$ appears as a direct summand of these convolutions (this is a geometric version of the so-called PRV conjecture, which was proved in \cite[Proposition 9.6]{NP}). The lemma follows.\end{proof}

\begin{rmk}
This lemma is specific to the case $\underline G=G\otimes \mO$. For example, it is not true for equivariant perverse sheaves on the affine flag variety $\Fl_G$ (see Remark \ref{terminology} for the terminology).
\end{rmk}

Note that combining Proposition \ref{soergel monoidal} and  \ref{left=right}, we actually endow the equivariant hypercohomology functor
$$\on{H}^*_{L^+G}: \Sat_G\to \on{Proj}_{R_G}$$ 
a canonical monoidal structure, where as before $\on{Proj}_{R_G}$ denotes the category of finite projective $R_G$-modules, equipped with the usual tensor product. 
By \eqref{de-equiv}, after specialising along the augmentation map $R_{G}\to \Ql$, we obtain the desired monoidal structure on $\on{H}^*$.

We leave it as an exercise to verify that when $k=\bC$, the above monoidal structure on $\on{H}^*$ coincides with the one constructed using the homeomorphism $\Gr\tilde\times\Gr\cong \Gr\times \Gr$.

Having endowed $\on{H}^*$ with a monoidal structure, we now explain the construction of the commutativity constraints on $\Sat_G$ that makes $\on{H}^*$ a symmetric monoidal functor. The key statement is as follows.
\begin{prop}\label{existence of comm}
For every $\mA_1,\mA_2\in\Sat_G$, there exists a unique isomorphism $c'_{\mA_1,\mA_2}:\mA_1\star\mA_2\cong\mA_2\star\mA_1$ such that the following diagram is commutative
\[\begin{CD}
\on{H}^*(\mA_1\star\mA_2)@>\on{H}^*(c'_{\mA_1,\mA_2})>>\on{H}^*(\mA_2\star\mA_1)\\
@V\cong VV@VV\cong V\\
\on{H}^*(\mA_1)\otimes\on{H}^*(\mA_2)@>\cong>c_{\on{gr}}>\on{H}^*(\mA_2)\otimes\on{H}^*(\mA_1),
\end{CD}\]
where the vertical isomorphisms come from Proposition \ref{monoidal functor}, and the isomorphism $c_{\on{gr}}$ is the commutativity constraint for graded vector spaces, i.e. 
$$c_{\on{gr}}(v\otimes w)=(-1)^{\deg(v)\deg(w)}(w\otimes v).$$
\end{prop}
\begin{rmk}The readers should notice that Proposition \ref{fiber functor}  puts an additional structure on $(\Sat_G, \on{H}^*)$ while Proposition \ref{existence of comm} establishes some properties of $(\Sat_G,\on{H}^*)$.
\end{rmk}
As $\on{H}^*:\Sat_G\to\on{Vect}_{\Ql}$ is faithful (by semisimplicity of $\Sat_G$), the uniqueness of $c'_{\mA_1,\mA_2}$ is clear. The content is its existence. It will be proved in \S~\ref{fusion} after we interpret the convolution product as a fusion product. Note that \cite[\S~2.4]{Z14} suggests that there might be a purely local approach (a.k.a without using a global curve) of this proposition.

Here is the corollary that we need.
First, we can modify the constraints $c'_{\mA_1,\mA_2}$  by a Koszul sign change as follows (see also \cite[\S~5.3.21]{BD} or \cite{MV} after Remark 6.2 for a more elegant treatment). Namely, the category $\on{P}_{L^+G}(\Gr)$ admits a $\bZ/2$-grading induced by \eqref{parity map}. We say $\mA$ has pure parity if $p(\on{Supp}(\mA))$ is $1$ or $-1$, in which case we define $p(\mA)=p(\on{Supp}(\mA))$. 
Here $\on{Supp}(\mA)$ denote the support of $\mA$. Then
\begin{equation}\label{koszul sign change}
c_{\mA_1,\mA_2}:=(-1)^{p(\mA_1)p(\mA_2)}c'_{\mA_1,\mA_2},
\end{equation}
if $\mA_1$ and $\mA_2$ have the pure parity $p(\mA_1)$ and $p(\mA_2)$. With this modified commutativity constraints, the following diagram is commutative
\[\begin{CD}
\on{H}^*(\mA_1\star\mA_2)@>\on{H}^*(c_{\mA_1,\mA_2})>>\on{H}^*(\mA_2\star\mA_1)\\
@V\cong VV@VV\cong V\\
\on{H}^*(\mA_1)\otimes\on{H}^*(\mA_2)@>\cong>c_{\on{vect}}>\on{H}^*(\mA_2)\otimes\on{H}^*(\mA_1),
\end{CD}\]
where the isomorphism $c_{\on{vect}}$ is the usual commutativity constraint for vector spaces, i.e. 
$c_{\on{vect}}(v\otimes w)=w\otimes v$.

\begin{thm}\label{Tannakian structure}
The monoidal category $\Sat_G$, equipped with the above constraints $c_{\mA_1,\mA_2}$, form a symmetric monoidal category. The hypercohomology functor $\on{H}^*$ is a tensor functor.  In fact, $(\on{P}_{L^+G}(\Gr),\star,c,\on{H}^*)$ form a neutral Tannakian category. 
\end{thm} 
\begin{proof}We need to check $c_{\mA_2,\mA_1}c_{\mA_1,\mA_2}=\id$, and the hexagon axioms.
Using the faithfulness of $\on{H}^*$, it is enough to prove these after taking the cohomology. Using Proposition \ref{existence of comm}, and the fact $c_{\on{vect}}^2=\id$, we see that $\on{H}^*(c_{\mA_2,\mA_1}c_{\mA_1,\mA_2})=\id$, and therefore $c_{\mA_2,\mA_1}c_{\mA_1,\mA_2}=\id$. The hexagon axioms can be proved similarly.
The second statement is clear.

To show that $(\on{P}_{L^+G}(\Gr),\star,c,\on{H}^*)$ is a neutral Tannakian category, we apply \cite[Proposition 1.20, Remark 2.18]{DM}. Namely, again by semisimplicity, $\on{H}^*$ is exact (and faithful). It is easy to check Condition (a)-(d) of Proposition 1.20 of \emph{ibid.} (with $\IC_0$ being the unit object). To verify Condition (e) there, one observes that $\dim \on{H}^*(\Gr,\IC_\mu)=1$ if and only if $\dim \Gr_{\leq \mu}=0$, in which case $\mu$ is a cocharacter of the center $Z_G\subset G$. Then $\IC_{-\mu}\star\IC_\mu\cong \IC_0$.
\end{proof}

\subsection{Langlands dual group.}\label{MV semiinfinite}
We have seen that $\Sat_G$ together with the commutativity constraints $c$ and the fiber functor $\on{H}^*$ form a neutral Tannakian category (modulo Proposition \ref{existence of comm}, which will be established in the next subsection).
Let 
$$\tilde G=\Aut^\otimes \on{H}^*$$ denote the Tannakian group. The next step is to give an explicit description of this affine group scheme in terms of $G$. First, it is easy to show that
\begin{lem}The affine group scheme $\tilde G$ over $\Ql$ is a connected reductive group.
\end{lem}
\begin{proof}First note $\IC_{\mu_1}\star\IC_{\mu_2}$ contains $\IC_{\mu_1+\mu_2}$ as a direct summand since the convolution map $\Gr_{\leq\mu_1}\tilde\times\Gr_{\leq \mu_2}\to \Gr_{\leq \mu_1+\mu_2}$ is birational. From this fact, we see that $\Sat_G$ is generated as a tensor category by finitely many $\IC_\mu$'s (e.g. by those $\IC_\mu$'s with $\mu$ being fundamental coweights), and there is no tensor subcategory that contains only direct sums of finite collection of $\IC$'s (e.g. the $n$-fold self convolution product of $\IC_\mu$ contains $\IC_{n\mu}$ as a direct summand). Then by \cite[Proposition 2.20, Corollary 2.22]{DM}, $\tilde G$ is algebraic and connected. Finally, since $\Sat_G$ is semisimple, $\tilde G$ is reductive (\cite[Proposition 2.23]{DM}).
\end{proof}

Recall that to a connected reductive group $G$, one associates a quadruple $(\xch,\xcoch,\Phi,\Phi^\vee)$, called the root datum of $G$ (cf. \S~\ref{Intro:group}). 
The surprising fact is that connected reductive group over an algebraically closed field is uniquely determined by this combinatoric datum  (up to isomorphism). In addition, if $(\xch,\xcoch,\Phi,\Phi^\vee)$ arises from some $G$, then $(\xcoch,\xch,\Phi^\vee,\Phi)$ also arises from some $\hat G$ in the sense that $\xcoch$ (resp, $\xch$) is the weight lattice (resp. coweight lattice) of the Cartan torus $\hat T$ of $\hat G$, and $\Phi^\vee\subset \xcoch$ (resp. $\Phi\subset\xch$) is the corresponding root sets (resp. coroot sets). See \cite[\S~1,2]{Sp} for a summary. The reductive group $\hat G$ is usually called the Langlands dual group of $G$.
Note that the passage from $G$ to $\hat G$ is completely combinatoric and relies on the classification theorem of reductive groups. The following table gives some examples of this passage.

\begin{center}
\begin{tabular}{| c | c | c | c | c | c |}
\hline 
$G$ & $\GL_n$ & $\on{SL}_n$ & $\on{SO}_{2n+1}$ & $\on{SO}_{2n}$ & $E_8$  \\ \hline 
$\hat G$ & $\GL_n$ & $\on{PGL}_{n}$ & $\on{Sp}_{2n}$ & $\on{SO}_{2n}$ & $E_8$ \\ \hline 
\end{tabular}
\end{center}

\begin{thm}\label{birth of dual group}
The Tannakian group $\tilde G$ has the root datum  $(\xcoch,\xch,\Phi^\vee,\Phi)$ and therefore is isomorphic to $\hat G$.
\end{thm}
\begin{rmk}
In a sense, this theorem can be regarded as a classification free definition of the Langlands dual group of $G$.
\end{rmk}

The rest of this subsection explains the idea of proving this theorem and some applications. In the course of the proof, we will put more structures on $\tilde G$ (see Proposition \ref{canonical pinning}). Note that there is another approach of this theorem, due to Richarz (cf. \cite[\S~3]{Ri}).

First, it is easy to check
\begin{lem}
Assume that $G=T$ is a torus. Then the theorem holds.
\end{lem}
\begin{proof}
Indeed, $(\Gr_T)_{\on{red}}$ is a discrete set of points canonically isomorphic to $\xcoch(T)$. Then it is easy to see $\Sat_T$ is equivalent to the category of $\xcoch(T)$-graded finite dimensional $\Ql$-vector spaces and $\on{H}^*$ is just the functor forgetting the grading. The lemma then follows.
\end{proof}

To deduce the general case from this special situation, we mimic the classical construction to define a ``categorical Satake transform'', i.e. a symmetric monoidal functor
\[\on{CT}:\Sat_G\to \Sat_T,\]
where $T$ is the abstract Cartan of $G$.

We use notations as in \S~\ref{Intro:group}. We fix a Borel subgroup $B\subset G$, and let $T=B/U$ be the abstract Cartan, where $U$ is the unipotent radical of $B$.
Consider the following induced diagram
\begin{equation}\label{CT diagram}
\Gr_{T}\stackrel{q}{\leftarrow} \Gr_{B}\stackrel{i}{\to} \Gr_{G},
\end{equation}
Note that $\Gr_{B}\to \Gr_{G}$ is bijective at the level of points but is far from being an isomorphism of ind-schemes. Indeed, the morphism $i$ itself is pathological. However, each connected component of $(\Gr_B)_{\on{red}}$ is locally closed in $\Gr_G$. Since we are only interested in $\ell$-adic sheaves on these spaces, we can safely replace the involved affine Grassmannians by their reduced sub-ind-schemes (but to simplify the notation we will omit the subscript $_{\on{red}}$ in the sequel). Then as mentioned above, $\Gr_T\simeq \xcoch(T)$ is a discrete set of points. For $\la\in\xcoch(T)$, we define
$$S_\la= i(q^{-1}(\la))\subset \Gr_G.$$
Note that it is an orbit through $t^\la$ under the action of $LU$ on $\Gr_G$ (via the embedding $LU\subset LG$). 
Recall the partial order $\leq$ on $\xcoch(T)$. First, we have
\begin{prop}\label{geom of Sla}
The union $S_{\leq \la}:=\cup_{\la'\leq \la} S_{\la'}$ is closed and $S_\la\subset S_{\leq \la}$ is open dense. In particular, $S_\la$ is locally closed in $\Gr_G$ and $S_{\leq \la}$ is its closure $\bar{S}_\la$.\end{prop}
\begin{proof}The proof is similar to the proof of Proposition \ref{HP} and Proposition \ref{geom of Sch}. First we show that $\cup_{\la'\leq \la} S_{\la'}$ is closed. We first assume that $G_\der$ is simply-connected. Then points on $S_{\leq \la}$ admit the following moduli interpretation. Namely, for a highest weight representation $V_\chi$ of $G$, let $\ell_\chi$ be the corresponding highest weight line. Then
\[S_{\leq \la}=\left\{(\mE,\beta)\in\Gr\mid \beta^{-1}(\ell_\chi)\subset t^{-(\chi,\la)}(V_{\chi,\mE}), \mbox{ for all } V_\chi\right\}.\] 
Then as argued in Proposition \ref{HP}, one can express these conditions as a collection of matrix equations, and therefore defines a closed subset. For general $G$, one can pass to a $z$-extension to reduce to the case when $G_\der$ is simply-connected.

Then it remains to shows that if $\al$ is a positive coroot, $S_{\la-\al}$ is in the closure $\overline{S}_\la$. But this follows from the construction of the curve $C_{\la,\al}$ in \eqref{curve C}. 
\end{proof}

Next, we have the following relation between these $LU$-orbits with Schubert and opposite Schubert cells.
\begin{lem}\label{intersection of semi and Sch}
If the intersection $S_\la \cap \Gr_\mu$ is non-empty, then $t^\la\in \Gr_{\leq \mu}$. If the intersection $S_\la\cap \Gr^{\mu}$ is non-empty, then $t^\la\in \Gr^{\geq \mu}$.
\end{lem}
\begin{proof}Let $2\rho^\vee$ denote the sum of positive coroots of $G$. We choose an embedding $T\subset B$. By composition
\[\bG_m\stackrel{2\rho^\vee}{\to} T\subset G\subset L^+G,\]
the action of $L^+G$ on $\Gr_G$ induces a $\bG_m$-action on $\Gr_G$. Then $\Gr_B$ becomes the attractor of this action and $\Gr_T\subset \Gr_G$ is the fixed point subset. I.e. 
$$S_\la=\left\{x\in \Gr_G\mid \lim\limits_{s\to 0}2\rho^\vee(s)x=t^\la\right\}.$$
Then since $\Gr_\mu$ is $2\rho^\vee(\bG_m)$-invariant, $S_\la\cap \Gr_\mu\neq\emptyset$ only if $t^\la\in \overline{\Gr_\mu}=\Gr_{\leq \mu}$. The same argument implies the second statement.
\end{proof}
\begin{cor}\label{boundary of semi}
The set $S_{<\la}:=\cup_{\la'<\la} S_{\la'}$ is the intersection of $\overline{S}_\la$ with a hyperplane section of the ample line bundle $\mO(1)$.
\end{cor}
When $\cha\ k=0$, Mirkovi\'c-Vilonen proved this (cf. \cite[Proposition 3.1(b)]{MV}) using the basic representation $L_1(\frakg_\Ga)$ of the corresponding Kac-Moody algebra (cf. \S~\ref{Kac-Moody}). We give an alternative argument which does not rely on the Kac-Moody theory (and works for arbitrary characteristic).
\begin{proof} By translation by $t^{-\la}$, we can assume $\la=0$. It is easy to see that $S_0\subset \Gr^0$. By Lemma \ref{intersection of semi and Sch}, $S_{<0}\cap \Gr^0=\emptyset$. In other words, $S_{<0}=S_{\leq 0}\cap \Theta$, where we recall from Theorem \ref{Pic of Gr} that $\Theta$ is an effective Cartier divisor on (the neutral connected component of) $\Gr_G$ corresponding to $\mO(1)$.  This finishes the proof.
\end{proof}

Now we can state the main geometric results of semi-infinite orbits proved by Mirkovi\'c-Vilonen (cf. \cite[\S\ 3]{MV}). 
\begin{thm}\label{MV theory}~
\begin{enumerate}
\item If $t^\la\in \Gr_{\leq \mu}$, then $S_\la\cap \Gr_\mu$ is non-empty and $S_\la \cap \Gr_{\leq \mu}$ is
pure of dimension $(\rho,\la+\mu)$.

\item Let $\mA\in\Sat_G$. Then $(q_!i^*\mA)_\la=\on{H}_c^*(S_\la, \mA)$ is concentrated in degree $(2\rho,\la)$. In addition, there is a canonical isomorphism
$$\on{H}_c^{(2\rho,\la)}(S_\la, \IC_\mu)\otimes\Ql((\rho,\la+\mu))\cong \Ql[\on{Irr}(S_\la\cap\Gr_{\leq \mu})],$$ where the r.h.s is the $\Ql$-span of irreducible components of $S_\la\cap\Gr_{\leq \mu}$.

\item There is a canonical isomorphism of functors
$$\on{H}^*(-)\cong \on{CT}:=\bigoplus_\la \on{H}_c^*(S_\la,-):\Sat_G\to \on{Vect}_{\Ql}.$$
\end{enumerate}
\end{thm}
\begin{rmk}
Mirkovi\'c-Vilonen established the theorem for a general coefficient ring $\Lambda$ (when $k=\bC$). In this generality, one needs to replace $\IC_\mu$ in (1) by the corresponding standard or costandard object.
\end{rmk}

Let us explain the basic ideas going into the proof of this theorem and refer to \cite[\S\ 3]{MV} for details.  Another approach to the theorem was given in \cite[\S~2.2]{Z14}, based on some ideas of \cite{NP}. This alternative approach is necessary to establish the geometric Satake for $p$-adic groups because (currently) we do have the corresponding Corollary \ref{boundary of semi} in that setting (there is no Kac-Moody theory nor the Cartier divisor $\Theta$ available in that setting). 

First, it follows from Corollary \ref{boundary of semi} that if $\al$ is a simple coroot, $\dim(S_{\la-\al}\cap\Gr_{\leq \mu})=\dim(S_\la\cap\Gr_{\leq \mu})-1$. Then  (1) follows by induction, starting with
\begin{equation}\label{extreme semi}
S_\la\cap \Gr_{\leq \mu}=S_\la\cap \Gr_\mu=\left\{\begin{array}{ll}t^{w_0(\mu)} & \la=w_0(\mu), \\
                               L^+U t^\mu\cong \bA^{(2\rho,\mu)} & \la=\mu, \end{array}\right.
\end{equation}
where $w_0$ is the longest element in the Weyl group of $G$.

For (2), note that
the functor $q_!i^*$ is nothing but
Braden's hyperbolic localisation functor (cf. \cite{Br}). More precisely, let $B^-$ be the opposite Borel such that $B\cap B^-=T$. Then $\Gr_{B^-}$ is the repeller of the $2\rho^\vee(\bG_m)$-action and we have the diagram 
$$\Gr_T\stackrel{q^-}{\leftarrow}\Gr_{B^-}\stackrel{i^-}{\to}\Gr_G.$$ By Braden's theorem, 
$$q_!i^*\mA=(q^-)_*(i^-)^!\mA, \quad \mbox{ for } \mA\in\Sat_G.$$ In addition, the stratification of $S_\la$ by $\{S_\la\cap\Gr_\mu, \mu\in\xcoch(T)^+\}$ induces a spectral sequence with $E_1$-term $\on{H}_c^*(S_\la\cap\Gr_\mu,\mA)$ and abutment $\on{H}_c^*(S_\la,\mA)$. 
These considerations, together with the dimension formula in (1), establish (2).

To prove (3), the stratification of $\Gr$ by semi-infinite orbits $\{S_\la, \la\in\xcoch(T)\}$ induces a spectral sequence with $E_1$-term $\on{H}_c^*(S_\la,-)$ and abutment $\on{H}^*(-)$. It degenerates at the $E_1$-term by degree reasons established in (ii). So there is a natural filtration on $\on{H}^*$ with associated graded $\oplus_\la \on{H}_c^*(S_\la,-)$. Explicitly, this is a filtration indexed by $(\xcoch(T),\leq)$ defined as
\[\on{Fil}_{\geq \mu}\on{H}^*(\mA)=\ker(\on{H}^*(\mA)\to \on{H}^*(S_{<\la},\mA)),\] 
where $S_{<\la}=\bar{S}_\la-S_\la$. On the other hand, the stratification of $\Gr$ by $$\left\{S^-_\la:=i^-(q^-(\la)), \la\in\xcoch(T)\right\}$$ gives another filtration of $\on{H}^*$ as
\[\on{Fil}'_{<\la}\on{H}^*(\mA)=\on{Im}(\on{H}^*_{S^-_{<\la}}(\mA)\to \on{H}^*(\mA)),\]
where $\on{H}^*_{S^-_{<\la}}(\mA)$ denotes the cohomology of the $!$-restriction of $\mA$ to $S^-_{<\la}$.
These two filtrations are complimentary to each other (by Braden's theorem and (2)), and together define the decomposition $\on{H}^*=\oplus_\la \on{H}_c^*(S_\la,-)$.

\begin{rmk}
(i) It follows that both $\dim (S_\la\cap \Gr_{\leq \mu})=(\rho,\la+\mu)$ and $\dim \Gr_{\leq \mu}-\dim (S_\la\cap \Gr_{\leq \mu})=(\rho,\mu-\la)$ tends to infinity as $\mu$ goes to infinity. In other words, $S_\la$ is infinite dimensional and infinite codimensional in $\Gr_G$. On the other hand, it is an orbit of $LU$. For this reason, $S_\la$ is usually called a semi-infinite orbit. 

(ii) The intersection of $S_\la\cap S^-_\mu$ is finite dimensional, and depends only on $\la-\mu$. So it is usually denoted by $Z_{\la-\mu}$ and is called the Zastava space. Note that $Z_{\la-\mu}$ is non-empty if and only if $\la\geq \mu$.

(iii) It was also proved in \cite[Theorem 3.6]{MV} that the filtration and the decomposition of $\on{H}^*$ do not depend on the choice of the embeddings $B\subset G$ and $T\subset B$. See also Corollary \ref{canonical TB}.
\end{rmk}

We can regard $\on{CT}$ as a functor from $\Sat_G\to \Sat_T$, and the isomorphism in Theorem \ref{MV theory} (iv) as an isomorphism 
\begin{equation}\label{geom sat tran}
\on{H}^*\circ \on{CT}\cong \on{H}^*: \Sat_G\to \on{Vect}_\Ql.
\end{equation}
\begin{prop}
There is a unique monoidal structure on $\on{CT}$ such that the isomorphism \eqref{geom sat tran} is monoidal. 
\end{prop}
The monoidal structure on $\on{CT}$, if exists, is clearly unique. The existence was proved in \cite[Proposition 6.4]{MV} using the fusion product interpretation of the convolution product. Here we sketch an alternative approach, using equivariant cohomology.
\begin{proof}
The monoidal structure on $\on{CT}$, if exists, is clearly unique. We sketch the existence.

Recall from the above discussion that there is a filtration $\on{Fil}^*_{\geq \mu}$ on $\on{H}^*$ such that $\on{CT}=\on{gr}\on{H}^*$. It is enough to show that the monoidal structure given in Proposition \ref{fiber functor} is compatible with this filtration. 
Then the Verdier dual argument shows that the complementary filtration is also compatible with the monoidal structure. By the construction of the isomorphism $\on{H}^*\cong \on{CT}$ from the above discussion, the proposition then follows.

The obvious problem is that the monoidal structure of $\on{H}^*$ is constructed using the $L^+G$-equivariant cohomology whereas $S_\la$ is not $L^+G$-equivariant. However, $S_\la$ is stable under the action of the constant torus $T\subset L^+T\subset L^+G$.
Then the stratification $\{S_\la\}$ also induces a spectral sequence with $E_1$-term $\on{H}_{T,c}^*(S_\la,-)$ and abutment $\on{H}^*_{T}$. By degree reasons (cf. Theorem \ref{MV theory} (2)), the spectral sequence degenerates at $E_1$-term so one can lift the filtration $\on{Fil}_{\geq \mu}\on{H}^*$ to a filtration of $\on{H}^*_{T}$ by free $R_{T}$-modules
\[\on{Fil}_{\geq \mu}\on{H}_{T}^*(\mA)=\ker(\on{H}_{T}^*(\mA)\to \on{H}_{T}^*(S_{<\la},\mA)),\]
so that the associate graded is $\oplus_\la\on{H}_{T,c}(S_\la,\mA)$ (which is in particular a free $R_{T}$-module). 
In addition, after base change along the augmentation map $R_{T}\to\Ql$, one recovers the filtration on $\on{H}^*$.
See \cite[\S\ 2.2]{YZ} for more details.

Since $\on{H}_{T}^*(\mA)\cong \on{H}^*_{L^+G}(\mA)\otimes_{R_{G}}R_{T}$, the monoidal structure on $\on{H}_{L^+G}^*$ constructed in Proposition \ref{monoidal functor} induces a monoidal structure on $\on{H}_{T}^*$ by base change. Therefore it is enough to show that the canonical isomorphism
\begin{equation}\label{equiv monoidal}
\on{H}^*_{T}(\mA_1)\otimes_{R_{T}}\on{H}^*_{T}(\mA_2)\cong \on{H}^*_{T}(\mA_1\star\mA_2)
\end{equation}
is compatible with the filtration. In addition, it is enough to check this over the generic points of $\Spec R_{T}$. Let  us write $\on{H}_\eta=\on{H}^*_{T}\otimes_{R_{T}}Q$, where $Q$ is the fractional field of $R_{T}$.
Applying the equivariant localisation theorem (cf. Theorem \ref{equiv localization}) we have
\[\on{H}_\eta(\mA)\cong \bigoplus_\nu \on{H}_\eta(\mA|_{t^\nu}),\quad \on{H}_\eta(S_{<\la},\mA)\cong\bigoplus_{\nu<\la}\on{H}_\eta(\mA|_{t^\nu}).\]
Therefore, $\on{Fil}_{\geq \la}\on{H}_\eta(\mA)=\oplus_{\nu\geq \la}\on{H}_\eta(\mA|_{t^\nu})$.
In addition, it follows from the definition of the monoidal structure on $\on{H}^*_{T}$ that under these identifications, \eqref{equiv monoidal} is induced by the canonical isomorphisms
\[\begin{split}\bigoplus_{\nu_1+\nu_2=\nu} \on{H}_\eta(\mA_1|_{t^{\nu_1}})\otimes_Q \on{H}_\eta(\mA_2|_{t^{\nu_2}})\cong &\bigoplus_{\nu_1+\nu_2=\nu}\on{H}_\eta((\mA_1|_{t^{\nu_1}})\otimes (\mA_2|_{t^{\nu_2}})) \\
\cong &\quad \on{H}_\eta(\mA_1\star\mA_2|_{t^\nu}),
\end{split}\]
where the first isomorphism is the equivariant K\"unneth formula \eqref{equiv Kunneth} and the second is again the localisation theorem. 

Now the proposition follows by putting everything together.
\end{proof}

\begin{rmk}\label{Parabolic MV}
Mirkovi\'c-Vilonen's theory generalises to parabolic subgroups. Let $P\subset G$ be a parabolic subgroup and $L_P$ be the Levi quotient. Then we have the induced diagram
\[\Gr_L\stackrel{q}{\leftarrow}\Gr_P\stackrel{i}{\to}\Gr_G.\]
Given $\mA\in\Sat_G$, one can then define $q_!i^*\mA$ as a complex of sheaves on $\Gr_L$. By choosing an appropriate one-parameter subgroup of $G$ we can also regard it as a hyperbolic localization functor and therefore $q_!i^*\mA$ is perverse (up to some shifts). Then we have the parabolic version of $\on{CT}=\on{CT}^G_T$
\[\on{CT}^G_L: \Sat_G\to \Sat_L,\quad \mA\mapsto q_!i^*\mA[\mbox{shifts}],\]
such that that $\on{CT}^G_T=\on{CT}^L_T\circ\on{CT}^G_L$. 
\end{rmk}

Applying Proposition \ref{existence of comm}, we see that $\on{CT}$ in fact respects to the symmetric monoidal structure, and thus is a tensor functor between two Tannakian categories. It thus induces a homomorphism
\[\hat{T}\cong \tilde T\to \tilde G.\]
This defines a subtorus of $\tilde G$. 
\begin{lem}
This is a maximal torus.
\end{lem} 
\begin{proof}
Indeed, let $L_\mu=\on{H}^*(\IC_\mu)$ be an irreducible representation of $\tilde G$. Then under the decomposition given by Theorem \ref{MV theory} (3), $\on{H}_c^{(2\rho,\la)}(S_\la,\IC_\mu)$ 
is the $\la$-weight subspace of $L_\mu$.
By \eqref{extreme semi},
$\dim \on{H}_c^{(2\rho,\mu)}(S_\mu,\IC_\mu)=1$. This implies that all the characters of $L_\mu$ are linearly independent. Therefore, the representation ring of $\tilde G$ is a subring of the representation ring of $\tilde T$, which implies that $\tilde T$ is maximal.
\end{proof}
\begin{rmk}Note that $T$ (defined as the abstract Cartan) is not a subgroup of $G$, but its Langlands dual $\hat T$ is canonically a subgroup of $\tilde G$.
\end{rmk}

We continue the proof of the theorem.

Let $\ell_\mu=\on{H}_c^{(2\rho,\mu)}(S_\mu,\IC_\mu)$.
Note that the natural map (up to scalar) 
$$L_{\mu_1}\otimes L_{\mu_2}\to L_{\mu_1+\mu_2}$$ sends $\ell_{\mu_1}\otimes\ell_{\mu_2}$ to $\ell_{\mu_1+\mu_2}$, since the monoidal structure on $\on{H}^*$ is compatible with the filtration $\on{Fil}$ on $\on{H}^*$.  Therefore via the Pl\"ucker relation, the collection $\{\ell_\mu\subset L_\mu, \mu\in\xcoch(T)^+\}$ defines a Borel subgroup $\tilde B\subset \tilde G$, and $\tilde T\subset \tilde B$.  
It follows by definition that $\mu$ is the highest weight (with respect to $\tilde B$) of the representation $L_\mu$. Therefore, the set of dominant coweights $\xcoch(T)^+$  coincide with the set of dominant weights $\xch(\tilde T)^+$ of $\tilde T$ with respect to $\tilde B$. 

Recall the following description of the semigroup $\tilde Q^+\subset \xch(\tilde T)$ generated by positive roots of $\tilde G$. A weight $\la\in \tilde Q^+$ if and only if there exists a highest weight representation $L_\mu$ such that $\mu-\la$ appears as a weight of $L_\mu$. By Theorem \ref{MV theory}, this is equivalent to the existence of some $\mu$ such that $t^{\mu-\la}\in \Gr_{\leq \mu}$, which is equivalent to $\la$ being a sum of positive coroots of $G$.
Therefore the semigroup  $(Q^\vee)^+\subset\xcoch(T)$ generated by positive coroots of $G$ coincides with the semigroup $\tilde Q^+$. It then follows that the set of simple coroots of $G$ coincide with the set of simple roots of $\tilde G$. 
To finish the proof of Theorem \ref{birth of dual group}, it remains to prove the following lemma.
\begin{lem}
A root datum $(\xch,\xcoch,\Phi,\Phi^\vee)$ is uniquely determined by the semi-group ${\xch}^+$ of dominant weights and the set $\Delta\subset\Phi$ of simple roots.
\end{lem}
\begin{proof}Indeed, ${\xch}^+$ determines the dominant Weyl chamber of $\xch\otimes \bR$, and therefore the Weyl group $W$. Then $\Phi=W(\Delta)$. We leave the details as an exercise.
\end{proof}
Note that the above arguments also give the following corollary.

\begin{cor}\label{canonical TB}
The grading on the hypercohomology functor $\on{H}^*$ corresponds to the one parameter subgroup 
$$2\rho:\bG_m\to \tilde T\subset \tilde G,$$ 
where the cocharacter $2\rho\in \xcoch(\tilde T)=\xch(T)$ is the sum of positive coroots of $\tilde G$. 

In particular, $\tilde T$ is the centraliser of $2\rho(\bG_m)$ in $\tilde G$ and $\tilde B$ is the unique Borel that contains $\tilde T$ and such that $2\rho$ is dominant with respect to $\tilde B$. Therefore, $(\tilde T,\tilde B)$ is independent of the choice of $T\subset B$ and $B\subset G$.
\end{cor}

At the end of this subsection, we give a few applications of the geometric Satake equivalence. First it produces certain ``canonical bases'' for representations of $\hat G$.

\begin{thm}\label{MV basis}~
\begin{enumerate}
\item Via Theorem \ref{MV theory} (2), irreducible components of $S_\la\cap \Gr_{\leq \mu}$ provide a natural basis of the $\la$-weight space of $L_\mu$.

\item Via Corollary \eqref{Satake fiber}, irreducible components of $\Gr_{\leq \mmu}\cap m^{-1}(t_\la)$ of dimension $(\rho,|\mmu|-\la)$ provide a natural basis of $\Hom_G(V_\la, V_{\mu_1}\otimes\cdots\otimes V_{\mu_n})$.
\end{enumerate}
\end{thm}
Bases in (1) are usually called Mirkovi\'c-Vilonen bases, and irreducible components of $S_\la \cap \Gr_{\leq \mu}$ are usually called MV cycles. Bases in (2) are sometimes called Satake bases.

Next, we give some cute applications of the geometric Satake equivalence to topology.

\begin{ex}(i) We have a canonical isomorphism 
$$\on{H}^*(\Gr(i,n))\cong \wedge^i (\bigoplus_{j=0}^{n-1}\Ql[-2j](-j)).$$

(ii) Let $X_n\subset \bP^{n+1}$ be the smooth $n$-dimensional quadric. Then
\[
\on{H}^*(X_n)\cong\left\{\begin{array}{ll}
\bigoplus_{j=0}^n \Ql[-2j](-j) &  n \mbox{ is odd,} \\ 
\bigoplus_{j=0}^n \Ql[-2j](-j)\bigoplus \Ql[-n](-\frac{n}{2}) &  n \mbox{ is even.}
\end{array}\right.
\] 
The proof is by identifying $\Gr(i,n)$ and $X_n$ as minuscule Schubert varieties of $\on{PGL}_n$, $\on{SO}_{n+2}$.
\end{ex}

Recall that if $k=\bC$, $\Gr\cong \Omega K$ and therefore is an $H$-space. As a result, the cohomology ring $\on{H}^*(\Gr)$ has a natural graded commutative and cocommutative Hopf algebra structure. If $k$ is general, this still holds. Namely, the convolution map $m:\Gr\tilde\times \Gr\to \Gr$ induces a comultiplication for the equivariant cohomology
\[\on{H}^*_{L^+G}(\Gr)\to \on{H}^*_{L^+G}(\Gr)\otimes_{R_G}\on{H}^*_{L^+G}(\Gr).\]
The similar statement in Lemma \ref{left=right} holds for $\on{H}^*_{L^+G}(\Gr)$. Therefore, we can specialise along $R_G\to\Ql$, which induces
a comultiplication of usual cohomology. 

By the classification of such Hopf algebras, $\on{H}^*(\Gr)=U(\fraka)$ (or $\on{H}^*(\Omega K)=U(\fraka)$ if $k=\bC$), where $\fraka$ is the primitive part of $\on{H}^*(\Gr_G)$. Since $\on{H}^*(\Gr)$ acts on $\on{H}^*(\mA)$ by the cup product, a general Tannakian formalism induces a map
$\fraka\to \hat \frakg=\Lie \hat G$. 

\begin{thm} \label{canonical pinning}We fix an isomorphism $\Ql(1)\simeq \Ql$.
\begin{enumerate}
\item The Chern class $c_1(\mO(1))\in \fraka$ and its image in $\hat \frakg$ is a principal nilpotent element $N$ of $\hat \frakg$. 

\item The map $\fraka\to\hat\frakg$ is injective and identifies $\fraka$ with the centraliser of $N$ in $\hat\frakg$.

\item $(\tilde G,\tilde B,\tilde T, N)$ is a canonical pinning of $\tilde G$.
\end{enumerate}
\end{thm}
(1) and (2) were proved in \cite[\S 4]{Gi1} and \cite{YZ}. We only mention that $N$ is principal nilpotent is a consequence of the hard Lefschetz theorem for intersection cohomology (cf. \cite[Lemma 4.4.2]{Gi1}). (3) was proved in \cite{YZ}.
This theorem in particular gives a description of the based loop group of a compact Lie group when $k=\bC$.
\begin{rmk}
(i) There is an integral version of this theorem, formulated in terms of homology of $\Gr$ (\cite[Theorem 1.1]{YZ}). The proof relies on the integral version of the geometric Satake equivalence \cite{MV}.

(ii) One can similarly describe the equivariant cohomology $\on{H}^*_{L^+G}(\Gr)$ by the regular centraliser of $\hat\frakg$, see \cite[Theorem 2.12]{BFM} and \cite[Proposition 6.6]{YZ}. This is the first step towards the study of $L^+G$-equivariant derived category $D_{L^+G}(\Gr)$ (also known as the derived Satake category) as in \cite{BF}.
\end{rmk}

\subsection{Fusion product.}\label{fusion}
The factorization structure of the Beilinson-Drinfeld Grassmannians allow us to interpret the convolution product \eqref{Lus conv} as a fusion product, which in term allows one to easily deduce Proposition \ref{fiber functor} and Proposition \ref{existence of comm} together.

First, by the same procedure to define $\on{P}_{L^+G}(\Gr_G)$ (as in \S~\ref{A:ind}), we have the abelian category $\on{P}_{(L^+G)_{X^I}}(\Gr_{X^I})$ of  $(L^+G)_{X^I}$-equivariant perverse sheaves on $\Gr_{X^I}$. Following Reich (cf. \cite{Rei}), we define $\Sat_{X^I}\subset \on{P}_{(L^+G)_{X^I}}(\Gr_{X^I})$ as the full subcategory of those sheaves that are ULA with respect to the morphism $q_I:\Gr_{X^I}\to X^I$ (see \S~\ref{ULA sheaf} for a review of the ULA property). By Theorem \ref{III:factorization} and Theorem \ref{ULA perverse}, for every $\phi: J\to I$, there is a functor
\begin{equation}\label{Sat restriction}
\Delta(\phi)^\bullet:=(\nu^\phi)^*[|I|-|J|]\cong(\nu^\phi)^![|J|-|I|]:\Sat_{X^J}\to\Sat_{X^I}.
\end{equation}

Given a surjective map $\phi:J\twoheadrightarrow I=\{1,2,\ldots,n\}$,  we have the convolution map from the convolution Grassmannian to the Beilinson-Drinfeld Grassmannian
\[m_\phi:\Gr_{X^{J_1}}\tilde\times\cdots\tilde\times\Gr_{X^{J_n}}\to \Gr_{X^J}.\]
For $\mA_i\in\Sat_{X^{J_i}}$, one can form the ``external twisted product'' $\mA_1\tilde\boxtimes\cdots\tilde\boxtimes\mA_n$ on $\Gr_{X^{J_1}}\tilde\times\cdots\tilde\times\Gr_{X^{J_n}}$ using \eqref{torsor E}. Then one can define the ``external convolution product'' 
\[\mA_1\boxstar \mA_2\boxstar \cdots\boxstar \mA_n:=(m_\phi)_!(\mA_1\tilde\boxtimes\cdots\tilde\boxtimes\mA_n).\]
Here is the crucial statement.

\begin{prop}\label{conv=fusion}~
\begin{enumerate}
\item The complex $\mA_1\boxstar\mA_2\boxstar\cdots\boxstar\mA_n$ belongs to $\Sat_{X^{J}}$ and therefore, there is the well defined external convolution product functor
\[\boxstar_{i\in I}:\Sat_{X^{J_1}}\times\cdots\times\Sat_{X^{J_n}}\to \Sat_{X^J},\quad (\mA_1,\ldots,\mA_n)\mapsto \mA_1\boxstar\cdots\boxstar\mA_n.\]

\item There is a canonical isomorphism
\begin{equation}\label{eqfusion}
\mA_1\boxstar\cdots\boxstar\mA_n\cong j_{!*}((\mA_1\boxtimes\cdots\boxtimes\mA_n)|_{X^{(\phi)}})
\end{equation}
where $j: X^{(\phi)}\to X^J$ denotes the open embedding and $(\mA_1\boxtimes\cdots\boxtimes\mA_n)|_{X^{(\phi)}}$ is regarded as a perverse sheaf on $\Gr_{X^J}\times_{X^J}X^{(\phi)}$ via the isomorphism \eqref{BD factorization}.
\end{enumerate}
\end{prop}
\begin{proof}First note that if all $\mA_i$s are ULA, then so is $\mA_1\tilde\boxtimes\cdots\tilde\boxtimes\mA_n$. This is because
locally (in smooth topology), the support $\on{Supp}(\mA_1\tilde\times\cdots\tilde\times\mA_n)$ is isomorphic to the product $\on{Supp}(\mA_1)\times\cdots\times\on{Supp}(\mA_n)$. Then the claim follows from Theorem \ref{ULA prop} (3) and (4)'.

Now applying Theorem \ref{ULA prop} (2), we see that $\mA_1\boxstar\cdots\boxstar\mA_n$ is ULA with respect to $\Gr_{X^J}\to X^J$. In addition, it is perverse when restricted to $X^{(\phi)}$, since under the isomorphism \eqref{BD factorization} 
\[\mA_1\boxstar \mA_2\boxstar \cdots\boxstar \mA_n|_{X^{(\phi)}}=\mA_1\boxtimes\cdots\boxtimes\mA_n|_{X^{(\phi)}},\]
by Lemma \ref{generic conv}. 
Therefore, it follows by Theorem \ref{ULA perverse} that it is perverse and  $\mA_1\boxstar\cdots\boxstar\mA_n\cong j_{!*}((\mA_1\boxtimes\cdots\boxtimes\mA_n)|_{X^{(\phi)}})$.
Finally, the $(L^+G)_{X^J}$-equivariance is clear.
\end{proof}

Let us deduce consequences ((i)-(iii) below) from this proposition. In particular we will obtain another proof of Proposition \ref{fiber functor} and a proof of Proposition \ref{existence of comm}.

(i) Combining with \eqref{Sat restriction}, we can define the internal convolution product functor on each $\Sat_{X^J}$ as
\begin{equation}\label{internal convolution}
\ostar_{i\in I}: \times_{i\in I}\Sat_{X^J}\to \Sat_{X^J},\quad \mA_1\ostar\mA_2\ostar\cdots\ostar\mA_n:= \Delta^\bullet(\mA_1\boxstar\mA_2\boxstar\cdots\boxstar\mA_n),
\end{equation}
where $\Delta: X^I\to (X^I)^n$ is the diagonal embedding and $\Delta^\bullet=\Delta^*[-(n-1)|I|]$. By the same reasoning as for $\Sat_G$, each $\Sat_{X^J}$ is a monoidal category.

\begin{rmk}\label{triv group}
One may try to define the internal convolution product $\ostar$ on the whole $\on{P}_{(L^+G)_{X^J}}(\Gr_{X^J})$ using the same formula \eqref{internal convolution}. However, the following example shows that $\ostar$ is not perverse exact on this bigger category and this is one of the reasons we impose the ULA condition on the sheaves.

Assume $G=\{e\}$ is trivial. Then $\on{P}_{L^+G}(\Gr_X)$ is nothing but the category of perverse sheaves on $X$ and $\Sat_X$ is the category of local systems on $X$ shifted to degree $-1$. Then by definition, 
$$\mA_1\boxstar\mA_2=\mA_1\boxtimes\mA_2,\quad \mA_1\ostar\mA_2=\mA_1\otimes\mA_2[-1].$$
In particular, if either $\mA_1$ or $\mA_2$ is a skyscraper sheaf, then $\mA_1\ostar\mA_2$ is not perverse. 
\end{rmk}

Note that we have a natural functor
\[\Sat_G\to \Sat_X,\quad \mA\mapsto \mA_X:=\Ql\tilde\boxtimes\mA[1],\]
where $\mA_X=\Ql\tilde\boxtimes\mA[1]$ is the external twisted product on $\Gr_X=X\tilde\times\Gr$ (see \eqref{III: twisted constr}). 
\begin{lem}~
\begin{enumerate}
\item The functor $\mA\mapsto \mA_X$ is fully faithful.
\item For each $n$, there is a canonical isomorphism 
$$(\mA_1\star\cdots\star\mA_n)_X\cong (\mA_1)_X\ostar\cdots\ostar(\mA_n)_X.$$
In particular, there is a natural monoidal structure on the functor $\Sat_G\to \Sat_X,\ \mA\mapsto \mA_X$. 
\end{enumerate}
\end{lem}
\begin{proof}Note that both maps in the diagram $X\tilde\times \on{Supp}(\mA)\leftarrow \hat X^+\times \on{Supp}(\mA)\to \on{Supp}(\mA)$ can be written as inverse limit of smooth surjective maps with geometrically connected fibers. Then (1) follows from \cite[Proposition 4.2.5]{BBD}. (2) follows from the canonical isomorphism 
$$X\tilde\times (\Gr\tilde\times\cdots\tilde\times\Gr)\cong \Gr_X\tilde\times\cdots\tilde\times\Gr_X|_\Delta.$$
\end{proof}
\begin{rmk}\label{conv perverse}
Note that this lemma together with Proposition \ref{conv=fusion} gives a proof of Proposition \ref{conv prod}. In fact, it proves a stronger statement (as in the remark after Proposition 6 of \cite{Ga}) since there is no need to assume that $\mA_1$ is $L^+G$-equivariant.
\end{rmk}

(ii) Note that $(q_I)_*\mA$ is ULA on $X^I$. Let
\begin{equation}\label{relative coh}
\mH^*(\mA):=\bigoplus_i \mH^i(q_I)_*\mA[i].
\end{equation} 
This is a graded local system on $X^I$ by Theorem \ref{ULA prop} (1). It follows that there is a canonical isomorphism
\begin{equation}\label{monoidal functor}
\mH^*(\mA_1\boxstar\cdots\boxstar\mA_n)\cong \mH^*(\mA_1)\boxtimes\cdots\boxtimes\mH^*(\mA_n).
\end{equation}
since this holds over $X^{(\phi)}$ by the K\"unneth formula. We thus obtain a functor $\mH^*: \Sat_{X^I}\to \on{Loc}^{\on{gr}}(X^I)$ to the category of graded local systems on $X^I$, compatible with the external convolution and tensor product.   As a consequence, the functor $\mH^*$ has a natural monoidal structure with respect to the internal convolution and tensor product
\[\mH^*: (\Sat_{X^I},\ostar)\to (\on{Loc}^{\on{gr}}(X^I),\otimes).\]
For every  point $i_x:\{x\}\to X$, taking the stalk $i_x^*$ is a symmetric monoidal functor from $\on{Loc}^{\on{gr}}(X)$ to the category of graded vector spaces $\on{Vect}^{\on{gr}}_{\Ql}$, both equipped with the graded commutativity constraints $c_{\on{gr}}$. Although $i_x^*$ depends on the choice of the point $x$,
the composition
\begin{equation}\label{decomp coh}
\Sat_G\to \Sat_X\stackrel{\mH}{\to} \on{Loc}^{\on{gr}}(X)\stackrel{i_x^*}{\to}\on{Vect}^{\on{gr}}_{\Ql}
\end{equation}
does not, and is canonically isomorphic to the hypercohomology functor $\on{H}^*$ by the proper base change. We thus also obtain a monoidal structure on $\on{H}^*$. We leave it as an exercise to show that this monoidal structure on $\on{H}^*$  coincides with the one constructed in Proposition \ref{fiber functor}.

(iii) The expression given in the right hand side of \eqref{eqfusion} is usually called the fusion product of $\mA_1,\cdots,\mA_n$. This interpretation gives more visible symmetries of the convolution product, coming from the symmetry of $X^I$. The following simple statements are enough to prove Proposition \ref{existence of comm}. For a more general discussion, see Remark \ref{symmetry on fusion}.

Let $\mA_1,\mA_2\in\Sat_X$. 
Recall that there is an involution $\Delta(\sigma):\Gr_{X^2}\cong\Gr_{X^2}$, see Proposition \ref{III:Gr2}. It follows from that proposition and \eqref{eqfusion} that 
there is a canonical isomorphism 
$$\Delta(\sigma)^*(\mA_1\boxstar\mA_2)\cong \mA_2\boxstar\mA_1$$ 
such that the following diagram is commutative
\[\xymatrix{
\sigma^*\mH^*(\mA_1\boxstar\mA_2)\ar^\cong[r]\ar_\cong[d]& \mH^*\Delta(\sigma)^*(\mA_1\boxstar\mA_2)\ar^\cong[r]&\mH^*(\mA_2\boxstar\mA_1)\ar^\cong[d]\\
\sigma^*(\mH^*(\mA_1)\boxtimes\mH^*(\mA_2))\ar_{c_{\on{gr}}}^\cong[rr]&&\mH^*(\mA_2)\boxtimes\mH^*(\mA_1),
}\]
where $c_{\on{gr}}$ is the isomorphism (say at $(x,y)\in X^2$) given by
\[\begin{split}\sigma^*(\mH^*(\mA_1)\boxtimes\mH^*(\mA_2))_{(x,y)}=\mH^*(\mA_1)_y\otimes\mH^*(\mA_2)_x \\ \stackrel{c_{\on{gr}}}{\cong}\mH^*(\mA_2)_x\otimes \mH^*(\mA_1)_y=(\mH^*(\mA_2)\boxtimes\mH^*(\mA_1))_{(x,y)}.
\end{split}\]

Therefore, 
\begin{cor}\label{existence of comm2}
There is a canonical isomorphism $c'_{\mA_1,\mA_2}:\mA_1\ostar\mA_2\cong\mA_2\ostar\mA_1$ such that the following diagram is commutative
\[\begin{CD}
\mH^*(\mA_1\ostar\mA_2)@>c'_{\mA_1,\mA_2}>\cong>\mH^*(\mA_2\ostar\mA_1)\\
@V\cong VV@VV\cong V\\
\mH^*(\mA_1)\otimes\mH^*(\mA_2)@>c_{\on{gr}}>>\mH^*(\mA_2)\otimes\mH^*(\mA_1).
\end{CD}\]
\end{cor}
Finally, since $\Sat_G$ is a full subcategory of  $\Sat_X$ and the hypercohomology functor $\on{H}^*:\Sat_G\to \on{Vect}_{\Ql}$ is isomorphic to \eqref{decomp coh}, Proposition \ref{existence of comm} follows.

\begin{rmk}\label{symmetry on fusion}
More generally, for every $\phi: J\twoheadrightarrow I$ ($I$ non-ordered) we have a functor
\[\underset{\phi}{\fstar}: \times_{i\in I}\Sat_{X^{J_i}}\to \Sat_{X^J},\quad ( \mA_{i}, i\in I)\mapsto \underset{\phi}{\fstar}\mA_i:=j_{!*}((\boxtimes_i \mF_i)|_{X^{(\phi)}}).\]
(That the right hand side is ULA with respect to $q_I$ is a consequence of Proposition \ref{conv=fusion}.) Then by Theorem \ref{III:factorization}, we have
for every $K\stackrel{\phi}{\twoheadrightarrow} J\stackrel{\psi}{\twoheadrightarrow} I$ canonical isomorphisms of functors
\begin{equation}\label{fusion prod}
\underset{\phi}{\fstar}= \underset{\psi\phi}{\fstar}(\times_{i\in I} \underset{\phi_i}{\fstar}): \times_{j} \Sat_{X^{K_j}}\to \Sat_{X^K},
\end{equation}
and
\begin{equation}
\Delta(\phi)^\bullet \underset{\psi\phi}{\fstar}= \underset{\psi}{\fstar}(\times_{i\in I}\Delta(\phi_i)^\bullet): \times_{i}\Sat_{X^{K_i}}\to \Sat_{X^J}.
\end{equation}
It directly follows from these formulas and \cite[Proposition 1.5]{DM} that $(\Sat_{X^I},\ostar)$ has a tensor category structure and that $\mH^*: (\Sat_{X^I},\ostar)\to (\on{Loc}^{\on{gr}}(X^I),\otimes)$ has a tensor structure.
\end{rmk}

\begin{ex}
We assume that $G=\GL_n$. The above construction of the commutativity constraints via fusion product and Example \ref{momega1BD} together imply 
that under the identification of $\Gr_{\omega_1,\ldots,\omega_1}\cap m^{-1}(t^\la)$ with a Springer fiber as in Example \ref{conv v.s. springer}, 
the action of the symmetric group $S_m$ on $\Hom(V_\la, V_{\omega_1}^{\otimes m})$ coincides with the Springer action up to a sign. In particular, it gives the Schur-Weyl duality.
Note that this identification is also the starting point of Laumon's construction of automorphic sheaves from the Springer sheaf (cf. \cite{FGKV}).
\end{ex}

\subsection{Bootstraps.}
To apply the geometric Satake correspondence to the (geometric) Langlands program, it is important not to assume that $k$ is algebraically closed, and it is important to extend it to equivalences between
$\Sat_{X^I}$ (for various $I$) and certain categories involving the Langlands dual group $\hat G$. 
In this subsection, we explain how to bootstrap the previous results to obtain such extensions.

We begin with some general nonsense.
\begin{dfn}
Let $K$ denote a linear algebraic group over a field $E$, and $\mC$ be an $E$-linear abelian category with infinite direct sums (so it makes sense to talk about tensoring an object in $\mC$ by an $E$-vector space). 
A representation of $K$ in $\mC$ is an object $M$ with a morphism $\rho: M\to M\otimes_{E}\mO_K$ satisfying the usual commutative diagram for a coaction. A morphism $f: (M_1,\rho_1)\to (M_2,\rho_2)$ of $K$-representations is a morphism $f:M_1\to M_2$ in $\mC$ intertwining $\rho_i$. The category of $K$-representations in $\mC$ is denoted by $\on{Rep}(K,\mC)$.  

In general if $\mC$ does not necessarily admit infinite direct sums, one first passes to its the ind-completion $\on{Ind}\mC$ and defines $\on{Rep}(K,\mC)$ as the full subcategory of $\on{Rep}(K,{\on{Ind}\mC})$ consisting of those $(M,\rho)$ whose underlying object $M$ is in $\mC$. One immediately checks that  if $\mC$ admits infinite direct sums, the new definition coincides with the old one.
\end{dfn}

\begin{ex}If $\mC=\on{Vect}_E$ is the category of finite dimensional $E$-vector spaces, then  $\on{Rep}(K,\mC)$ is the category of finite dimensional algebraic representations of $K$, and is denoted by $\Rep_E(K)$ for simplicity.
\end{ex}
\begin{ex}\label{KLoc}
Let $\Ga$ be a profinite group, and let $\on{Rep}^c_\Ql(\Ga)$ denote the category of continuous $\ell$-adic representations of $\Ga$. Let $\on{Rep}(K,{\on{Rep}^c_\Ql(\Ga)})$ be the above defined category of representations of $K$ in $\on{Rep}^c_\Ql(\Ga)$. Then it is easy to see that an object in $\on{Rep}(K,{\on{Rep}^c_\Ql(\Ga)})$ is nothing but a representation of $\Ga\times K$, that is continuous in the first factor and algebraic in the second factor. 
\end{ex}

Note that if $F:\mC\to \mD$ is an additive functor, it canonically lifts a functor $F:\on{Rep}(K,{\mC})\to \on{Rep}(K,{\mD})$.

Some classical constructions in representation theory generalise to this setting. We mention a few. Given a $K$-linear representation $V$ and an object $M\in \on{Rep}(K,{\mC})$, the tensor product $M\otimes V$ makes sense as an object in $\on{Rep}(K,{\mC})$. If $\mC$ is a monoidal abelian category, and $M,N\in \on{Rep}(K,{\mC})$, then $M\otimes N$ admits a natural diagonal $K$-action and therefore the monoidal structure can be upgraded to a monoidal structure on $\on{Rep}(K,{\mC})$. We can also define the invariant $M^K$ as the equaliser of the map $\rho:M\to M\otimes \mO_K$ and the map $M=M\otimes_EE\stackrel{e_K}{\to} M\otimes_E\mO_K$. Then $M\mapsto M^K$ is a left exact functor from $\on{Rep}(K,{\mC})$ to $\mC$.

Now let $k$ be a not necessarily algebraically closed field and $\bar k$ its separable closure. Let $\Ga_k=\Gal(\bar k/k)$.
Let $G$ be a connected reductive group over $k$, $\underline G=G\otimes_k\mO$, and let $\Gr$ denote the affine Grassmannian of $\underline G$. 
Let $\hat G=\Aut^{\otimes}(\on{H}^*)$ be the Tannakian group of $\Sat_{G_{\bar k}}$, equipped with the canonical pinning from Proposition \ref{canonical pinning}\footnote{We do not really need Theorem \ref{birth of dual group} to identify it with the one arising from the usual combinatoric construction.}. We will first explain how to define the full Langlands dual group ${^L}G$ as a semi-direct product of $\hat G$ and $\Ga_k$. In fact, as we shall see, there are two versions of ${^L}G$.

As explained in \cite{RZ}, there is a canonical action $\on{act}^{\on{geo}}$ of $\Ga_k$ on $\hat G$ defined as follows. Every $\ga\in \Ga_k$ induces an automorphism $\Gr\otimes \bar k\cong \Gr\otimes \bar k$ and therefore the pullback $\ga^*$ induces a tensor automorphism of $\Sat_{G_{\bar k}}$.  In addition, there is a canonical isomorphism 
$\al_\ga: \on{H}^*\cong \on{H}^*\circ \ga^*$
of tensor functors. 
Then for every $g\in \hat G$, regarded as an automorphism of the tensor functor $\on{H}^*$, $\ga(g)$ is the tensor automorphism of $\on{H}^*$ defined as
\begin{equation}\label{gaaction}
\on{H}^*\stackrel{\al_\ga}{\longrightarrow}\on{H}^*\circ\ga^*\stackrel{g\circ \id}{\to}\on{H}^*\circ\ga^*\stackrel{\al^{-1}_\ga}{\to}\on{H}^*.
\end{equation}
Since the isomorphism $\al_\ga$ preserves the cohomological degree,  $\on{act}^{\on{geo}}$ preserves $2\rho:\bG_m\to \hat G$ and therefore preserves
$\tilde T\subset \tilde B\subset \hat G$ by Corollary \ref{canonical TB}. But it does not preserve the principal nilpotent element $N$ from Proposition \ref{canonical pinning}.

Let $\hat G_\ad$ be the adjoint quotient of $\hat G$. Then cocharacter $\bG_m\stackrel{2\rho}{\to} \hat G\to \hat G_\ad$ admits a unique square root, denoted by $\rho_\ad$.
Let $\on{cycl}:\Ga_k\to\bZ_\ell^\times$
be the cyclotomic character of $\Ga_k$ defined by the action of
$\Ga_k$ on the $\ell^\infty$-roots of unity of $\bar k$.  We define a homomorphism
\[\chi:\Ga_k\stackrel{\on{cycl}}{\to} \bZ_\ell^\times\stackrel{\rho_\ad}{\to} \hat G_{\ad}(\Ql),\]
which induces an action $\Ad_{\chi}:\Ga_k\to\Aut(\hat G)$ of $\Ga_k$ on 
$\hat G$ by inner automorphisms. 
\begin{lem}\label{comparison}The action $\on{act}^{\on{alg}}:=\on{act}^{\on{geo}}\circ \Ad^{-1}_{\chi}$ preserves the pinning of $\hat G$ defined in Proposition \ref{canonical pinning}.
\end{lem}
\begin{proof}This follows from the fact that the first Chern class $c_1(\mO(1))$ is an element in $\on{H}^2(\Gr_G\otimes\bar k, \Ql(1))^{\Ga_k}$. See \cite[Proposition A.6]{RZ} for details.
\end{proof}
\begin{rmk}
This remark is not needed in the sequel. Recall that there is always an action $\psi: \Ga_k\to \on{Out}(G\otimes \bar k)$ of $\Ga_k$ on $G\otimes \bar k$ by outer automorphisms. On the other hand,
by Theorem \ref{birth of dual group}, there is a canonical isomorphism $\on{Out}(G\otimes\bar k)\cong \on{Out}(\hat G)=\Aut(\hat G, \tilde B, \tilde T, N)$. As explained in \cite{RZ}, under this isomorphism, $\on{act}^{\on{alg}}$ coincides with $\psi$. 
\end{rmk}
We define the full Langlands dual group of $G$ as
\[{^L}G^{\on{alg}}:= \hat G\rtimes_{\on{act}^{\on{alg}}}\Ga_k.\]
Since $\Ga_k$ is profinite and the action $\on{act}^{\on{alg}}$ of $\Ga_k$ on $\hat G$ factors through a finite quotient, we may regard ${^L}G^{\on{alg}}$ as a pro-algebraic group.

On the other hand, according to the above discussion, it is also natural to consider
\[{^L}G^{\on{geo}}:= \hat G\rtimes_{\on{act}^{\on{geo}}}\Ga_k,\]
which is more closely related to $\Sat_G$. We regard ${^L}G^{\on{geo}}$ as a topological group.

As an erratum, we pointed out that \cite[Corollary A.8]{RZ} was wrong. The correct statement is as follows.
\begin{lem}\label{geom Langlands dual}
Assume that $\chi: \Ga_k\to \hat G_\ad(\Ql)$ lifts to a continuous homomorphism $\tilde \chi:\Ga_k\to \hat G(\Ql)$. Then there is a natural isomorphism (of topological groups)
\[{^L}G^{\on{geo}}\cong{^L}G^{\on{alg}},\quad (g,\ga)\mapsto (g\tilde\chi(\ga),\ga).\]
\end{lem}

\begin{rmk}\label{lifting condition}
The lifting $\tilde \chi$ exists in the following cases. 
\begin{enumerate}
\item[(i)] The cocharacter $\rho_\ad$ lifts to a  cocharacter $\tilde \rho: \bG_m\to \hat G$. For example, this holds if $\hat G=\GL_n$. Note that we do not require $2\tilde \rho=2\rho$. 
\item[(ii)]  The cyclotomic character $\on{cycl}: \Ga_k\to \bZ_\ell^\times$ admits a square root, i.e. there exists a half Tate twist $\Ql(\frac{1}{2})$. For example, this holds if $\on{char} k=p>0$ or $k$ is a $p$-adic field ($p\neq \ell$). 
\end{enumerate}
\end{rmk}

\begin{ex}\label{geom v.s. alg dual}
Here is an example when these two groups are not isomorphic. We consider $k=\bR$ and $G=\on{PGL}_2(\bR)$. Then ${^L}G^{\on{alg}}=\SL_2\times \Gal(\bC/\bR)$ and ${^L}G^{\on{geo}}=\SL_2\rtimes \Gal(\bC/\bR)$, where the complex conjugation acts on $\SL_2$ by sending $\begin{pmatrix} a & b\\ c & d \end{pmatrix}$ to $\begin{pmatrix} a & -b\\ -c & d \end{pmatrix}$. Note that the center of ${^L}G^{\on{alg}}$ is isomorphic to $\bZ/2\times \bZ/2$ whereas the center of ${^L}G^{\on{geo}}$ is isomorphic to $\bZ/4$. Therefore, they are non-isomorphic.
\end{ex}

Although ${^L}G^{\on{geo}}$ can not be regarded as a pro-algebraic group, one can always embed it into a larger pro-algebraic group as follows. Let 
\begin{equation}\label{Deligne modification}
\hat G^T:=\hat G\rtimes\bG_m
\end{equation}
where $\bG_m$ acts on $\hat G$ by inner automorphisms through  $\rho_\ad:\bG_m\to \hat G_\ad$. For example if $\hat G=\SL_2$, then $\hat G^T=\GL_2$.
The action $\on{act}^{\on{alg}}$ naturally extends to an action of $\Ga_k$ on $\hat G^T$, denoted by the same notation. Let 
$${^L} G^T=\hat G^T\rtimes_{\on{act}^{\on{alg}}}\Ga_k,$$ which can be regarded as a pro-algebraic group. Let $d: \hat G^T\to \bG_m$ denote the natural projection. Then there is the following short exact sequence
\[1\to {^L}G^{\on{geo}}\to {^L}G^T\stackrel{d\times \on{cycl}^{-1}}{\longrightarrow}\bG_m\to 1.\]
\begin{rmk}
Note that $\hat G^T$ is in fact the dual group of a central extension of $G$ by $\bG_m$. It seems that ${^L}G^T$ is exactly the $C$-group appeared in \cite[Definition 5.3.2]{BuGe}. As suggested as above and also by \cite{BuGe}, this seems to be a more natural object in order to formulate the attached Galois representations to automorphic representations.
\end{rmk}

Let $\on{Rep}^c_{\Ql}({^L}G^{\on{geo}})$ denote the category of finite dimensional continuous representations of ${^L}G^{\on{geo}}$ that is algebraic when restricted to $\hat G$.  There is a similarly defined category  $\on{Rep}^c_{\Ql}({^L}G^T)$.  Let $B$ be the regular representation of $\bG_m$, regarded as an algebra object in the ind-completion of $\on{Rep}_{\Ql}({^L}G^T)$ via the homomorphism $d\times \on{cycl}^{-1}$. Recall that for a monoidal category $\mC$ and an algebra object $B\in \mC$, it makes sense to talk about $B$-modules in $\mC$, which form a category denoted by $B\on{-mod}(\mC)$.
Then there is an equivalence
\[B\on{-mod}(\on{Rep}^c_{\Ql}({^L}G^T))\simeq \on{Rep}^c_{\Ql}({^L}G^{\on{geo}}),\quad V\mapsto V\otimes_B\Ql,\]
where $B\to \Ql$ is the counit map (dual to $1\in \bG_m$).

Now let $\Sat_G$ denote the category of $L^+G$-equivariant perverse sheaves on $\Gr$, equipped with the convolution product (note that \eqref{Lus conv} is defined over $k$). For example if $G=\{e\}$ is trivial, this is the category $\on{Rep}^c_{\Ql}(\Ga_k)$ of continuous $\ell$-adic representations of $\Ga_k$ with the usual tensor product. 
Here is the full version of the geometric Satake equivalence.

\begin{thm}\label{full Sat equiv}
The  hypercohomology functor $\on{H}^*:\Sat_G\to \on{Rep}^c_{\Ql}(\Ga_k)$ lifts to a canonical equivalence 
\begin{equation}\label{full Sat}
F:\Sat_G\simeq \on{Rep}^c_{\Ql}({^L}G^{\on{geo}})
\end{equation} 
such that the following diagram is commutative (up to a canonical isomorphism)
\[\begin{CD}
\Sat_G@>\simeq>> \on{Rep}^c_{\Ql}({^L}G^{\on{geo}})\\
@V\on{Pull\ back}VV@VV\Res V\\
\Sat_{G_{\bar k}}@>\simeq>> \on{Rep}_{\Ql}(\hat G).
\end{CD}\]
\end{thm}
\begin{proof}
We sketch the proof (see \cite[\S~5]{Ri} for a different approach). Since the pullback of a sheaf on $\Gr$ to $\Gr\otimes \bar k$ is $\Ga_k$-equivariant, by general non-sense (\cite[Lemma A.3]{RZ}) the hypercohomology factors in the way as stated in the theorem such that the diagram is commutative. It remains to prove that it is an equivalence. We will prove this by constructing a quasi-inverse functor. 

To explain the ideas, we will assume that $G$ is split so every $L^+G$-orbit is defined over $k$. Indeed, if $G$ is non-split, one can first pass to a finite extension $k'/k$ such that $G_{k'}$ is split and then perform a further Galois descent.

For a Schubert variety $\Gr_{\leq \mu}$, let  $\IC_\mu$ be its intersection cohomology complex (so $\IC_\mu|_{\Gr_\mu}=\Ql[(2\rho,\mu)]$). 
We first consider the full subcategory $\Sat_G^T$ of $\Sat_G$ spanned by direct sums of $\IC_\mu(i)$ for  $\mu\in \xcoch(T)^+, i\in \bZ$. 

\begin{lem}~
\begin{enumerate}
\item The category $\Sat_G^T$ is a Tannakian category with a fiber functor given by $\on{H}^*$.

\item The Tannakian group $\Aut^{\otimes}(\on{H}^*)$ is isomorphic to $\hat G^T$ defined in \eqref{Deligne modification}.
\end{enumerate}
\end{lem}
\begin{proof}
(1) Indeed Corollary \ref{Satake fiber} implies that  $\Sat_G^T$ is closed under the convolution product. In addition, as the commutativity constraints $c$ in Proposition \ref{existence of comm} is $\Ga_k$-equivariant (since $c_{\on{vect}}$ is $\Ga_k$-equivariant),  $\Sat_G^T$ is symmetric monoidal. Then as arguing in Theorem \ref{Tannakian structure}, it is Tannakian.

(2) Let us denote  the Tannakian group by $\tilde G^T$. There is an exact sequence
\[1\to \hat G\to \tilde G^T\stackrel{d}{\to} \bG_m\to 1,\]
where the first map is induced by the pullback $\Sat_G^T\to \Sat_{G_{\bar k}}$ and the second map is given by the full subcategory spanned by direct sums of $\IC_e(i), i\in\bZ$. There is a splitting $\bG_m\to \tilde G^T$ of $d:\tilde G^T\to \bG_m$
corresponding to the following grading on $\on{H}^*(\mA)$
$$\on{H}^*(\mA)=\bigoplus_i \Ql(i)\otimes \Hom(\Ql(i), \on{H}^*(\mA)).$$
Then as in Lemma \ref{geom Langlands dual},  using the fact that $c_1(\mO(1))\in \on{H}^2(\Gr_G\otimes\bar k,\bQ_\ell(1))^{\Ga_k}$, one concludes that $\tilde G^T=\hat G^T$.
\end{proof}

We will fix a quasi-inverse $S^T: \Rep_\Ql(\hat G^T)\simeq \Sat_G^T$ of $\on{H}^*$.

Now we consider the full category $\Sat_G$. 
We will make use of the abstract nonsense developed at the beginning of this subsection. Let $R^T$ be the ring of regular functions on $\hat G^T$, regarded as an object in $\on{Ind}\on{Rep}_{\Ql}(\hat G^T)$ via the left regular representation. Let $\mR^T=S^T(R^T)\in \on{Ind}\Sat^T_G$ be the corresponding object. Since $R^T$ also admits a right regular representation, $\mR^T$ is in fact an object in $\on{Rep}(\hat G^T,{\on{Ind}\Sat_G})$. Since $G$ is split, $\Rep^c_\Ql({^L}G^T)=\on{Rep}(\hat G^T,{\on{Rep}^c_{\Ql}(\Ga_k)})$. An object $V$ in $\on{Rep}(\hat G^T,{\on{Rep}^c_{\Ql}(\Ga_k)})$ can be regarded as a sheaf on $\Spec k$ with an action by $\hat G^T$. Then by general nonsense, $e_*V\star\mR^T$ makes sense as an  object $\Rep(\hat G^T, \on{Ind}\Sat_G)$, and therefore $(e_*V\star \mR^T)^{\hat G^T}$ is an object in $\on{Ind}\Sat_G$. We thus obtain a tensor functor
\[ \on{Rep}(\hat G^T,{\on{Rep}^c_{\Ql}(\Ga_k)})\to \on{Ind}\Sat_G,\quad V\mapsto (e_*V\star \mR^T)^{\hat G^T},\]
still denoted by $S^T$, because its restriction to  $\Rep_\Ql(\hat G^T)\subset \on{Rep}(\hat G^T,{\on{Rep}^c_{\Ql}(\Ga_k)})$ is nothing but the functor $S^T$ we fix at the beginning.

Recall that we regard $B=\mO_{\bG_m}$ as an object in $\on{Rep}(\hat G^T,{\on{Rep}^c_{\Ql}(\Ga_k)})$. Then 
$$\mB:=S^T(B)= \IC_e\otimes B.$$
Therefore $S^T$ sends $B$-module objects in $\on{Rep}(\hat G^T,{\on{Rep}^c_{\Ql}(\Ga_k)})$ to $\mB$-module objects in $\on{Ind}\Sat_G$. We thus can define 
\[S: \on{Rep}^c_{\Ql}({^L}G^{\on{geo}})\simeq B\on{-mod}(\on{Rep}^c_{\Ql}({^L}G^T)) \to \on{Ind}\Sat_G: V\mapsto S^T(V)\otimes_{B}\Ql,\]
where as before $B\to\Ql$ is the counit map.

Note that for $V\in \on{Rep}(\hat G^T,{\on{Rep}^c_{\Ql}(\Ga_k)})$,
$$\on{H}^*\circ S^T(V)=\on{H}^*((e_*V\star \mR^T)^{\hat G^T})=(V\otimes R^T)^{\hat G^T}=V,$$ 
regarded as a representation of ${^L}G^{\on{geo}}$. Then it follows that $S$ is the desired quasi-inverse of $\on{H}^*:\Sat_G\to \on{Rep}^c_{\Ql}({^L}G^{\on{geo}})$. We leave the details as an exercise.
\end{proof}
\begin{rmk}\label{simplify}
The proof of the theorem can be simplified if one of conditions in Remark \ref{lifting condition} holds. For example, if $\rho\in\xch(T)$  or if a half Tate twist exists, one can define the normalised intersection cohomology complex 
$$\IC_\mu^{\on{norm}}=\IC_\mu((2\rho,\mu)),$$ and can define a subcategory $\Sat_G^N\subset \Sat_G$ spanned by direct sums of these normalised IC sheaves. This is a monoidal subcategory of $\Sat_G$ (by Corollary \ref{Satake fiber}) and the pullback induces an equivalence 
$$\Sat_G^N\simeq \Sat_{G_{\bar k}}$$ 
of monoidal categories. Therefore, the ring of regular functions $R$ on $\hat G$ gives an object $\mR\in \on{Rep}(\hat G,{\on{Ind}\Sat_G})$ which can be used as a replacement of $\mR^T$ to define the functor $S$ directly.
\end{rmk}
As ${^L}G^{\on{alg}}$ can be regarded a pro-algebraic group, there is the category of algebraic representations $\on{Rep}_\Ql({^L}G^{\on{alg}})$. 
If a lifting $\tilde \chi:\Ga_k\to \hat G(\Ql)$ exists, then via the natural isomorphism in Lemma \ref{geom Langlands dual}, we can regard $\on{Rep}_\Ql({^L}G^{\on{alg}})$ as a full subcategory of $\on{Rep}^c_{\Ql}({^L}G^{\on{geo}})$. Under the assumption that $\rho\in \xch(T)$ or if a half Tate twist exists, it was explained in \cite{RZ} how to select out a full subcategory of $\Sat_G$ (denoted by $\on{P}_{L^+G}^f(\Gr_G)$ in \emph{loc. cit.}) corresponding to $\on{Rep}_\Ql({^L}G^{\on{alg}})$  under the equivalence \eqref{full Sat}. The statement is particularly neat when $k$ is finite. 
\begin{prop}
Assume that $k=\bF_q$ is finite. Let $\Sat_G^0\subset \Sat_G$ denote the full subcategory of semisimple pure perverse sheaves of weight zero. Then the equivalence \eqref{full Sat equiv} restricts to an equivalence
\[F:\Sat_G^0\simeq \on{Rep}_\Ql({^L}G^{\on{alg}})\]
\end{prop}
\begin{proof}See \cite[Remark A.9, Proposition A.10]{RZ}. Note that for $\mA\in\Sat_G^0$, the action of $\Ga_k$ on $F(\mA)$ coincides with the natural action of $\Ga_k$ on $\oplus_i \on{H}^i(\mA(\frac{i}{2}))$.
\end{proof}

Next, we will describe a version of the geometric Satake equivalence for all $\Sat_{X^I}$. For simplicity, we assume that $G$ is split and either $\rho\in \xch(T)$ or a half Tate twist exists. Then as explained above, there is a fully faithful functor
\[S: \Rep(\hat G)\simeq \Sat_G^N\to \Sat_G.\]
Let $X$ be a curve over $k$. For a finite non-empty set $I$, let $\Loc(X^I)$ denote the category of $\Ql$-local systems on $X^I$. We also have $\Rep(\hat G^I,\Loc(X^I))$.
There are the following functors relating these categories. Let $\phi:J\twoheadrightarrow I$.
\begin{enumerate}
\item[(i)] the restriction $\Delta(\phi)^*: \on{Rep}(\hat G^J,{\on{Loc}(X^J)})\to \on{Rep}(\hat G^I, {\on{Loc}(X^I)})$, where $\hat G^I$ acts on $\Delta(\phi)^*\mF$ via the embedding $\hat G^I\to \hat G^J$.
\item[(ii)] the external product $\boxtimes_{i\in I}:\prod_{i\in I} \on{Rep}(\hat G^{J_i},{\on{Loc}(X^{J_i})})\to \on{Rep}(\hat G^J,\on{Loc}(X^J))$.
\end{enumerate}

Here is the theorem. 
\begin{thm}\label{multiple point Sat}
There is a canonical equivalence 
\[S_I: \on{Rep}(\hat G^I, \on{Loc}(X^I))\simeq \Sat_{X^I}\]
compatible with the restriction functors and the external product functors in obvious sense.
\end{thm}
\begin{rmk}
A version of the this theorem was first proved by Gaitsgory (\cite{Ga2}). In fact, he considered larger categories in both sides, namely $\on{Rep}(\hat G^I, \on{P}(X^I))$ on the l.h.s and $\on{P}_{(L^+G)_{X^I}}(\Gr_{X^I})$ on the r.h.s. However, the restriction functors as in (1) are not defined on these bigger categories (see Remark \ref{triv group}). Our approach is different from \emph{loc. cit.}
\end{rmk}
The following corollary plays a crucial role in \cite{Laf}.
\begin{cor}For each $I$, there is a fully faithful functor $S_I: \on{Rep}(\hat G^I)\to \Sat_{X^I}$ such that for every surjective map $\phi: J\twoheadrightarrow I$, there is a canonical isomorphism of functors $\Delta(\phi)^*S_J\cong S_I\circ \Res_{\hat G^I}^{\hat G^J}$, where $\Res_{\hat G^I}^{\hat G^J}$ is the usual restriction functor for the embedding $\hat G^I\to \hat G^J$.
\end{cor}

We explain the ideas to prove this theorem.  
First, we claim that the functor
$\mH^* :\Sat_{X^I}\to \on{Loc}(X^I)$
canonically lifts to a functor $$F_I: \Sat_{X^I}\to \on{Rep}(\hat G^I,{\on{Loc}(X^I)}).$$ Namely,
let $\eta^I$ be the generic point of $X^I$, and $\overline{\eta^I}$ a geometric point over $\eta^I$. Since the restriction functor 
$$i_{\overline{\eta^I}}^*: \on{Loc}(X^I)\to \on{Rep}^c_{\Ql}(\Gal(\overline{\eta^I}/\eta^I))$$ 
is fully faithfully, it suffices to show that every $\mH^*(\mA)_{\overline{\eta^I}}$ admits a canonical action of $\hat G^I$. However, since $(\Gr_{X^I})_{\eta^I}\cong (\prod_i \Gr_X)_{\eta^I}$, the desired lifting follows from Theorem \ref{full Sat equiv}. 

To define a functor in another direction, we need an object in $\on{Rep}(\hat G^I, {\on{Ind}\Sat_{X^I}})$ analogous to the sheaf $\mR^T$ on $\Gr$ (or under our assumption the sheaf $\mR$ in Remark \ref{simplify}).  
We define
\[\mR_X:= S(R)_X,\]
which is an object in $\on{Rep}(\hat G, {\on{Ind}\Sat_X})$ ( by the reasoning as in Theorem \ref{full Sat equiv}). For general $I$, we define
\[\mR_{X^I}:=\boxstar_{i\in I} \mR_X\]
as the self external convolution product of $\mR_X$, naturally regarded as an object in $\on{Rep}(\hat G^I,{\on{Ind}\Sat_{X^I}})$.

Then we can define the functor 
$S_{I}: \on{Rep}(\hat G^I,{\on{Ind}\on{Loc}(X^I)})\to \on{Ind}\Sat_{X^I}$ as
$$S_{I}(\mF)= ((e_I)_*\mF[|I|]\ostar\mR_{X^I})^{\hat G^I}.$$
We claim that for every $\mA\in\Sat_{X^I}$, there is a canonical isomorphism
$S_IF_I(\mA)\cong \mA$,
and for every $\mF\in \on{Rep}(\hat G^I, {\on{Loc}(X^I)})$ there is a canonical isomorphism
$F_IS_I(\mF)\cong \mF$.
Note that by the ULA property, the restriction functor
\[\Sat_{X^I}\to \on{P}_{((L^+G)_{X^I})_{\eta^I}}((\Gr_{X^I})_{\eta^I})\]
is fully faithful. Therefore, it is enough to construct these isomorphisms over $\eta^I$, which reduces to Theorem \ref{full Sat equiv}.

From the factorization property of $\mR_{X^I}$, it is easy to check that $S_I$'s are compatible with external convolution product. By the interpretation of the convolution product as the fusion product, $F_I$'s are compatible with restrictions to diagonals. The theorem then follows.

\subsection{From the geometric Satake to the classical Satake.}\label{applications of sat}

In this subsection, we discuss the relationship between the geometric Satake equivalence and the classical Satake isomorphism. For simplicity, we will assume that $G$ is split over $k=\bF_q$. The general case was discussed in \cite{RZ}. Let $\sigma$ denote the geometric Frobenius automorphism of $\bF_q$.
We fix a half Tate twist $\Ql(\frac{1}{2})$, or equivalently a square root $q^{1/2}$ ($=\tr(\sigma,\Ql(\frac{1}{2}))^{-1}$)  of $q$ in $\Ql$.  To relate the geometric Satake with the classical Satake, we also need to choose an isomorphism $\iota:\Ql\cong \bC$. We can arrange our choice of $\Ql(\frac{1}{2})$ and $\iota$ such that 
$$\iota(q^{1/2})\in \bR^{<0}.$$ The advantage of this choice is that under $\iota$, the trace of the geometric Frobenius on $\Ql[i](\frac{i}{2})$  is \emph{positive} for every $i\in \bZ$.

We will let $H_G$ denote the classical spherical Hecke algebra, consisting of $G(\mO)$-bi-invariant $\bZ$-valued functions on $G(F)$, equipped with the convolution product
\[(f*g)(x)=\int_{G(F)}f(y)g(y^{-1}x)dy,\]
where the measure is chosen such that the volume of $G(\mO)$ is $1$. For $\mu\in \xcoch(T)^+$, let $c_\mu$ denote the characteristic function on $G(\mO)t^\mu G(\mO)$. Then $\{c_\mu,\ \mu\in\xcoch(T)^+\}$ form a $\bZ$-basis of $H_G$.

On the other hand, let $R(\hat G)$ denote the representation ring of $\hat G$, which is a $\bZ$-algebra with a basis given by characters of irreducible representations. For $\mu\in\xcoch(T)^+$, let $L_\mu=\on{H}^*(\IC_\mu)$ denote the corresponding irreducible representation and let $\chi_\mu$ denote the corresponding character.

The classical Satake isomorphism (or rather, Langlands' reinterpretation) is a canonical isomorphism
\[\mS: H_G\otimes\bC\cong R(\hat G)\otimes\bC.\]
We refer to \cite{Gr} for details.
Now we explain the geometric approach.

Recall that the Grothendieck fonctions-faisceaux dictionary attaches to every $\mA\in \Sat_G$ (or more generally every $\ell$-adic complex on $\Gr_G$) a function 
\[f_\mA: \Gr(k)\to \Ql,\quad f_\mA(x)=\sum_i(-1)^i\tr(\sigma_x, \on{H}^i_{\bar x}(\mA)),\]
where $\bar x$ denotes a geometric point over $x$, $\on{H}^i_{\bar x}$ denotes the stalk cohomology at $\bar x$, and $\sigma_x\in \Gal(\bar x/x)$ denotes the geometric Frobenius element. Since $\mA$ is $L^+G$-equivariant, we can regard $f_\mA\in H_G\otimes_\bZ\Ql$. In addition, it follows from definition that
\begin{lem}\label{conv vs classical conv}
For $\mA_1,\mA_2\in\Sat_G$,
$f_{\mA_1\star\mA_2}=f_{\mA_1}*f_{\mA_2}$.
\end{lem}
\begin{proof}Let $m:\Gr\tilde\times\Gr\to \Gr$ be the convolution map. Given $x\in \Gr(k)$, the fiber $m^{-1}(x)(k)$ can be identified with $\{(y, y^{-1}x)\mid y\in G(F)/G(\mO)\}$, and by the construction there is a canonical isomorphism of stalks $(\mA_1\tilde\boxtimes\mA_2)_{(y,y^{-1}x)}\cong (\mA_1)_y\otimes(A_2)_{y^{-1}x}$. Then according to our choice of the measure, 
\[\begin{split}
f_{\mA_1}*f_{\mA_2}(x)&=\int_{G(F)}f_{\mA_1}(y)f_{\mA_2}(y^{-1}x)dy=\sum_{y\in G(F)/G(\mO)}f_{\mA_1}(y)f_{\mA_2}(y^{-1}x)\\
&=\sum_{y\in m^{-1}(x)(k)} \tr(\sigma_{\bar y}, (\mA_1\tilde\boxtimes\mA_2)_{\bar y})=f_{\mA_1\star\mA_2}(x).
\end{split}\]
\end{proof}

Now, let $f_\mu=f_{\IC_\mu[-(2\rho,\mu)]}$, where we recall that $\IC_\mu|_{\Gr_\mu}=\Ql[(2\rho,\mu)]$. By the purity of stalk cohomology of the intersection cohomology complex of (affine) Schubert varieties (cf. \cite{KL}), one can write
\begin{equation}\label{stalk IC}
f_\mu=c_\mu+\sum_{\la<\mu}a_{\mu\la}(q)c_\la,
\end{equation}
where $$a_{\mu\la}(x)=\sum a_{\mu\la,i}x^i\in \bN[x]$$ are Kazhdan-Lusztig polynomials.
Recall that $a_{\mu\la,i}=\dim \on{H}^{2i-(2\rho,\mu)}_{t^\la}\IC_\mu$ is the degree $2i-(2\rho,\mu)$ stalk cohomology of $\IC_\mu$ at $t^\la$. 

Recall that the $K$-group of a monoidal abelian category is a natural $\bZ$-algebra (so called the $K$-ring). Then the geometric Satake equivalence induces an isomorphism of $\bZ$-algebras
\[K(\Sat_G^N)\cong K(\on{Rep}_\Ql(\hat G)),\]
where we recall $\Sat_G^N$ is the monoidal subcategory of $\Sat_G$ spanned by normalised IC sheaves (see Remark \ref{simplify}). 
Let $\sigma_q$ denote the geometric Frobenius of $k=\bF_q$. Then 
by \eqref{stalk IC} and Lemma \ref{conv vs classical conv}, the Grothendieck fonctions-faisceaux dictionary induces an algebra isomorphism
\[f: K(\Sat_G^N)\otimes\Ql\to H_G\otimes\Ql.\]
On the other hand, there is always a canonical isomorphism $\on{Ch}:K(\on{Rep}_\Ql(\hat G))\cong R(\hat G)$ by sending $[V]$ to its character $\chi_V$. Putting together, we obtain the following isomorphism
\[
\mS':H_G\otimes\Ql\stackrel{f^{-1}}{\cong} K(\Sat_G^N)\otimes\Ql \cong K(\on{Rep}_\Ql(\hat G))\otimes\Ql\stackrel{\on{Ch}}{\cong}R(\hat G)\otimes\Ql
\]
\begin{lem}Under the chosen isomorphism $\iota:\Ql\cong\bC$, $\mS'=\mS$. In particular
$$\mS^{-1}(\chi_\mu)=(-\iota(q^{1/2}))^{-(2\rho,\mu)}f_\mu.$$
\end{lem}

The proof was sketched at the end of \cite{RZ}. First, the functor $\on{CT}:\Sat_G\to \Sat_T$ naturally lifts to a functor 
$$\on{CT}^N: \Sat_G^N\to \Sat_T^N.$$ 
(Indeed, it follows from Theorem \ref{MV theory} (iii) that $\on{H}_c^*(\mA)$ is of Tate type).
Under the Grothendieck fonctions-faisceaux dictionary and our choice of $\iota$ the functor $\on{CT}^N$ exactly corresponds to the classical Satake transform (e.g. see \cite[\S\ 3]{Gr} for explicit formulas). 

\begin{rmk}
(i) If one prefer not to introduce a square root $q^{1/2}$, one can formulate a version of the Satake isomorphism via the equivalence $\Sat_G^T\cong \Rep_\Ql(\hat G^T)$. 

(ii) Since $\iota f_\mA\in H_G\otimes\bZ[q^{\pm 1/2}]$ for every $\mA\in \Sat_G^N$, the isomorphism $\mS'$ restricts to an isomorphism $H_G\otimes \bZ[q^{\pm 1/2}]\simeq R(\hat G)\otimes \bZ[q^{\pm 1/2}]$. If in addition $\rho\in \xch(T)$, it further restricts to an isomorphism $H_G\otimes\bZ[q^{-1}]\simeq R(\hat G)\otimes\bZ[q^{-1}]$. This coincides with the classical theory (see \cite[\S 3]{Gr}).
\end{rmk}

Let us also mention that the Kazhdan-Lusztig polynomials $a_{\mu\la}$ admits the following interpretation via $\hat G$.
Recall the pinning on $\hat G$ given by Proposition \ref{canonical pinning}. It defines an increasing filtration (the
Brylinski-Kostant filtration) on any representation $V$ of $\hat G$ as
\begin{equation}\label{BK fil}
N_iV=(\ker N)^{i+1}.
\end{equation}
For $\la\in\xcoch(T)$, denote by $V(\la)$ the corresponding weight subspace of $V$, under the action of
$\tilde T$. Then filtration \eqref{BK fil} induces
\begin{equation}\label{BKfilmu}
N_iV(\la)=V(\la)\cap N_iV.
\end{equation}

The following theorem was proved in \cite{Lu} and \cite{Bry}, and by another method in \cite[\S~5]{Gi1} (see also \cite[\S~5]{Z11}). 
\begin{thm}For $\mu\in\xcoch(T)^+$, let $L_\mu=\on{H}^*(\IC_\mu)$ be the corresponding representation of $\hat G$.
\[a_{\mu\la}(x)=\sum_i \dim \on{gr}^N_iL_\mu(\la)x^i.\]
\end{thm}

\appendix

\section{Complements on sheaf theory}\label{App}
In this appendix, we review some sheaf theory that are used in \S~\ref{Lecture V}. We assume that $k$ is an algebraically closed field. Let $\ell$ be a prime different from $\cha\ k$. Sheaves will mean $\ell$-adic sheaves. All pushforward and pullback are \emph{derived}. (Ind)schemes are always (ind-)of finite type over $k$. The standard reference is \cite{BBD}.

\subsection{Equivariant category of perverse sheaves.}
\subsubsection{Basic properties.} Recall that for a scheme $X$ over $k$, there is the bounded derived category of constructible $\ell$-adic complexes $D_c^b(X)$. It contains the category of perverse sheaves $\on{P}(X)$ as a full abelian subcategory. If $X$ admits an action of a linear algebraic group $K$ (with the action map denoted by $\on{act}$), it makes sense to define the abelian category of $K$-equivariant perverse sheaves $\on{P}_K(X)$ on $X$. I.e., an object in $\on{P}_K(X)$ is a perverse sheaf $\mF$ on $X$ together with an isomorphism $\theta$ along the two maps 
$$\on{act},\ \pr_2:K\times X\substack{\longrightarrow\\ \longrightarrow} X,$$ satisfying: (i) $e^*\theta=\id$ where $e:X\to K\times X$ is given by the unit of $K$; (ii) a natural cocycle condition on $K\times K\times X$. 
\begin{rmk}
There is a more involved notion of (bounded) equivariant derived category $\on{D}_{K,c}^b(X)$
of (constructible) $\ell$-adic sheaves on $X$. It  cannot be defined in a simple way as above. Instead, one can either define this category as in \cite{BeLu}, or view it as the derived category of $\ell$-adic sheaves on the algebraic stack $[X/K]$ (e.g. cf. \cite{LZ}). In the note, we only occasionally use this concept in a very formal way, i.e. there exists a good Grothendieck's six operation formalism between these triangulated categories (over different schemes). 
\end{rmk}

By definition, there is a forgetful functor $\on{P}_K(X)\to \on{P}(X)$. 
\begin{lem}\label{A:forgetful}
If $K$ is a connected algebraic group, this forgetful functor is fully faithful, with essential image consisting of those perverse sheaves $\mF$ such that $\on{act}^*\mF\cong\pr_2^*\mF$. 
\end{lem}
The significance of this lemma is that being $K$-equivariant is a property rather than an additional structure of a perverse sheaf on $X$. Note that the statement is false for disconnected groups. Indeed, if $K$ is finite, $\on{P}_K(\on{pt})$ is equivalent to the category of finite dimensional representations of $K$. The statement is also false for the equivariant derived category.
\begin{proof}We sketch the proof of the first statement and leave the second statement as an exercise. The functor is clearly faithful. Let $(\mF_1,\theta_1)$ and $(\mF_2,\theta_2)$ be two $K$-equivariant perverse sheaves and $\varphi: \mF_1\to \mF_2$ be a morphism, we need to show that it automatically intertwines with $\theta_i$, i.e. we need to show that
\begin{equation}\label{A:comp}
\theta_2^{-1}\circ\on{act}^*(\varphi)\circ\theta_1=\pr_2^*(\varphi). 
\end{equation}
Note that since $K$ is connected, $\pr^*[\dim K]: \on{P}(X)\to \on{P}(K\times X)$ is fully faithful by \cite[Proposition 4.2.5]{BBD}. Then to check \eqref{A:comp}, it is enough to show $e^*(\theta_2^{-1}\circ\on{act}^*(\varphi)\circ\theta_1)=e^*(\pr_2^*(\varphi))$, which then is tautological.
\end{proof}

We will make use the following two properties of the equivariant category. 
\begin{lem}\label{A:unip}
Let $K_1\subset K$ be a closed normal subgroup. 
\begin{enumerate}
\item If the action of $K_1$ on $X$ is free and $[X/K_1]$ is represented by an algebraic space $\bar X$, then the pull back along $q:X\to \bar X$ induces an equivalence of categories 
\begin{equation}\label{equivariance free action}
q^*[\dim K_1]:\on{P}_{K/K_1}(\bar X)\simeq \on{P}_K(X)
\end{equation}

\item Assume that $K_1$ is connected. If the action of $K_1$ on $X$ is trivial, then the forgetful functor
\begin{equation}\label{unipotence}
\on{P}_K(X)\to \on{P}_{K/K_1}(X)
\end{equation}
is an equivalence of categories.
\end{enumerate}
\end{lem}
\begin{proof}By the \'etale descent of perverse sheaves, we can assume that $X=\bar X\times K_1$. Then (1) follows easily.
 (2) follows from Lemma \ref{A:forgetful}.
\end{proof}
\begin{rmk}\label{equivariant descent}
The equivalence \eqref{equivariance free action} continues to hold if one replaces $\on{P}_{K/K_1}$ by the equivariant derived categories. But it is less trivial (see \cite[\S\ 2.6.2]{BeLu}).
\end{rmk}

\subsubsection{Twisted external product.}\label{A:twist product}
Now let $p:E\to X$ be a $K$-torsor and $Y$ be a scheme with an action of $K$. Let $\mF\in\on{P}(X)$ and $\mG\in \on{P}_K(Y)$, then one can form the twisted external product $\mF\tilde\boxtimes\mG$ as a perverse sheaf on $X\tilde\times Y$ as follows. We have the perverse sheaf $p^*\mF[\dim K]\boxtimes \mG$ on $E\times Y$ which is $K$-equivariant by construction. Then by Lemma \ref{A:unip} (i), it descends to a unique (up to a unique isomorphism) perverse sheaf on $X\tilde\times Y$, which is our $\mF\tilde\boxtimes\mG$. Note that by Remark \ref{equivariant descent}, one can actually define the external twisted product for $\ell$-adic complexes.

\subsubsection{Equivariant cohomology.}\label{A:char class}
Let $K$ be a linear algebraic group over $k$ and let $R_{K}=\on{H}^*(\bB K)$ be the cohomology of the classifying stack of $K$. It is a finitely generated augmented commutative $\Ql$-algebra isomorphic to a polynomial algebra. Concretely, it can be realised as follows. Let  $\{E_n\to B_n\}$ denote a sequence of $K$-torsors over $\{B_n\}$, which approximates of the classifying space of $K$. E.g. we can embed $K$ into some $\GL_r$ such that $\GL_r/K$ is quasi-affine. For $n$ large, let $E_n:=S_{n,r}$ be the Stiefel variety, i.e. the tautological $\GL_r$-torsor over $\Gr(r,n)$. Then $B_n:=E_n/K$ is represented by a scheme, $B_n\subset B_{n+1}$ is a closed embedding, and $\on{H}^*(\underrightarrow\lim_n B_n)= \on{H}^*(\bB K)$. 
For example, if $K=T$ is a torus, then there is a canonical isomorphism $\Spec R_T\cong \Lie T$. If $K$ is unipotent, then $R_K=\Ql$.

Let $\mF\in \on{P}_K(X)$ (or more generally $\mF\in\on{D}^b_{K,c}(X)$). It makes sense to talk about the $K$-equivariant cohomology $\on{H}_K^*(X,\mF)$. It can be either defined as the cohomology of $\mF$ as a sheaf on the quotient stack $[X/K]$, or more concretely
$$\on{H}_K^*(X,\mF):=\on{H}^*(\underrightarrow\lim_n (B_n\tilde\times X), \Ql\tilde\boxtimes \mF).$$
From the construction, $\on{H}^*_K(X,\mF)$ is a module over $\on{H}^*(\underrightarrow\lim_n B_n)=R_{K}$. 

It follows from the definition that if $K_1\subset K$ acts freely on $X$ with $\bar X$ the quotient as in Lemma \ref{A:unip}, then
\begin{equation}\label{coh free action}
\on{H}^*_K(X,q^*\mF)=\on{H}^*_{K/K_1}(\bar X,\mF).
\end{equation}
On the other hand, if $K_1$ is unipotent and acts trivially on $X$, then
\begin{equation}\label{coh unip}
\on{H}^*_K(X,\mF)=\on{H}^*_{K/K_1}(X,\mF).
\end{equation}

Now assume that $K$ is a connected linear algebraic group. Let $\mF\in\on{P}_K(X)$. The Leray spectral sequence associated to the projection $B_n\tilde\times X\to B_n$ induces a spectral sequence with the $E_2$-term $\on{H}^*(\bB K)\otimes \on{H}^*(X,\mF)$ and the abutment $\on{H}^*_K(X,\mF)$. Following \cite{GKM}, we say that $\mF$ is equivariantly formal if the spectral sequence degenerates at $E_2$-term. 
\begin{thm}
If $X$ is proper, then the intersection cohomology sheaf $X$, viewed as a $K$-equivariant perverse sheaf on $X$, is equivariantly formal. 
\end{thm}
The proof relies on the theory of ``weights''. See \cite[Theorem 14.1]{GKM}. In fact, the same argument applies to semi-simple pure perverse sheaves of geometric origin.

For an equivariantly formal complex $\mF$, there is a non-canonical isomorphism $$\on{H}^*_K(X,\mF)\cong R_{K}\otimes\on{H}^*(X,\mF).$$ In particular, in this case $\on{H}^*_K(X,\mF)$ is a free $R_{K}$-module of rank equal to the dimension of $\on{H}^*(X,\mF)$. Canonically, there is an isomorphism
\begin{equation}\label{de-equiv}
\on{H}^*(X,\mF)\cong \on{H}^*_K(X,\mF)\otimes_{R_{K}}\Ql,
\end{equation}
where $R_{K}\to\Ql$ is the augmentation map. In addition, one has the equivariant K\"unneth formula. Assume that both $X$ and $Y$ are proper and $\mF\in\on{P}_K(X)$ and $\mG\in\on{P}_K(Y)$ are semisimple pure perverse sheaves. Then there is the canonical isomorphism
\begin{equation}\label{equiv Kunneth}
\on{H}_K^*(X,\mF)\otimes_{R_{K}}\on{H}_K^*(Y,\mG)\cong \on{H}^*_K(X\times Y,\mF\boxtimes\mG).
\end{equation}
This follows from the usual K\"unneth formula and the degeneration of the spectral sequence.

Let us also state the localization theorem for equivariant cohomology (see \cite[(6.2)]{GKM} for the topological version). First, for $\mF\in \on{D}_{T,c}^b(X)$, one can define the $T$-equivariant cohomology with compact support as 
$$\on{H}_{T,c}^*(X,\mF):=\on{H}_c^*(\underrightarrow\lim_n (B_n\tilde\times X), \Ql\tilde\boxtimes \mF).$$

\begin{thm}\label{equiv localization}
Let $K=T$ be a torus acting on $X$ and let $i:X^0\subset X$ be the subscheme of $T$-fixed point. Then for $\mF\in\on{D}_{T,c}^b(X)$, the natural maps
\[\on{H}^*_T(X^0,i^!\mF)\to \on{H}^*_T(X,\mF),\quad \on{H}^*_{T,c}(X,\mF)\to \on{H}^*_{T,c}(X^0,i^*\mF)\]
become isomorphisms after localising to the generic point of $R_{T}$.
\end{thm}
\begin{proof}We sketch the proof of the first isomorphism. The second then follows by the Verdier duality. Let $j: X-X^0\to X$ be the open complement. From the distinguished triangle $i_!i^!\mF\to \mF\to j_*j^*\mF\to$, it reduces to show that if $X^0=\emptyset$, then $\on{H}^*_T(X,\mF)$ is a torsion $R_T$-module. We can find a finite cover of $X=\cup X_i$ such that on each $X_i$ is $T$-invariant and there is a non-trivial subtorus $T_i\subset T$ acting on $X_i$ freely. Then it follows from \eqref{coh free action} that the cohomology of $\on{H}^*_T(X_i,\mF|_{X_i})$ is a torsion $R_T$-module. Therefore $\on{H}^*_T(X,\mF)$ is also a torsion $R_T$-module by a spectral sequence argument.
\end{proof}

\subsubsection{Pro-algebraic groups acting on ind-schemes.}\label{A:ind}
Let $X$ be an ind-scheme. We can define the category of perverse sheaves on $X$ as $\on{P}(X)=\underrightarrow\lim\on{P}(X_i)$, where $X=\underrightarrow\lim X_i$ is a presentation of $X$ as closed subschemes of finite type over $k$. This is independent of the choice of the presentation.

Now assume that $X$ admits an action by a pro-algebraic group $K$. We assume that 
\[(S)\quad\quad \left.\begin{array}{c}\mbox{the stabiliser of each geometric point only has} \\ \mbox{ finitely many connected components.}\end{array}\right.\]
Then it makes sense to define the category of $K$-equivariant perverse sheaves on $X$ as follows.

First, there exists a presentation $X=\underrightarrow\lim X_i$, such that each $X_i$ is a $K$-stable closed subscheme of finite type over $k$, and the action of $K$ on $X_i$ factors through an algebraic quotient $K_i$. Then we can choose $K_i$ to be an algebraic quotient of $K$ by a connected normal subgroup. In addition, if $X_i\subset X_j$ is a closed subscheme, we can arrange $K_i$ to be a quotient of $K_j$.

Then we define $\on{P}_K(X_i):=\on{P}_{K_i}(X_i)$. By  Lemma \ref{A:unip} (ii), we see that $\on{P}_K(X)$ is independent of the choice of $K_i$ up to a canonical equivalence.  Therefore, we have a fully faithful functor
\[\on{P}_K(X_i)=\on{P}_{K_i}(X_i)\simeq \on{P}_{K_j}(X_i)\to \on{P}_{K_j}(X_j)=\on{P}_K(X_j).\]  
Finally, we define 
\[\on{P}_K(X)=\underrightarrow\lim \on{P}_K(X_i).\]

Now assume that $E\to X$ is a $K$-torsor over an ind-scheme ind-of finite type, and $Y$ is a $K$-ind-scheme satisfying (S). Let $\mF\in\on{P}(X)$ and $\mG\in\on{P}_K(Y)$, we can similarly define the twisted product $\mF\tilde\boxtimes\mG$ as an object in $\on{P}(X\tilde\times Y)$. Namely, $\mF$ is supported on some $X_i$ and $\mG$ is supported on some $Y_j$. The action of $K$ on $Y_j$ factors through $K_j$. Then $E_{ij}:=E|_{X_i}\times^KK_i\to X_i$ is a $K_i$-torsor and 
$$X_i\tilde\times Y_j= E|_{X_i}\times^KY_j= E_{ij}\times^{K_j}Y_j$$ 
is a closed subscheme of $X\tilde\times Y$. Then the construction in \S\ \ref{A:twist product} gives $\mF\tilde\boxtimes\mG$ as an object in $\on{P}(X_i\tilde\times Y_j)\subset \on{P}(X\tilde\times Y)$.

Finally, let $K^u$ be the pro-unipotent radical of $K$ and assume that $K/K^u$ is algebraic. Let $X$ be a $K$-ind-scheme satisfying (S). Then for $\mF\in \on{P}_K(X)$, we have the well-defined $K$-equivariant cohomology $H_K^*(X,\mF)$, which is a module over $R_{K}=\on{H}^*(\bB (K/K^u))$. All the properties of the equivariant cohomology discussed in \S\ \ref{A:char class} extend to this setting.

\subsection{Universally local acyclicity.}\label{ULA sheaf}

Let $S$ be a scheme and $s$ be a geometric point of $S$, we denote by $S_{(s)}$ the strict Henselisation of $S$ at $s$.  We will formally write $t\to s$ if $t$ is a geometric point of $S_{(s)}$, and call it the specialisation map.

We recall the following definition, as in \cite[]{SGA4.5}. 
\begin{dfn}Let $f:X\to S$ be a morphism of schemes of finite type over $k$. An $\ell$-adic complex $\mF$ on $X$ is called locally acyclic with respect to $f$ if for every geometric point $x\in X$, and every geometric point $t\in S_{(f(x))}$, the natural map $R\Gamma(X_{(x)},\mF)\to R\Gamma( (X_{(x)})_t,\mF)$ is an isomorphism, where $(X_{(x)})_t=X_{(x)}\times_{S_{(f(x))}}t$. It is called universally locally acyclic (ULA) if it is locally acyclic after arbitrary base change $S'\to S$.
\end{dfn}

We can reformulate local acyclicity as follows.  Let $f:X\to S$ be a morphism as above and $\mF$ be an $\ell$-adic complex on $X$. Let $s$ be a geometric point of $S$ and $t$ be a geometric point of $S_{(s)}$. Let $j_t: X\times_{S} t\to X$ be the natural map. We write
\[\Psi_{t\to s}(\mF):= ((j_t)_*j_t^*\mF)|_{X_s} \]
and think it as the nearby cycles along the specialisation map $t\to s$ (recall that by our convention $(j_t)_*$ is derived). Indeed, if $(S,s,\eta)$ is a strictly local Henselian trait (i.e. $S$ is the spectrum of a local strictly Henselian DVR with $s$ and $\eta$ its closed and generic points)  and $t$ is a geometric point over $\eta$, $\Psi_{t\to s}(\mF)$ is by definition the 1-dimensional nearby cycles $\Psi(\mF_\eta)$ of $\mF$.

\begin{lem}\label{LA reform}
The complex $\mF$ is locally acyclic with respect to $f$ if and only if for every $s$ and $t$ as above, the natural map 
\begin{equation}\label{sp}
\mF|_{X_s}\to \Psi_{t\to s}(\mF)
\end{equation}
is an isomorphism. 
\end{lem}
\begin{proof}Indeed, local acyclicity says the above map induces the isomorphism of stalks at every geometric point $x$ of $X_s$.
\end{proof}
\begin{rmk}\label{LA explanation}
Note that if $(S,s,\eta)$ is a trait, Lemma \ref{LA reform} says that $\mF$ is locally acyclic with respect to $f$ if and only if $\mF_s\to \Psi(\mF_\eta)$ is an isomorphism, or equivalently the vanishing cycles $\Phi(\mF)$ vanish. 
Therefore, for a higher dimensional base $S$, roughly speaking a sheaf $\mF$ is ULA with respect to $f: X\to S$ if there are no vanishing cycles along any direction on $S$. This can be made precise by introducing  nearby cycles and vanishing cycles functors over higher dimensional base (cf. \cite[Example 1.7(b)]{Il3}). 
\end{rmk}

Here are some basic properties of ULA sheaves that we need.

\begin{thm}\label{ULA prop}~
\begin{enumerate}
\item Let $S$ be a smooth variety over $k$. Then a complex $\mF$ on $S$ is ULA with respect to the identity map if and only if each cohomology sheaf $\mH^i(\mF)$ is a local system on $S$.
 
\item If $f:X\to Y$ is a proper morphism over $S$, and assume that $\mF$ is ULA with respect to $X\to S$. Then $f_*\mF$ is ULA with respect to $S$.

\item If $f:X\to Y$ is a smooth morphism over $S$. Then $f^*\mF$ is ULA with respect to $X\to S$ if and only if $\mF$ is ULA with respect to $Y\to S$.

\item Let $f_i:X_i\to S, i=1,2$ be two morphisms and $\mF_i$ a complex on $X_i$ ULA with respect to $f_i$. Then $\mF_1\boxtimes_S\mF_2$ is ULA with respect to $f_1\times_Sf_2:X_1\times_SX_2\to S$.

\item[(4)'] Let $f_i:X_i\to S_i, i=1,2$ be two morphisms and let $\mF_i$ be a complex on $X_i$ ULA with respect to $f_i$. Then $\mF_1\boxtimes\mF_2$ is ULA with respect to $f_1\times f_2$.
\end{enumerate}
\end{thm}
\begin{proof} 
(1) It follows by definition that an honest sheaf on $S$ is ULA with respect to $\id:S\to S$ if and only if it is a local system. The general case follows by induction on the cohomological amplitude of $\mF$.

(2) and (3) follow from the fact that the functor
$\Psi_{t\to s}$ commutes with the proper push-forward and the smooth pullback.
(4) is \cite[Corollary 2.5]{Il3}. (4)' follows from (4) by noting that for every morphism $S\to S_1\times S_2$, $(X_1\times X_2)\times_{S_1\times S_2}S=(X_1\times_{S_1}S)\times_S(X_2\times_{S_2}S)$.
\end{proof}

Now we discuss a result related to ULA perverse sheaves. 
\begin{thm}\label{ULA perverse}
Let $f: X\to S$ be a morphism of $k$-varieties with $S$ smooth and $D\subset S$ a smooth effective divisor. Let $i: Z=f^{-1}(D)\to X$ be the closed embedding and $j:U=X-Z\to X$ be the open complement. If a complex $\mF$ on $X$ is ULA with respect to $f$ and $\mF|_U$ is perverse, then 
\[\mF\cong j_{!*}(\mF|_U), \quad i^!\mF[1]\cong i^*\mF[-1].\]
In particular, $\mF$ is also perverse. In addition, $i^*\mF[-1]$ is perverse and ULA with respect to $fi$. 
\end{thm}
\begin{proof} The question is local in \'etale topology, so we can assume that there is a \emph{smooth} map $g: S\to \bA^1$ such that $D=g^{-1}(0)$. By \cite[Lemma 2.14]{SGA4.5}, $\mF$ is ULA with respect to $gf: X\to \bA^1$.
Then we reduce to the situation $f:X\to \bA^1$ and $Z=f^{-1}(0)$. According to Remark \ref{LA explanation}, $i^*\mF\to \Psi(j^*\mF)$ is an isomorphism. It follows that
\[i^*\mF[-1]\cong \Psi(j^*\mF)[-1]=:\mC\]
is perverse (\cite[Corollary 4.5]{Il2}). Using the distinguished triangle
\[ j_!j^*\mF\to \mF\to i_*i^*\mF\to,\]
we see that $\mF$ lives in perverse cohomological degree $-1$ and $0$.
In addition,  it follows from the distinguished triangle (see proof of \cite[Lemma 3.11]{SGA4.5})
\[i^*j_*j^*\mF\to \Psi(j^*\mF)^P\stackrel{1-t}{\to} \Psi(j^*\mF)^P\to\]
that the perverse cohomology sheaf ${^p}\mH^r i^*j_*j^*\mF\cong \mC$ for $r=-1,0$ and vanishes for $r\neq -1,0$, where $P$ denotes the wild inertia and $t$ denotes a generator of the tame inertia.
It then follows from the distinguished triangle 
$$i_*i^!\mF\to \mF\to j_*j^*\mF\to $$ that $i^!\mF[1]\cong \mC$, that $\mF$ is perverse and that there is a short exact sequence of perverse sheaves
$$0\to \mF\to j_*j^*\mF\to {^p}\mH^0j_*j^*\mF\to 0.$$ 
This in turn implies that $\mF\cong j_{!*}j^*\mF$.
\end{proof}

\begin{rmk}
When the base $S$ is smooth, there is another equivalent definition of local acyclicity as introduced in \cite[\S~5.1,~Appendix B]{BGa}. Under this other definition, Theorem \ref{ULA prop} (iv) was proved by Gaitsgory (private communication) and the rest part of Theorem \ref{ULA prop} and Theorem \ref{ULA perverse} were proved by Reich in \cite{Rei}.
\end{rmk}

%
%
%
%
%
%
%


\end{document}